\documentclass[11pt,american,english]{article}
\usepackage[latin9]{inputenc}
\usepackage{geometry}
\usepackage{color}
\usepackage{babel}
\usepackage{array}
\usepackage{units}
\usepackage{multirow}
\usepackage{amsthm}
\usepackage{amsmath}
\usepackage{amssymb}
\usepackage{graphicx}
\usepackage{esint}
\usepackage[all]{xy}
\PassOptionsToPackage{normalem}{ulem}
\usepackage{ulem}
\usepackage[unicode=true,pdfusetitle,
 bookmarks=true,bookmarksnumbered=false,bookmarksopen=false,
 breaklinks=false,pdfborder={0 0 1},backref=false,colorlinks=false]
 {hyperref}

\makeatletter

\providecommand{\tabularnewline}{\\}

\numberwithin{equation}{section}
\numberwithin{figure}{section}
\theoremstyle{plain}
\newtheorem{thm}{\protect\theoremname}
  \theoremstyle{remark}
  \newtheorem{rem}[thm]{\protect\remarkname}
  \theoremstyle{definition}
  \newtheorem{defn}[thm]{\protect\definitionname}
  \theoremstyle{definition}
  \newtheorem{example}[thm]{\protect\examplename}
  \theoremstyle{plain}
  \newtheorem{cor}[thm]{\protect\corollaryname}
  \theoremstyle{plain}
  \newtheorem{lem}[thm]{\protect\lemmaname}
  \theoremstyle{plain}
  \newtheorem{prop}[thm]{\protect\propositionname}
  \theoremstyle{plain}
  \newtheorem{fact}[thm]{\protect\factname}
  \theoremstyle{plain}
  \newtheorem{conjecture}[thm]{\protect\conjecturename}
  \theoremstyle{remark}
  \newtheorem{claim}[thm]{\protect\claimname}

\date{\today}

\usepackage{amsthm}\usepackage{amsbsy}\usepackage{amsfonts}\usepackage{tikz}\usetikzlibrary{calc,fadings,decorations.pathreplacing}
\usepackage{tikz}
\usepackage{epsfig}
\usepackage{appendix}

\usetikzlibrary{arrows}
\usetikzlibrary{shapes}

\theoremstyle{plain}     

\long\def\symbolfootnote[#1]#2{\begingroup%
\def\thefootnote{\fnsymbol{footnote}}\footnote[#1]{#2}\endgroup}

\numberwithin{thm}{section}

\theoremstyle{plain}
\newtheorem{claim}[thm]{\protect\claimname}

\theoremstyle{plain}
\newtheorem{remark}[thm]{\protect\remarkname}

\newcommand{\genus}{\mathfb{genus}}

\usepackage[leftcaption]{sidecap}

\newcommand{\FigBesBeg}[1][1.0]{%
 \let\MyFigure\figure
 \let\MyEndfigure\endfigure
 \renewenvironment{figure}[1]{\begin{SCfigure}[#1]##1}{\end{SCfigure}}}

\newcommand{\FigBesEnd}{%
 \let\figure\MyFigure
 \let\endfigure\MyEndfigure}

\makeatother

  \addto\captionsamerican{\renewcommand{\claimname}{Claim}}
  \addto\captionsamerican{\renewcommand{\conjecturename}{Conjecture}}
  \addto\captionsamerican{\renewcommand{\corollaryname}{Corollary}}
  \addto\captionsamerican{\renewcommand{\definitionname}{Definition}}
  \addto\captionsamerican{\renewcommand{\examplename}{Example}}
  \addto\captionsamerican{\renewcommand{\factname}{Fact}}
  \addto\captionsamerican{\renewcommand{\lemmaname}{Lemma}}
  \addto\captionsamerican{\renewcommand{\propositionname}{Proposition}}
  \addto\captionsamerican{\renewcommand{\remarkname}{Remark}}
  \addto\captionsamerican{\renewcommand{\theoremname}{Theorem}}
  \addto\captionsenglish{\renewcommand{\claimname}{Claim}}
  \addto\captionsenglish{\renewcommand{\conjecturename}{Conjecture}}
  \addto\captionsenglish{\renewcommand{\corollaryname}{Corollary}}
  \addto\captionsenglish{\renewcommand{\definitionname}{Definition}}
  \addto\captionsenglish{\renewcommand{\examplename}{Example}}
  \addto\captionsenglish{\renewcommand{\factname}{Fact}}
  \addto\captionsenglish{\renewcommand{\lemmaname}{Lemma}}
  \addto\captionsenglish{\renewcommand{\propositionname}{Proposition}}
  \addto\captionsenglish{\renewcommand{\remarkname}{Remark}}
  \addto\captionsenglish{\renewcommand{\theoremname}{Theorem}}
  \providecommand{\claimname}{Claim}
  \providecommand{\conjecturename}{Conjecture}
  \providecommand{\corollaryname}{Corollary}
  \providecommand{\definitionname}{Definition}
  \providecommand{\examplename}{Example}
  \providecommand{\factname}{Fact}
  \providecommand{\lemmaname}{Lemma}
  \providecommand{\propositionname}{Proposition}
  \providecommand{\remarkname}{Remark}
\providecommand{\theoremname}{Theorem}

\begin{document}
\global\long\def\genus{\mathrm{genus}}
 \global\long\def\Hom{\mathrm{Hom}}
 \global\long\def\surface{\Sigma_{g,1} }
 \global\long\def\wedger{{\textstyle \bigvee^{r}S^{1}} }
 \global\long\def\Stab{\mathrm{Stab}}
 \global\long\def\aa{\vec{\alpha}}
 \global\long\def\bb{\vec{\beta}}
 \global\long\def\cc{\vec{\gamma}}
 \global\long\def\MCG{\mathrm{MCG}}
\global\long\def\ap{{\cal AP} }
\global\long\def\pmp{{\cal PMP} }
\global\long\def\trw{{\cal T}r_{w} }
\global\long\def\trwl{{\cal T}r_{w_{1},\ldots,w_{\ell}} }
 \global\long\def\tr{{\cal T}r }
 \global\long\def\cl{{\cal \mathrm{cl}} }
\global\long\def\wg{{\cal \mathrm{Wg}} }
\global\long\def\moeb{\mathrm{M\ddot{o}b} }
\global\long\def\F{\mathrm{\mathbf{F}} }
\global\long\def\P{{\cal P} }

\global\long\def\id{\mathrm{id}}
\global\long\def\e{\varepsilon}
\global\long\def\U{\mathcal{U}}
\global\long\def\Aut{\mathrm{Aut}}
\global\long\def\E{\mathbb{E}}
 \global\long\def\Q{\mathbb{\mathbb{\mathbf{Q}}}}
\global\long\def\d{\delta}
\global\long\def\G{\Gamma}
\global\long\def\g{\gamma}
\global\long\def\ch{\mathrm{chi}}
\global\long\def\sol{\mathrm{Solu}}
\global\long\def\match{\mathrm{Match}}
\global\long\def\wl{w_{1},\ldots,w_{\ell}}
\global\long\def\st{\sigma,\tau}

\title{Word Measures on Unitary Groups}

\author{Michael Magee%
\thanks{Author Magee was partially supported by the National Science Foundation
under agreement No. DMS-1128155.%
} ~~and~~ Doron Puder%
\thanks{Author Puder was supported by the Rothschild fellowship and by the
National Science Foundation under agreement No. DMS-1128155.%
}}
\maketitle
\begin{abstract}
We combine concepts from random matrix theory and free probability
together with ideas from the theory of commutator length in groups
and maps from surfaces, and establish new connections between the
two. 

More particularly, we study measures induced by free words on the
unitary groups $\U\left(n\right)$. Every word $w$ in the free group
$\F_{r}$ on $r$ generators determines a word map from $\U\left(n\right)^{r}$
to $\U\left(n\right)$, defined by substitutions. The $w$-measure
on $\U\left(n\right)$ is defined as the pushforward via this word
map of the Haar measure on $\U\left(n\right)^{r}$.

Let $\trw\left(n\right)$ denote the expected trace of a random unitary
matrix sampled from $\U\left(n\right)$ according to the $w$-measure.
It was shown by Voiculescu \cite{VOIC} that for $w\ne1$ this expected
trace is $o\left(n\right)$ asymptotically in $n$. We relate the
numbers $\trw\left(n\right)$ to the theory of commutator length of
words and obtain a much stronger statement: $\trw\left(n\right)=O\left(n^{1-2g}\right)$,
where $g$ is the commutator length of $w$. Moreover, we analyze
the number\linebreak{}
$\lim_{n\to\infty}n^{2g-1}\cdot\trw\left(n\right)$ and show it is
an integer which, roughly, counts the number of (equivalence classes
of) solutions to the equation $\left[u_{1},v_{1}\right]\ldots\left[u_{g},v_{g}\right]=w$
with $u_{i},v_{i}\in\F_{r}$. 

Similar results are obtained for finite sets of words and their
commutator length, and we deduce that one can {}``hear'' the stable
commutator length of a word by {}``listening'' to its unitary measures.
\end{abstract}
\tableofcontents{}

\section{Introduction\label{sec:Introduction}}

\subsection{The expected trace \label{sub:The-expected-trace}}

Let $x_{1},\ldots,x_{r}$\marginpar{$x_{1},\ldots,x_{r}$} denote
generators of the free group $\F_{r}$ on $r$ generators. Consider
a word $w\in\F_{r}$, given by
\begin{equation}
w=\prod_{1\leq j\leq\left|w\right|}x_{i_{j}}^{\varepsilon_{j}},\label{eq:word-expression}
\end{equation}
where each $\e_{j}\in\{\pm1\}$ and%
\footnote{We use the standard notation $\left[r\right]$ for $\left\{ 1,\ldots,r\right\} $.%
} $i_{j}\in[r]$. Let $\left(\U\left(n\right),\mu_{n}\right)$\marginpar{$\left(\U\left(n\right),\mu_{n}\right)$}
be the probability space of $n\times n$ unitary matrices, equipped
with unit-normalized Haar measure. We consider a tuple $\{U_{i}^{\left(n\right)}\}{}_{i\in\left[r\right]}$
of $r$ independent random matrices sampled from $\left(\U\left(n\right),\mu_{n}\right)$.
For each $n$ we can form the \textbf{\textit{\emph{word map}}}\textit{}%
\footnote{Unless we stick to reduced forms, every word $w\in\F_{r}$ has different
expressions as products of the generators $x_{1},\ldots,x_{r}$ and
their inverses. However, the word map $w:\U\left(n\right)^{r}\to\U\left(n\right)$
is well-defined independently of the particular expression. Namely,
omitting from the expression for $w$ or adding to it subwords of
the form $x_{i}x_{i}^{-1}$ or $x_{i}^{-1}x_{i}$ does not effect
the resulting word map.%
} 
\begin{equation}
w:\U\left(n\right){}^{r}\to\U\left(n\right),\quad w\left(u_{1},\ldots,u_{r}\right)\equiv\prod_{1\leq j\leq\left|w\right|}u_{i_{j}}^{\e_{j}}\label{eq:word}
\end{equation}
where we abuse notation to identify $w$ with the corresponding map
and suppress the dependence on $n$. We call the pushforward by $w$
of the Haar measure $\mu_{n}^{\, r}$ on $\U\left(n\right)^{r}$ the
\textbf{$w$-measure} on $\U\left(n\right)$. In this paper we study
word measures on $\U\left(n\right)$ and relate them to algebraic
properties of the word $w$.

Word measures on unitary groups were studied mostly in the context
of free probability. Let $\mathrm{tr}$\marginpar{$\mathrm{tr}$}
denote the standard trace on complex $n\times n$ matrices, and denote
by $\trw\left(n\right)$\marginpar{$\trw\left(n\right)$} the expected
value of the trace of a random unitary matrix in ${\cal U}\left(n\right)$
under the $w$-measure. It is a fundamental result of Voiculescu \cite[Theorem 3.8]{VOIC}
that for $w\in\F_{r}$, 
\begin{equation}
\trw\left(n\right)\overset{\mathrm{def}}{=}\E\left[\mathrm{tr}\left(w\left(U_{1}^{\left(n\right)},\ldots,U_{r}^{\left(n\right)}\right)\right)\right]=\begin{cases}
n & \text{if }w=1\\
o\left(n\right) & \text{else}
\end{cases}\label{eq:firstorder}
\end{equation}
(the small $o$ notation is in the regime $n\to\infty$). It follows
that the random variables\linebreak{}
$U_{1}^{\left(n\right)},(U_{1}^{\left(n\right)})^{*},\ldots,U_{r}^{\left(n\right)},(U_{r}^{\left(n\right)})^{*}$
are \textit{asymptotically free}%
\footnote{This is sometimes called asymptotically $*$-freeness of $U_{1}^{\left(n\right)},\ldots,U_{r}^{\left(n\right)}$.
The statement of \cite[Theorem 3.8]{VOIC} is actually stronger: it
involves additional deterministic matrices.%
}, referring to the fact that in the limit, as $n\to\infty$, the family
$\{U_{i}^{\left(n\right)},(U_{i}^{\left(n\right)})^{*}\}_{i\in\left[r\right]}$
can be modeled by the {}``Free Probability Theory'' developed by
Voiculescu (see, for example, \cite{VOIC85} and the monograph \cite{VDN}).
Such asymptotic freeness results are known for broad families of ensembles%
\footnote{In the case of unitary matrices, we analyze expressions with negative
exponents because $(U_{i}^{(n)})^{-1}=(U_{i}^{(n)})^{*}$. In the
general case, one does not allow negative exponents $\e_{j}$.%
}, including general Gaussian random matrices (due to Voiculescu in
the same paper \cite[Theorem 2.2]{VOIC}). In later works \eqref{eq:firstorder}
is strengthened to $\trw\left(n\right)=O\left(\frac{1}{n}\right)$
whenever $w\ne1$ \cite{MSS07,Radulescu06}.

Although our results are more general, we first describe them in the
special case of the expected trace $\trw\left(n\right)$, and defer
the discussion of the general results to Section \ref{sub:Expected-product-of}.
The starting point for this paper is the intriguing observation that
the $w$-measure on any compact group, and in particular, the $w$-measure
on $\U\left(n\right)$ and the quantity $\trw\left(n\right)$, are
invariant under $w\mapsto\theta\left(w\right)$ for any $\theta\in\Aut\left(\F_{r}\right)$
(see Section \ref{sub:Word-measures}). It follows that this quantity
is determined by some algebraic, $\mathrm{Aut}\left(\F_{r}\right)$-invariant,
properties of the word $w$. 

The first step in our analysis of $\trw\left(n\right)$ builds on
results of Xu and of Collins and \'{S}niady \cite{xu1997random,CS}.
In Section \ref{sec:A-Rational-Expression} we explain how it follows
readily from these results that $\trw\left(n\right)$ is a rational
function of $n$ with coefficients in $\mathbb{Q}$ (which can be
algorithmically computed)%
\footnote{Such a formula, in a slightly more restricted version, appears also
in \cite{Radulescu06}.%
}. For example, this function is $\frac{-4}{n^{3}-n}$ for $w=\left[x_{1},x_{2}\right]^{2}$
-- see \eqref{eq:rational-expression-for-[a,b]^2} below. This function
can hence be written as a Laurent series in $n^{-1}$ with rational
coefficients. By \eqref{eq:firstorder}, whenever $w\ne1$ we may
write $\trw\left(n\right)$ as a power series: 
\[
\trw\left(n\right)\in\mathbb{Q}\left[\frac{1}{n}\right].
\]
Unlike previous works, we are not only interested in the limit $\lim_{n\to\infty}\trw\left(n\right)$.
Rather, in this paper our aim is to explain the leading term of $\trw\left(n\right)$.
That is, we give algebraic interpretation for the following two quantities: 
\begin{description}
\item [{Leading exponent}] The exponent of the leading order term of $\trw(n)$
\item [{Leading coefficient}] The coefficient of the leading order term
of $\trw(n)$
\end{description}
The second of these two quantities is the more subtle%
\footnote{\label{fn:leading-coef-vanishes}To be precise, there are degenerate
cases where the coefficient we explain vanishes --- see Example \ref{example:[x,y][x,z] leading term 0}
and Section \ref{sec:Examples}. In these cases we lose track of the
leading coefficient and only obtain a lower bound for the leading
exponent.%
}.

In fact, an easy observation is that unless $w$ is in the commutator
subgroup $[\F_{r},\F_{r}]$, the expected trace $\trw\left(n\right)$
vanishes for every $n$ (Claim \ref{claim: tr=00003D0 for non-balanced words}
below). The interesting case is, therefore, when $w\in[\F_{r},\F_{r}]$.
Every word in this subgroup is a product of commutators, and the \textbf{\textit{\emph{commutator
length}}} $\cl(w)$\marginpar{$\cl(w)$}\label{cl} of the word $w$
is the smallest $g$ such that $w$ is a product of $g$ commutators.
Namely, the smallest $g$ for which 
\begin{equation}
w=[u_{1},v_{1}][u_{2},v_{2}]\ldots[u_{g},v_{g}]\label{eq:commutators}
\end{equation}
for some $u_{i},v_{i}\in\F_{r}$. The theory of commutator length
suffices to explain the leading exponent of $\trw(n)$ (modulo the
exceptional event mentioned in Footnote \ref{fn:leading-coef-vanishes}):
\begin{thm}
\label{thm:leading-exponent}Let $w\in\left[\F_{r},\F_{r}\right]$
and denote $g=\cl(w)$. Then,
\[
\trw\left(n\right)=O\left(\frac{1}{n^{2g-1}}\right).
\]
(The big $O$ notation is in the regime $n\to\infty$.)
\end{thm}
The analysis of the leading coefficient necessitates a subtler study,
not only of the commutator length of $w$, but also of the set of
products of commutators of length $\cl\left(w\right)$ giving $w$.
To formalize this, consider the following. Let $a_{1},b_{1},\ldots,a_{g},b_{g}$\marginpar{${\scriptstyle a_{1},b_{1},\ldots,a_{g},b_{g}}$}
be generators of $\F_{2g}$, where $g=\cl\left(w\right)$ as above,
and let \marginpar{$\delta_{g}$}$\delta_{g}=[a_{1},b_{1}]\ldots[a_{g},b_{g}]$.
Solutions to \eqref{eq:commutators} correspond to elements $\phi\in\Hom(\F_{2g},\F_{r})$
such that 
\begin{equation}
\phi(\delta_{g})=w.\label{eq:deltaimage}
\end{equation}
We write $\Hom_{w}(\F_{2g},\F_{r})$\marginpar{${\scriptstyle \Hom_{w}\left(\F_{2g},\F_{r}\right)}$}
for the set of homomorphisms $\F_{2g}\to\F_{r}$ satisfying \eqref{eq:deltaimage}.
The group $\Aut(\F_{2g})$ acts on $\Hom(\F_{2g},\F_{r})$ by precomposition.
We define $\Aut_{\delta}(\F_{2g})$\marginpar{$\Aut_{\delta}(\F_{2g})$}
to be the stabilizer in $\Aut(\F_{2g})$ of $\delta_{g}$. For example,
the automorphism $a_{1}\mapsto a_{1}b_{1}$ (leaving all other generators
unchanged) is in%
\footnote{Our convention is that $\left[x,y\right]=xyx^{-1}y^{-1}$.%
} $\mathrm{Aut}_{\delta}\left(\F_{2g}\right)$ while $a_{1}\longleftrightarrow b_{1}$
is not. 

Clearly, $\Aut_{\d}(\F_{2g})$ acts on $\Hom_{w}(\F_{2g},\F_{r})$,
the solution space to \eqref{eq:commutators}, for every $w$. We
think of the orbits $\Aut_{\d}(\F_{2g})\backslash\Hom_{w}(\F_{2g},\F_{r})$
as equivalence classes of solutions. So the elements of $\mathrm{Aut}_{\delta}\left(\F_{2g}\right)$
permute the solutions inside the same equivalence class. For instance,
the automorphism $a_{1}\mapsto a_{1}b_{1}$ mentioned above shows
that the solutions $\left[x_{1},x_{2}\right]$ and $\left[x_{1}x_{2},x_{2}\right]$
belong to the same class. Occasionally, elements of $\mathrm{Aut}_{\delta}\left(\F_{2g}\right)$
stabilize a solution. For example, consider the word $w=\left[x_{1},x_{2}\right]^{2}$.
It can be shown that its commutator length is $g=2$, and that it
has a single class of solutions. The solution $\left[x_{1,}x_{2}\right]\left[x_{1},x_{2}\right]$
is stabilized by the automorphism%
\footnote{We often use the handy convention that capital letters mark inverses.
For example, $A_{1}$ is $a_{1}^{-1}$, the inverse of $a_{1}$.%
} 
\[
a_{1}\mapsto a_{1}a_{2}a_{1}A_{2}A_{1}\,\,\,\,\, b_{1}\mapsto a_{1}a_{2}A_{1}A_{2}b_{1}a_{1}^{\,2}A_{2}A_{1}\,\,\,\,\, a_{2}\mapsto a_{1}a_{2}A_{1}\,\,\,\,\, b_{2}\mapsto b_{2}a_{2}A_{1},
\]
which belongs to $\mathrm{Aut}_{\delta}\left(\F_{4}\right)$. For
every class $[\phi]\in\Aut_{\d}(\F_{2g})\backslash\Hom_{w}(\F_{2g},\F_{r})$,
the stabilizer of any representative $\phi$ belongs to a well-defined
conjugacy class of subgroups of $\Aut_{\delta}(\F_{2g})$.

As we show below, the leading coefficient of $\trw\left(n\right)$
is controlled by the set of equivalence classes of solutions to \eqref{eq:commutators},
and by the isomorphism type of the stabilizer in every class. The
important invariant of the stabilizers is their \emph{Euler characteristic}.

The Euler characteristic of a group is defined for a large class of
groups of certain finiteness conditions (see \cite[Chapter IX]{BROWN}).
The simplest case is when a group $\Gamma$ admits a finite CW-complex
as Eilenberg-MacLane space of type%
\footnote{\label{fn:An-Eilenberg-MacLane-space}An Eilenberg-MacLane space of
type $\mathrm{K}\left(\Gamma,1\right)$, or simply a $\mathrm{K\left(\Gamma,1\right)}$-space,
is a path-connected topological space with fundamental group isomorphic
to $\Gamma$ and with a contractible universal cover (e.g.~\cite[Section I.4]{BROWN}).%
} $\mathrm{K}\left(\Gamma,1\right)$. In this case, the Euler characteristic\label{The-Euler-characteristic}
$\chi\left(\Gamma\right)$ coincides with the topological Euler characteristic
of the $\mathrm{K}\left(\Gamma,1\right)$-space, and, in particular,
is an integer.

We can now state our main theorem regarding $\trw\left(n\right)$,
which is a more detailed version of Theorem \ref{thm:leading-exponent}:
\begin{thm}
\label{thm:main}Let $w\in\left[\F_{r},\F_{r}\right]$ and denote
$g=\cl(w)$. Then, 
\begin{align*}
\trw(n)=\frac{1}{n^{2g-1}}\left[\sum_{[\phi]\in\Aut_{\d}(\F_{2g})\backslash\Hom_{w}(\F_{2g},\F_{r})}\chi\left(\Stab_{\Aut_{\d}(\F_{2g})}\left(\phi\right)\right)\right]+O\left(\frac{1}{n^{2g+1}}\right).
\end{align*}
(Again, the big $O$ notation is in the regime $n\to\infty$.)\end{thm}
\begin{rem}
\label{remark: free solutions}Note that when $\phi\in\Hom_{w}(F_{2g},F_{r})$
is injective, $\Stab_{\Aut_{\d}(\F_{2g})}\left(\phi\right)$ is trivial
(indeed, even $\mathrm{Stab}_{\mathrm{Aut}\left(\F_{2g}\right)}\left(\phi\right)$
is trivial), and so its Euler characteristic is 1. This is the case
precisely when $\left\{ \phi\left(a_{1}\right),\phi\left(b_{1}\right),\ldots,\phi\left(a_{g}\right),\phi\left(b_{g}\right)\right\} $
is a free set in $\F_{r}$, which is in some sense the generic case.
Therefore, one could say

{}``The leading coefficient of $\trw(n)$ counts the number of equivalence
classes of solutions to \eqref{eq:commutators}, up to corrections
for the existence of non-trivial stabilizers.''

For instance, when $g=1$, namely, when $w$ is a commutator, $\phi\left(a_{1}\right)$
and $\phi\left(b_{1}\right)$ are necessarily free (otherwise they
commute and $w=1$). Hence, if $\cl\left(w\right)=1$ and $K$ marks
the number of equivalence classes of solutions to $\left[u,v\right]=w$,
then $\trw\left(n\right)=\frac{K}{n}+O\left(\frac{1}{n^{3}}\right)$.
As an example%
\footnote{In fact, in this particular case, $\tr_{\left[x_{1}^{\, k},x_{2}\right]}\left(n\right)=\frac{k}{n}$
with no further terms.%
}, $\tr_{\left[x_{1}^{k},x_{2}\right]}\left(n\right)=\frac{k}{n}+O\left(\frac{1}{n^{3}}\right)$,
the different solution classes represented by $\left[x_{1}^{\, k},x_{2}x_{1}^{\, j}\right]$,
$0\le j\le k-1$.

\medskip{}

\end{rem}
The fact that $\mathrm{Stab}_{\mathrm{Aut}_{\delta}\left(\F_{2g}\right)}\left(\phi\right)$
has a well-defined Euler characteristic, which is moreover an integer,
follows from the following:
\begin{thm}
\label{thm:stabilizers K(G,1)}Let $w\in\left[\F_{r},\F_{r}\right]$
and denote $g=\cl\left(w\right)$. For every $\phi\in\mathrm{Hom}_{w}\left(\F_{2g},\F_{r}\right)$,
the stabilizer 
\[
G\overset{\mathrm{def}}{=}\Stab_{\Aut_{\delta}(\F_{2g})}\left(\phi\right)\leq\Aut_{\delta}(\F_{2g})
\]
admits a finite simplicial complex as a $\mathrm{K}\left(G,1\right)$-space.
\end{thm}
In particular, the stabilizer is finitely presented. The particular
finite simplicial complex we construct as a $\mathrm{K}\left(G,1\right)$-space
for the stabilizer yields further properties such as solvability of
the word problem. We elaborate more in Section \ref{sec:More-Consequences}.

\subsection{Expected product of traces\label{sub:Expected-product-of}}

For every finite set of words $w_{1},\ldots,w_{\ell}\in\F_{r}$ consider
the expected product of traces\marginpar{$\tr_{w_{1},\ldots,w_{\ell}}\left(n\right)$}
\[
\tr_{w_{1},\ldots,w_{\ell}}\left(n\right)\overset{\mathrm{def}}{=}\mathbb{E}\left[\mathrm{tr}\left(w_{1}\left(U_{1}^{\left(n\right)},\ldots,U_{r}^{\left(n\right)}\right)\right)\cdot\ldots\cdot\mathrm{tr}\left(w_{\ell}\left(U_{1}^{\left(n\right)},\ldots,U_{r}^{\left(n\right)}\right)\right)\right].
\]
The results we described in Section \ref{sub:The-expected-trace}
for single words generalize to finite sets of words.

The numbers $\tr_{w_{1},\ldots,w_{\ell}}\left(n\right)$ were studied
before. Diaconis and Shahshahani \cite{diaconis1994eigenvalues} consider
the joint distribution of $\mathrm{tr}\left(U^{\left(n\right)}\right),\mathrm{tr}((U^{\left(n\right)})^{2}),\mathrm{tr}((U^{\left(n\right)})^{3}),\ldots$
(here $U^{\left(n\right)}\in{\cal U}\left(n\right)$ is Haar random).
They show that these random variables converge in distribution to
independent variables, and as $n\to\infty$, $\mathrm{tr}((U^{\left(n\right)})^{j})$
converges to $\sqrt{j}Z$, where $Z$ is a standard complex normal
variable. This work can be interpreted as the study of (limits of)
word measures when the words are in $F_{1}\cong\mathbb{Z}$. Later,
Mingo, \'{S}niady and Speicher \cite{MSS07}, and independently R\v{a}dulesco
\cite{Radulescu06}, generalized this result to words in $\F_{r}$,
$r\ge2$. (The main goal of \cite{MSS07} is to establish {}``second
order freeness'' of random unitary matrices.) Namely, given $w_{1},\ldots,w_{k}\in\F_{r}$,
they consider the random variables 
\[
\mathrm{tr}\left(w_{1}\left(U_{1}^{\left(n\right)},\ldots,U_{r}^{\left(n\right)}\right)\right)\,,\,\ldots\,,\,\mathrm{tr}\left(w_{k}\left(U_{1}^{\left(n\right)},\ldots,U_{r}^{\left(n\right)}\right)\right),
\]
and study their joint distribution in the limit as $n\to\infty$.
All of \cite{diaconis1994eigenvalues}, \cite{MSS07} and \cite{Radulescu06}
use the method of moments which translates the study of the joint
limit distribution to the study of (limits as $n\to\infty$ of) expected
products of traces, namely, of $\tr_{w_{1},\ldots,w_{\ell}}\left(n\right)$
for all possible finite subsets $\left\{ w_{1},\ldots,w_{\ell}\right\} $
of $\F_{r}\setminus\left\{ 1\right\} $.

As in the case of $\trw\left(n\right)$ -- the expected trace of a
single word -- $\tr_{w_{1},\ldots,w_{\ell}}\left(n\right)$ can also
be written as a rational expression in $n$ (see Theorem \ref{thm:trw-rational}).
As in \eqref{eq:firstorder}, the main interest of \cite{MSS07} and
\cite{Radulescu06} is in $\lim_{n\to\infty}\tr_{w_{1},\ldots,w_{\ell}}\left(n\right)$,
namely, in the free coefficient of $\tr_{w_{1},\ldots,w_{\ell}}\left(n\right)$
as a Laurent series in $\frac{1}{n}$. We explain their result in
Example \ref{example:special case of radulescu and MSS} below. Our
goal is to explain the leading term (exponent and coefficient) of
this rational expression, even when $\tr_{w_{1},\ldots,w_{\ell}}\left(n\right)=O\left(\frac{1}{n}\right)$.

Indeed, we establish parallels to Theorems \ref{thm:leading-exponent},
\ref{thm:main} and \ref{thm:stabilizers K(G,1)} for the more general
object $\tr_{w_{1},\ldots,w_{\ell}}\left(n\right)$. We introduce
these general results in geometric terms rather than algebraic: the
geometric language here is more natural both in terms of the statements
of the results and in terms of the proofs.

\medskip{}

The geometric interpretation of commutator length of words goes back
to Culler \cite{CULLER} and explains why $\cl\left(w\right)$ is
often called {}``the genus of $w$''. In the geometric approach,
solutions to the commutator equation \eqref{eq:commutators} are given
in terms of maps%
\footnote{All maps in this paper are assumed to be continuous.%
} from surfaces with boundary to a wedge of circles. More concretely,
we think of the free group $\F_{r}$ as the fundamental group of a
bouquet of $r$ cycles, denoted $\wedger$\marginpar{$\wedger$},
pointed at the wedge point $o$\marginpar{$o$}. For the free group
$\F_{2g}=\F\left(a_{1},b_{1},\ldots,a_{g},b_{g}\right)$ we consider
a different topological space: the oriented surface of genus $g$
with one boundary component, which we denote by $\surface$\marginpar{$\surface$}.
Let $v_{1}$\marginpar{$v_{1}$} be a basepoint of $\surface$ at
the boundary. Let $\left(S^{1},1\right)$ be a cycle pointed at $1$,
and let \marginpar{$\partial_{1}$}$\partial_{1}:\left(S^{1},1\right)\to\left(\surface,v_{1}\right)$
be a fixed map which identifies the boundary of $\surface$ with $S^{1}$.
Identify $a_{1},b_{1},\ldots,a_{g},b_{g}$ with a suitable basis of
$\pi_{1}\left(\surface,v_{1}\right)$ so that $\delta_{g}=\left[a_{1},b_{1}\right]\ldots\left[a_{g},b_{g}\right]$
is represented by $\left[\partial_{1}\right]$. It is shown in \cite{CULLER}
that every solution to \eqref{eq:commutators} can be given by a map
$f\colon\left(\surface,v_{1}\right)\to\left(\wedger,o\right)$ with
$f_{*}\left(\left[\partial_{1}\right]\right)=w$. In fact, there is
a one-to-one correspondence between the solutions in $\mathrm{Hom}_{w}\left(\F_{2g,}\F_{r}\right)$
and homotopy classes of such maps $\left(\surface,v_{1}\right)\to\left(\wedger,o\right)$
(see Proposition \ref{prop:genus(w)}).

We now describe the geometric analogue of $\mathrm{Aut}_{\delta}\left(\F_{2g}\right)$.
For this sake, we first fix the map from the boundary of $\surface$
to the wedge. Formally, for every $w\in\F_{r}$ fix\marginpar{$f_{w}$}\label{f_w}
\[
f_{w}:\left(S^{1},1\right)\to\left(\wedger,o\right)
\]
a map which describes a fixed loop in $\wedger$ representing $w$,
namely, $\left[f_{w}\right]=w\in\pi_{1}\left(\wedger,o\right)$, and
consider the set of maps 
\begin{equation}
\left\{ f:\left(\surface,v_{1}\right)\to\left(\wedger,o\right)\,\middle|\, f\circ\partial_{1}=f_{w}\right\} .\label{eq:maps-with-prescribed-boundary}
\end{equation}
Let $\mathrm{Homeo}_{\delta}\left(\surface\right)$ denote the group
of homeomorphisms of $\surface$ that fix the boundary pointwise,
and write $\mathrm{Homeo}_{0}\left(\surface\right)$ for the normal
subgroup of $\mathrm{Homeo}_{\delta}\left(\surface\right)$ consisting
of homeomorphisms isotopic to the identity. While $\mathrm{Homeo}_{\delta}\left(\surface\right)$
acts on the set of maps in \eqref{eq:maps-with-prescribed-boundary}
by precomposition, the quotient by $\mathrm{Homeo}_{0}\left(\surface\right)$
acts on homotopy classes of these maps. This quotient is precisely
the \textbf{mapping class group} of $\surface$:\marginpar{$\mathrm{MCG}\left(\surface\right)$}
\[
\MCG\left(\surface\right)\overset{\mathrm{def}}{=}\mathrm{Homeo}_{\delta}\left(\surface\right)/\mathrm{Homeo}_{0}\left(\surface\right).
\]
The Dehn-Nielsen-Baer theorem (Theorem \ref{thm:Dehn-Nielsen-Baer}
below) states there is a natural isomorphism $\mathrm{Aut}_{\delta}\left(\F_{2g}\right)\cong\mathrm{MCG}\left(\surface\right)$.
Through this isomorphism, the action of $\mathrm{Aut}_{\delta}\left(\F_{2g}\right)$
on $\mathrm{Hom}_{w}\left(\F_{2g},\F_{r}\right)$ is identical to
the action of $\mathrm{MCG}\left(\surface\right)$ on the homotopy
classes of maps in \eqref{eq:maps-with-prescribed-boundary}. We summarize
this algebra-geometry dictionary in Table \ref{tab:Algebra-geometry-dictionary}.
We give more details and further explanations in Section \ref{sub:cl-of-word}. 

\begin{table}
\begin{tabular}{|>{\centering}p{7cm}|>{\centering}p{9cm}|}
\hline 
$\F_{r}$ & $\pi_{1}\left(\wedger,o\right)$\tabularnewline
\hline 
$\F_{2g}=\F\left(a_{1},b_{1},\ldots,a_{g},b_{g}\right)$\\
~

$\delta_{g}=\left[a_{1},b_{1}\right]\ldots\left[a_{g},b_{g}\right]$ & $\pi_{1}\left(\surface,v_{1}\right)$\\
with fixed loops representing $a_{1},b_{1},\ldots,a_{g},b_{g}$ so
that \\
$\left[\partial_{1}\right]=\delta_{g}$\tabularnewline
\hline 
$\mathrm{Hom}_{w}\left(\F_{2g},\F_{r}\right)$ & homotopy classes of $\left\{ f\colon\left(\surface,v_{1}\right)\to\left(\wedger,o\right)\,\middle|\, f\circ\partial_{1}=f_{w}\right\} $\tabularnewline
\hline 
$\cl\left(w\right)$ & $\min\left\{ g\,\middle|\,\exists f\colon\surface\to\wedger\,\mathrm{with}\, f\circ\partial_{1}=f_{w}\right\} $\tabularnewline
\hline 
$\mathrm{Aut}_{\delta}\left(\F_{2g}\right)$ & $\mathrm{MCG}\left(\surface\right)$\tabularnewline
\hline 
equivalence classes of solutions: $\mathrm{Aut}_{\delta}\left(\F_{2g}\right)\backslash\mathrm{Hom}_{w}\left(\F_{2g},\F_{r}\right)$ & $\mathrm{MCG}\left(\surface\right)\backslash\left\{ \mathrm{homotopy\, classes\, of\, maps\, as\, above}\right\} $\tabularnewline
\hline 
\end{tabular}

\caption{\label{tab:Algebra-geometry-dictionary}Algebra-geometry dictionary.}
\end{table}

\medskip{}

We can now describe our general results. To deal with multiple words,
we need surfaces with multiple boundary components. More concretely, 
\begin{defn}
\label{def: admissible maps}Let $\Sigma$ be a surface and $f\colon\Sigma\to\wedger$.
We say that $\left(\Sigma,f\right)$ is \textbf{admissible}\marginpar{$\left(\Sigma,f\right)$ admissible for $w_{1},\ldots,w_{\ell}$}
for $w_{1},\ldots,w_{\ell}\in\F_{r}$ if the following three conditions
hold:
\begin{enumerate}
\item $\Sigma$ is compact, oriented, with $\ell$ boundary components,
and contains no closed connected components (but is not necessarily
connected).
\item $\Sigma$ has $\ell$ marked points $v_{1},\ldots,v_{\ell}$ on $\Sigma$,
one point in every boundary component, and fixed identifications of
the boundaries with $S^{1}$ with common orientation given by\marginpar{$\partial_{1},\ldots,\partial_{\ell}$}
\[
\partial_{1}:\left(S^{1},1\right)\to\left(\Sigma,v_{1}\right)\,\,\,\,\,\,\,\,\ldots\,\,\,\,\,\,\,\,\partial_{\ell}:\left(S^{1},1\right)\to\left(\Sigma,v_{\ell}\right)
\]

\item $f$ maps the boundary components to $w_{1},\ldots,w_{\ell}$, namely,
\[
f\circ\partial_{1}=f_{w_{1}}\,\,\,\,\,\,\,\,\ldots\,\,\,\,\,\,\,\, f\circ\partial_{\ell}=f_{w_{\ell}}.
\]

\end{enumerate}
\end{defn}
In particular, every admissible map sends the marked points $v_{1},\ldots,v_{\ell}$
to $o$. The next definition captures the maximal possible Euler characteristic
of an admissible surface:
\begin{defn}
\label{def:chi(words)}For $w_{1},\ldots,w_{\ell}\in\F_{r}$ define\marginpar{${\scriptstyle \ch\left(w_{1},\ldots,w_{\ell}\right)}$}
\[
\ch\left(w_{1},\ldots,w_{\ell}\right)\overset{\mathrm{def}}{=}\max\left\{ \chi\left(\Sigma\right)\,\middle|\,\left(\Sigma,f\right)\,\mathrm{is\, admissible\, for}\, w_{1},\ldots,w_{\ell}\right\} ,
\]
where $\chi\left(\Sigma\right)$ is the Euler characteristic of $\Sigma$.
If no such surface exists, define $\ch\left(w_{1},\ldots,w_{\ell}\right)=-\infty$.
\end{defn}
As we explain below, $\ch\left(w_{1},\ldots,w_{\ell}\right)\ne-\infty$
(i.e.~there exists an admissible map for $w_{1},\ldots,w_{\ell}$),
if and only if the product $w_{1}w_{2}\cdots w_{\ell}\in\left[\F_{r},\F_{r}\right]$.
Equivalently, this holds if and only if the sum of exponents of the
letter $x_{i}$ across $w_{1},\ldots,w_{\ell}$ is zero for every
$1\le i\le r$. As for a single word, if $w_{1}\cdots w_{\ell}\notin\left[\F_{r},\F_{r}\right]$
then $\tr_{w_{1},\ldots,w_{\ell}}\left(n\right)\equiv0$ vanishes
for every $n$ (Claim \ref{claim: tr=00003D0 for non-balanced words}). 
\begin{rem}
For a single word, $\ch\left(w\right)=1-2\cdot\cl\left(w\right)$.
More generally, the commutator length of a finite set of words $w_{1},\ldots,w_{\ell}\in\F_{r}$,
introduced by Calegari (e.g.~\cite[Definition 2.71]{calegari2009scl}),
is defined as the smallest number of commutators whose product is
equal to an expression of the form 
\[
w_{1}t_{1}w_{2}t_{1}^{-1}\ldots t_{\ell}w_{\ell}t_{\ell}^{-1}
\]
with $t_{2},\ldots,t_{\ell}\in\F_{r}$. This number, which can be
denoted $\cl\left(w_{1},\ldots,w_{\ell}\right)$, relates to $\ch\left(w_{1},\ldots,w_{\ell}\right)$
by 
\[
\ch\left(w_{1},\ldots,w_{\ell}\right)=2-\ell-2\cdot\cl\left(w_{1},\ldots,w_{\ell}\right)
\]
(when $w_{1},\ldots,w_{\ell}\ne1$). However, $\ch\left(\right)$
is more natural then $\cl\left(\right)$ in this general case: it
simplifies the statement of our results below, and appears more directly
in the proofs%
\footnote{Another advantage of $\ch\left(\right)$ compared with $\cl\left(\right)$
is that with $\ch\left(\right)$, the statements of our results remain
valid when some of the words are the identity element $1\in\F_{r}$.
(Observe that $\ch\left(w_{1},\ldots,w_{\ell},1\right)=\ch\left(w_{1},\ldots,w_{\ell}\right)+1$.)%
}. In fact, Calegari himself also mostly uses the geometric definition
in his works.
\end{rem}
With this definition, the leading exponent from Theorem \ref{thm:leading-exponent}
is simply $n^{\ch\left(w\right)}$. This generalizes to
\begin{thm}
\label{thm:leading-exponent-general}For $w_{1},\ldots,w_{\ell}\in\F_{r}$
we have 
\[
\tr_{w_{1},\ldots,w_{\ell}}\left(n\right)=O\left(n^{\ch\left(w_{1},\ldots,w_{\ell}\right)}\right).
\]

\end{thm}
Next, in order to state our result for the leading coefficient of
$\tr_{w_{1},\ldots,w_{\ell}}\left(n\right)$, we define equivalence
classes of {}``solutions'', namely, of admissible maps of maximal
Euler characteristic, for the words $w_{1},\ldots,w_{\ell}$. We say
that two admissible maps $\left(\Sigma,f\right)$ and $\left(\Sigma',f'\right)$
are equivalent, and denote $\left(\Sigma,f\right)\sim\left(\Sigma',f'\right)$\marginpar{${\scriptstyle \left(\Sigma,f\right)\sim\left(\Sigma',f'\right)}$}\label{(Sigma,f) sim (Sigma',f')},
if there is an homeomorphism $\rho\colon\Sigma\to\Sigma'$ so that
$f\simeq f'\circ\rho$ are homotopic relative $\partial\Sigma$ while
the boundary components are identified pointwise, that is, $\partial_{i}'=\rho\circ\partial_{i}$
for $1\le i\le\ell$.
\begin{defn}
\label{def:sol(words)}For $w_{1},\ldots,w_{\ell}\in\F_{r}$ let $\sol\left(w_{1},\ldots,w_{\ell}\right)$\marginpar{${\scriptstyle \sol\left(w_{1},\ldots,w_{\ell}\right)}$}
denote the set of equivalence classes of {}``solutions'', or admissible
maps of maximal Euler characteristic, for $w_{1},\ldots,w_{\ell}$.
Namely,
\[
\sol\left(w_{1},\ldots,w_{\ell}\right)\overset{\mathrm{def}}{=}\left\{ \left(\Sigma,f\right)\,\middle|\,\begin{gathered}\left(\Sigma,f\right)\,\mathrm{is\, admissible\, for\:}w_{1},\ldots,w_{\ell},\,\mathrm{and}\\
\chi\left(\Sigma\right)=\ch\left(w_{1},\ldots,w_{\ell}\right)
\end{gathered}
\right\} /\left(\Sigma,f\right)\sim\left(\Sigma',f'\right).
\]
We denote by \marginpar{$\left[\left(\Sigma,f\right)\right]$}$\left[\left(\Sigma,f\right)\right]$
the equivalence class of the admissible map $\left(\Sigma,f\right)$.
\end{defn}
We can now state the more detailed version of Theorem \ref{thm:leading-exponent-general}
which generalizes Theorem \ref{thm:main}:
\begin{thm}
\label{thm:main - general}Let $w_{1},\ldots,w_{\ell}\in\F_{r}$.
Then, 
\[
\tr_{w_{1},\ldots,w_{\ell}}\left(n\right)=n^{\ch\left(w_{1},\ldots,w_{\ell}\right)}\left[\sum_{\left[\left(\Sigma,f\right)\right]\in\sol\left(w_{1},\ldots,w_{\ell}\right)}\chi\left(\mathrm{Stab}_{\mathrm{MCG}\left(\Sigma\right)}\left(\tilde{f}\right)\right)\right]+O\left(n^{\ch\left(w_{1},\ldots,w_{\ell}\right)-2}\right),
\]
where $\tilde{f}$\marginpar{$\tilde{f}$} is the homotopy class of
$f$ (relative the boundary of $\Sigma$).

As above, $\mathrm{MCG}\left(\Sigma\right)$\label{MCG(Sigma)}\marginpar{$\mathrm{MCG}\left(\Sigma\right)$}
is the mapping class group of the surface $\Sigma$, consisting of
mapping classes which fix the boundary pointwise. It acts on homotopy
classes of maps from the surface by precomposition.
\end{thm}
Indeed, $\sol\left(\wl\right)$ is always a finite set (see Corollary
\ref{cor:Finitely many incompressible maps}). Finally, we need to
justify our usage of the Euler characteristic%
\footnote{Note that our results use two different instances of Euler characteristics.
On the one hand, they use Euler characteristics of compact surfaces,
and on the other hand the Euler characteristic of stabilizer subgroups,
or of the corresponding $K\left(G,1\right)$-spaces. We try to ease
the confusion by using the notation $\ch\left(w_{1},\ldots,w_{\ell}\right)$
for the former (instead of, say, the more natural $\chi\left(w_{1},\ldots,w_{\ell}\right)$).%
} of the stabilizers of the maps in $\sol\left(w_{1},\ldots,w_{\ell}\right)$,
namely, to give a generalized version of Theorem \ref{thm:stabilizers K(G,1)}.
It turns out that the crucial property of these maps is their being
incompressible:
\begin{defn}
\label{def:incompressible}A map $f$ from a surface $\Sigma$ to
a topological space is called \textbf{compressible} if there is a
non-nullhomotopic simple closed curve $\gamma$ in $\Sigma$ such
that $f\left(\gamma\right)$ is (freely) nullhomotopic. Otherwise,
$f$ is called \textbf{incompressible.}
\end{defn}
This term is standard (see, e.g., \cite{calegari2009scl}). It incorporates
maps solving the commutator equation \eqref{eq:commutators}, and
more generally, maps in $\sol\left(w_{1},\ldots,w_{\ell}\right)$:
if $\left(\Sigma,f\right)$ is admissible for $w_{1},\ldots,w_{\ell}$
and $f$ is compressible, one can cut $\Sigma$ along the compressing
simple closed curve $\gamma$, cap with two discs to obtain a new
surface $\Sigma'$ and extend $f$ to a map $f'$ from $\Sigma'$.
But then $\left(\Sigma',f'\right)$ is also admissible for $w_{1},\ldots,w_{\ell}$
with $\chi\left(\Sigma'\right)=\chi\left(\Sigma\right)+2$. So $\Sigma$
cannot be of maximal Euler characteristic. 
\begin{thm}
\label{thm:K(G,1) for incompressible}Let $\tilde{f}\colon\Sigma\to\wedger$
be a homotopy class (relative $\partial\Sigma$) of incompressible
maps from a compact oriented surface to the wedge. Then the stabilizer
\[
G=\mathrm{Stab}_{\mathrm{MCG}\left(\Sigma\right)}\left(\tilde{f}\right)
\]
admits a finite simplicial complex as $K\left(G,1\right)$-space.
\end{thm}
We remark the statement is void when $\Sigma$ has a closed connected
component of positive genus: there are no incompressible maps from
a closed surface to the wedge%
\footnote{For example, this can be seen using the proof of Theorem 1.4 in \cite{CULLER},
by turning a map from a closed surface to a {}``tight'' map.%
}. 

The following special case of Theorem \ref{thm:main - general} is
due to \cite{MSS07} and \cite{Radulescu06}:
\begin{example}
\label{example:special case of radulescu and MSS}Consider the limit
\begin{equation}
\lim_{n\to\infty}\tr_{w_{1},\ldots,w_{\ell}}\left(n\right)\label{eq:limit-of-tr_w1...wl}
\end{equation}
for $w_{1},\ldots,w_{\ell}\ne1$. Assume $\left(\Sigma,f\right)$
is admissible for $w_{1},\ldots,w_{\ell}$. By definition, every connected
component of $\Sigma$ has non-empty boundary, so its Euler characteristic
is negative unless it is a disc or an annulus. But a disc is impossible
as we assume $w_{1},\ldots,w_{\ell}\ne1$. Thus, the only case in
which \eqref{eq:limit-of-tr_w1...wl} is non-zero is when there is
an admissible pair $\left(\Sigma,f\right)$ with $\Sigma$ is a disjoint
union of (one or more) annuli. In every annulus $A$, if $w$ and
$w'$ are the two words on the boundary components, then necessarily
$w^{-1}$ is conjugate to $w'$. Moreover, write $w=u^{d}$ with $u\in\F_{r}$
a non-power and $d\ge1$, then the number of equivalence classes of
maps $h$ such that $\left(A,h\right)$ is admissible for $w,w'$
is exactly $d$. Since the mapping class group of the annulus is simple
to analyze (isomorphic to $\mathbb{Z},$ generated by a Dehn twist),
it is not hard to see the stabilizers $\mathrm{Stab}_{\mathrm{MCG}\left(A\right)}\left(\tilde{h}\right)$
are always trivial. 

These considerations yield Theorem 4.1 in \cite{Radulescu06}%
\footnote{The same theorem is an immediate consequence of Theorem 2 in \cite{MSS07}.
In \cite{Radulescu06} the theorem is shown, for simplicity, only
for $\F_{2}$ (in our analysis there is no saving in restricting to
$\F_{2}$). %
}: \eqref{eq:limit-of-tr_w1...wl} is non-zero if and only if $w_{1},\ldots,w_{\ell}$
can be matched in pairs in which each word is conjugate to the inverse
of its mate. In this case, the limit in \eqref{eq:limit-of-tr_w1...wl}
is equal to the number of such matchings, times the product of exponents
of the words (one exponent for every pair).
\end{example}
Because of the degenerate case described in Footnote \ref{fn:leading-coef-vanishes},
it is not clear whether the commutator length $\cl\left(w\right)$
is determined by the $w$-measures on $\left\{ {\cal U}\left(n\right)\right\} _{n\in\mathbb{N}}$,
or, more generally, if $\ch\left(\wl\right)$ is determined by the
joint measures of $\wl$ on unitary groups. However, the measures
do determine a related number, the \emph{stable commutator length
}of $w$. This algebraic quantity is defined by 
\begin{equation}
\mathrm{scl}(w)\equiv\lim_{m\to\infty}\frac{\cl(w^{m})}{m}.\label{eq:scl}
\end{equation}
(There is an analogous definition for finite set of words.) There
is a deep theory behind this invariant, and for background we refer
to the short survey \cite{CALWHATIS} and long one \cite{calegari2009scl}
by Calegari. Relying on the rationality result of Calegari \cite{CALRATIONAL}
that shows, in particular, that $\mathrm{scl}$ takes on rational
values in $\F_{r}$, we are able to show the following:
\begin{cor}
\label{cor:can hear scl}The stable commutator length of a word $w\in\F_{r}$
can be {}``read'' from the measures it induces on unitary groups
in the following way:

\begin{equation}
\mathrm{scl}\left(w\right)=\inf_{\ell>0;\, j_{1},\ldots,j_{\ell}>0}\frac{-\lim_{n\to\infty}\log_{n}\tr_{w^{j_{1}},\ldots,w^{j_{\ell}}}\left(n\right)}{2\left(j_{1}+\ldots+j_{\ell}\right)}.\label{eq:read scl}
\end{equation}

\end{cor}
A similar result is true for the stable commutator length of several
words. We explain how Corollary \ref{cor:can hear scl} follows from
Theorem \ref{thm:main - general} and Calegari's rationality theorem
in Section \ref{sub:Stable-commutator-length}.

\subsection{More related work and further motivation\label{sub:More-related-work}}

Our work is inspired by that of the second author and Parzanchevski
\cite{PP15}, where word measures on finite symmetric groups are considered.
An element of a free group $\F$ is called \textit{primitive} if it
belongs to some free generating set of $\F$. The following estimate
from \cite[Theorem 1.8]{PP15} is analogous to Theorem \ref{thm:main}:
\begin{thm}[Puder-Parzanchevski]
\label{thm:pp15} Let $S_{n}$ be the symmetric group on $n$ elements.
For $w\in\F_{r}$ given as in \eqref{eq:word-expression}, let $w$
be the word map 
\[
w:S_{n}^{r}\to S_{n},\quad w\left(\sigma_{1},\ldots,\sigma_{r}\right)\equiv\prod_{1\leq j\leq\left|w\right|}\sigma_{i_{j}}^{\e_{j}},
\]
just as in \eqref{eq:word}. Let $\sigma_{1}^{(n)},\ldots,\sigma_{r}^{(n)}$
be $r$ independent random permutations in $S_{n}$ taken with respect
to the uniform measure, viewed as $0$-$1$ $n\times n$ matrices.
Then 
\[
\E\left[\mathrm{tr}\left(w\left(\sigma_{1}^{(n)},\ldots,\sigma_{r}^{(n)}\right)\right)\right]=1+\frac{\left|\mathrm{Crit}\left(w\right)\right|}{n^{\pi\left(w\right)-1}}+O\left(\frac{1}{n^{\pi(w)}}\right),
\]
where $|\mathrm{Crit}(w)|$ and $\pi(w)$ are invariants of $w$.
The \emph{primitivity rank} $\pi(w)$ is the minimal rank of a subgroup
in 
\[
\{\: J\:\lvert\: w\in J\leq\F_{r}\text{ and \ensuremath{w}\ is \textbf{not} primitive in \ensuremath{J}}\:\}.
\]
$\mathrm{Crit}(w)$ is the set of subgroups attaining this minimum
rank. 
\end{thm}

The study leading to Theorem \ref{thm:pp15} had two main motivations,
both of which are also relevant to the main result of the current
paper. The first motivation is related to questions about word measures
on finite, or more generally compact, groups. As mentioned above,
the measure induced by $w\in\F_{r}$ on some compact group $G$ is
identical to the measure induced by $\theta\left(w\right)$ for any
$\theta\in\mathrm{Aut}\left(\F_{r}\right)$. In particular, since
the $x_{1}$-measure on $G$ (the measure induced by the single letter
word {}``$x_{1}$'') is the Haar measure, or simply the uniform
measure for finite groups, the same holds for the $w$-measure of
every word $w$ in the $\mathrm{Aut}\left(\F_{r}\right)$-orbit of
$x_{1}$. This orbit consists precisely of the \emph{primitive} words
in $\F_{r}$. Several mathematicians have asked whether primitive
words are the only words inducing the uniform (Haar) measure on every
finite (compact, respectively) group (see \cite{PP15} and the references
therein). Theorem \ref{thm:pp15} answered this question to the positive,
showing that every non-primitive word induces a non-uniform measure
on $S_{n}$ for $n$ large enough. However, many conjectures revolving
around word measures on groups remain open, and we see the current
paper as a step towards their resolution. More details are given in
Section \ref{sub:Word-measures}.

The second motivation for Theorem \ref{thm:pp15} lies in the field
of random graphs, and more precisely that of spectra of random graphs.
A strengthened version of the asymptotic formula in Theorem \ref{thm:pp15}
appears in \cite{PUDEREXP}, where it is used in an approach to \textit{Alon's
second eigenvalue conjecture} from \cite{ALON} that says 
\begin{quote}
`Almost all $d$-regular graphs are weakly Ramanujan.' 
\end{quote}
This conjecture was proved by Friedman in \cite{FRIEDMAN} and a new
proof has been given recently by Bordenave \cite{BORDENAVE}. While
an approach using asymptotics of word maps has not yet proved the
full strength of Alon's conjecture, the approach in \cite{PUDEREXP}
comes very close (up to a small additive constant) while keeping the
proof manageable. This approach has also given the best result to
date regarding a natural generalization of Alon's conjecture to families
of irregular graphs (see \cite{PUDEREXP}).

One can ask analogous questions about the spectrum of sums of Haar
distributed unitary matrices in the large $n$ limit. Consider, for
example, the sum 
\begin{equation}
\sum_{i=1}^{r}U_{i}^{(n)}+(U_{i}^{(n)})^{*}.\label{eq:unitarysum}
\end{equation}
The connection to word measures on $\U\left(n\right)$ is that the
$N$\textsuperscript{th} power of \eqref{eq:unitarysum} is equal
to the sum, over all not-necessarily-reduced words $w$ of length
$N$, of $w(U_{1}^{\left(n\right)},\ldots,U_{r}^{\left(n\right)})$. 

When one replaces unitaries in \eqref{eq:unitarysum} with random
permutation matrices, one gets the adjacency matrix of a graph sampled
from the \textit{permutation model} of random regular graphs. Hence
the analogy with spectral graph theory. Heuristically, questions about
the spectra of sums of unitary matrices should be much easier than
the corresponding questions about sums of 0-1 permutation matrices%
\footnote{We thank Peter Sarnak for an illuminating conversation about this
subject.%
}, owing to the random unitary matrices being denser, and thus having
more variables to average over.

Nevertheless, interesting analytic problems about random unitary matrices
remain. In \cite{HT2005} Haagerup and Thorbj{ø}rnsen proved that
a certain operator-theoretic semigroup $\mathrm{Ext}(\F_{r})$ is
not a group for $r\geq2$, which had been an open problem for about
25 years. Their approach uses an observation of Voiculescu from \cite{VOICAROUND}
that reduces the question to one about the existence of unitary representations
of $\F_{r}$ with certain spectral features%
\footnote{Voiculescu in \cite{VOICAROUND} also relates these questions to the
existence of Ramanujan graphs, pleasantly completing a circle of ideas.%
}. Building on the work of \cite{HT2005}, Collins and Male \cite{CM}
proved the \textit{strong asymptotic freeness} of Haar unitary matrices
from which they obtain: 
\begin{thm}[Collins-Male]
\label{thm:cm} Almost surely 
\[
\left\Vert \sum_{i=1}^{r}U_{i}^{(n)}+(U_{i}^{(n)})^{*}\right\Vert \xrightarrow{n\to\infty}2\sqrt{2r-1}.
\]

\end{thm}
We expect that our Theorems \ref{thm:main} and \ref{thm:main - general},
made suitably uniform in $w$ or $w_{1},\ldots,w_{\ell}$, should
give an alternative approach to bounds such as in Theorem \ref{thm:cm},
as well as to the related questions of strong asymptotic freeness
and properties of $\mathrm{Ext}(\F_{r})$. Going further with these
questions, one expects the following {}``folklore'' conjecture: 
\begin{quote}
`The largest eigenvalue of \eqref{eq:unitarysum} should be governed
by a suitably normalized Tracy-Widom law, in the limit $n\to\infty$.'
\end{quote}
\medskip{}

The set of solutions to \eqref{eq:commutators} along with its $\Aut_{\delta}(\F_{r})$-action
is interesting even considered apart from the connection with Random
Matrix Theory made in Theorem \ref{thm:main}. In fact, it is the
content of quite a few research papers. 

Algorithms to compute commutator lengths of words in free groups were
found independently by\label{Algorithms-to-compute-cl} \cite{edmunds1975endomorphism},
\cite{goldstein1979applications} and \cite{CULLER}. The latter work,
by Culler, is the most relevant to ours. His geometric approach to
$\cl\left(w\right)$ which we mentioned above (and see Proposition
\ref{prop:genus(w)} below), is further developed in the current paper
and stands in the core of our methods. Culler also introduces an algorithm
to obtain a representative of every equivalence class of solutions
to \eqref{eq:commutators}, namely of every orbit of $\mathrm{Aut}_{\delta}\left(\F_{2g}\right)\backslash\mathrm{Hom}_{w}\left(\F_{2g},\F_{r}\right)$
where $g=\cl\left(w\right)$. Although similar in spirit, our analysis
yields a clearer description of the set of classes of solutions and,
in particular, a more direct way to distinguish them from each other.
See Remark \ref{remark: algo-for-cl} and Section \ref{sec:More-Consequences}
for comparison between Culler's approach and ours.

In addition, Culler proves that for every $w\in\left[\F_{r},\F_{r}\right]$
there are only finitely many equivalence classes of solutions to \eqref{eq:commutators}.
This extends an older result regarding words $w$ with $\cl\left(w\right)=1$
\cite{hmelevskiui1971systems}, and we extend it further to equivalence
classes of admissible incompressible maps of $\wl$ -- see Corollary
\ref{cor:Finitely many incompressible maps}. We remark that some
researchers have looked at a larger group $\widehat{\Aut}_{\d}(\F_{2g})\supset\Aut_{\d}(\F_{2g})$
acting on the solution space to \eqref{eq:commutators}. In geometric
terms, one allows not only ordinary Dehn twists, but also {}``fractional''
ones -- see \cite{BF}. Bestvina and Feighn \cite{BF} study the problem
of counting the number of $\widehat{\Aut}_{\d}(\F_{2g})$-orbits of
solutions to \eqref{eq:commutators}. They prove that for all $g\geq1$
there is a word $w$ with $\cl\left(w\right)=g$ which has at least
$2^{g}$ distinct $\widehat{\Aut}_{\d}(\F_{2g})$-orbits of solutions
to \eqref{eq:commutators}. When $g=1$, this is a result of Lyndon
and Wicks \cite{LYNDONWICKS}. The motivation for \cite{BF} comes
for questions raised by Sela, who has introduced a very general framework
for studying the solutions to systems of equations such as \eqref{eq:commutators}
in free groups (e.g.~\cite{SELA1}).\medskip{}

Before giving an overview of our proofs in Section \ref{sub:Overview-of-proof}
below, we trace the history of the ideas of this paper. A \emph{ribbon
graph}, also called a \emph{fat graph}, is a graph where each vertex
comes with a cyclic ordering of its incident edges. Ribbon graphs
commonly serve as a combinatorial way to describe orientable surfaces
with boundary: every vertex is magnified to a disc, and every edge
widened to a strip. A standard reference is \cite[Section 1]{PENNER}.
With some extra information ribbon graphs appear as the {}``dessins
d'enfants'' of Grothendieck \cite{GROT}. The book of Lando and Zvonkin
\cite{LZ} gives an encyclopedic overview of subjects related to ribbon
graphs. 

There are two central themes in the current paper:
\begin{description}
\item [{A}] Certain integrals over random matrices can be computed by a
sum of terms encoded by {}``ribbon graphs''. Moreover, the order
of contribution of each term corresponds to the \emph{genus} or to
the \emph{degree sequence} of the corresponding ribbon graph\textit{.}
\item [{B}] Certain contributions from the sum in \textbf{A} coincide with
homotopy invariants of some topological spaces. 
\end{description}
One early synthesis of these ideas is the following seminal result
of Harer and Zagier \cite{HARERZAGIER}, independently discovered
by Penner \cite{PENNER}.
\begin{thm}[Harer-Zagier, Penner]
\label{thm:mcgzeta} Assume $g\ge1$. Let $\Sigma_{g}^{1}$ be the
closed genus $g$ surface with one point removed and let $\mathrm{MCG}\left(\Sigma_{g}^{1}\right)$
be the mapping class group of isotopy classes of orientation preserving
homeomorphisms $\Sigma_{g}^{1}\to\Sigma_{g}^{1}$. Then 
\begin{equation}
\chi\left(\mathrm{MCG\left(\Sigma_{g}^{1}\right)}\right)=\zeta\left(1-2g\right),
\end{equation}
where $\zeta$ is Riemann's zeta function. 
\end{thm}
Penner's approach in \cite{PENNER} clarifies our discussion so we
give a brief outline. Penner begins with the apriori unrelated%
\footnote{As Penner puts it, {}``It is also noteworthy that the technique of
perturbative series from particle physics so effectively captures
the combinatorics of the bundle over Teichmüller space {[}...{]}.''.%
} matrix integral 
\begin{equation}
P_{v_{3},\ldots,v_{K}}(n)=\frac{1}{\mu_{n}\prod_{j=1}^{K}v_{j}!}\int\prod_{j=1}^{K}\left(\frac{\mathrm{tr}H^{j}}{j}\right)^{v_{j}}\exp\left(\frac{-\mathrm{tr}H^{2}}{2}\right)dH,\label{eq:pennerint}
\end{equation}
where the integral is taken over the probability space of GUE $n\times n$
Hermitian matrices, $v_{k}$ are non-negative integers and $\mu_{n}$
is a normalization factor. He proves that $P_{v_{3},\ldots,v_{K}}$
is a polynomial in $n$ that can be expressed as a sum over ribbon
graphs with exactly $v_{j}$ vertices of degree $j$ for every $3\le j\le K$
(and no vertices of degree $1$, $2$ or larger than $K$).

The general idea of equating matrix integrals with sum of terms encoded
by diagrams goes back to the celebrated {}``Feynman diagrams'' of
\cite{FEYNMAN}, and the first encoding by ribbon graphs seems to
be due to by 't Hooft \cite{HOOFT}. In \cite{BIZ}, Bessis, Itzykson
and Zuber consider a matrix integral roughly similar to \eqref{eq:pennerint},
with an extra generating parameter $\lambda$, and show that in the
sum they obtain over ribbon graphs, the exponent of $\lambda$ in
every term coincides with the genus of the corresponding ribbon graph.\medskip{}

\textcolor{black}{As for Theme }\textbf{\textcolor{black}{B}}\textcolor{black}{,
the key topological object related to Theorem \ref{thm:mcgzeta} is
the }\textcolor{black}{\emph{fat graph complex }}\textcolor{black}{$\mathcal{\mathcal{G}}_{g}^{1}$
of Penner, defined in \cite[Page 41]{PENNER}}%
\footnote{\textcolor{black}{Penner defines arc complexes $\mathcal{G}_{g}^{s}$
for surfaces of genus $g$ with $s$ punctures, and everything we
say about Penner's work naturally extends to general $g$ and $s$.}%
}\textcolor{black}{. An equivalent definition, and one more clearly
related to our setting, is that $\mathcal{\mathcal{G}}_{g}^{1}$ is
a simplicial complex with one simplex of dimension $k$ for each isotopy
class of $k$ disjoint embedded arcs in $\Sigma_{g}^{1}$ with the
following properties. The arcs begin and end at the puncture, must
be pairwise non parallel, individually not homotopic into the puncture,
and must cut $\Sigma_{g}^{1}$ into discs. Each of these discs must
be bounded by at least 3 arcs. One simplex is a face of another if
it can be obtained by deleting some arcs. Thus $\mathcal{\mathcal{G}}_{g}^{1}$
carries the obvious action of the mapping class group by change of
markings.}

\textcolor{black}{This $\mathcal{\mathcal{G}}_{g}^{1}$ arises naturally
from the Teichm}ü\textcolor{black}{ller space of $\Sigma_{g}^{1}$
and furthermore inherits its homotopy type}%
\footnote{\textcolor{black}{Following \cite[Page 41]{PENNER}, $\mathcal{\mathcal{G}}_{g}^{1}$
is $\MCG$-equivariantly homotopy equivalent to a $\MCG$-invariant
spine of some decorated Teichm}ü\textcolor{black}{ller space. This
decorated version is homeomorphic to the Cartesian product of the
usual Teichm}ü\textcolor{black}{ller space and $\mathbf{R}_{+}$.}%
}\textcolor{black}{. By the well known work of Fenchel and Nielsen
\cite{FENCHELNIELSEN}, the Teichm}ü\textcolor{black}{ller space of
$\Sigma_{g}^{1}$ is contractible and thus so is $\mathcal{\mathcal{G}}_{g}^{1}$.
This is the fact that allows one to obtain an Euler characteristic
in Theorem \ref{thm:mcgzeta}. Indeed this Euler characteristic can
be obtained by counting $\MCG$-orbits of simplices of $\mathcal{\mathcal{G}}_{g}^{1}$,
and after translation to fat/ribbon graphs this is exactly what shows
up in the Feynman diagram expansion of Theme }\textbf{\textcolor{black}{A}}\textcolor{black}{. }

\textcolor{black}{A similar combinatorial model of the moduli space
of curves was given by Kontsevich in \cite[Theorem 2.2]{KONTSEVICH}
by means of Jenkins-Strebel quadratic differentials, and in Appendix
D of }\textcolor{black}{\emph{loc.~cit.~}}\textcolor{black}{Kontsevich
gives a short proof of Theorem \ref{thm:mcgzeta}. These results appear
in the context of the proof of a conjecture of Witten from \cite{WITTEN1}
asserting that two models of quantum gravity are equal.}

\subsection{\label{sub:Overview-of-proof}Overview of the proof and paper organization}

We now sketch the outline of the proofs of our main results.

\subsubsection*{Section \ref{sec:A-Rational-Expression}: Formula for $\tr$ using
pairs of matchings of letters }

In the first stage of our analysis, a crucial role is played by a
formula developed in \cite{xu1997random} and extended in \cite{collins2003moments}
and \cite{CS} in the aim of giving a new proof to the asymptotic
freeness of Haar Unitary matrices, namely, to \eqref{eq:firstorder}.
This is an integration formula for polynomials in the entries of a
Haar unitary matrix and their conjugates, appearing as Theorem \ref{thm:collins-sniady}
below. For example, it allows one to compute 
\begin{equation}
\int_{u\in\U\left(n\right)}u_{1,2}u_{3,4}\overline{u_{1,4}}\overline{u_{3,2}}d\mu_{n}.\label{eq:integration-example}
\end{equation}
This formula is parallel to a moment formula for Gaussian variables
that appears in the corresponding GUE analysis, a formula which usually
goes under the name {}``Wick formula''.

As shown in \cite{CS}, the evaluation of every such polynomial is
a rational function in $n$. For example, the integral in \eqref{eq:integration-example}
is equal to $\frac{-1}{n^{3}-n}$ for every $n\ge4$. A key feature
of this formula is that the leading term (exponent and coefficient)
has combinatorial significance, and is related to the Möbius function
of the poset (partially ordered set) of non-crossing partitions. 

In the current paper, we fully expand out the product\linebreak{}
$\mathrm{tr}(w_{1}(U_{1}^{(n)},\ldots,U_{r}^{(n)}))\cdots\mathrm{tr}(w_{\ell}(U_{1}^{(n)},\ldots,U_{r}^{(n)}))$
as a sum over indices of rows and columns of the matrices $U_{1}^{(n)},\ldots,U_{r}^{(n)}$,
and, using the integration formula mentioned above, show its expected
value $\tr_{w_{1},\ldots,w_{\ell}}\left(n\right)$ is indeed a rational
function in $n$, which can be computed explicitly. This is the content
of Theorem \ref{thm:trw-rational} below.

The formula we obtain for $\tr_{w_{1},\ldots,w_{\ell}}\left(n\right)$
can be viewed as a sum over pairs $\left(\sigma,\tau\right)$ of matchings
of the letters of $w_{1},\ldots,w_{\ell}$, where every letter $x_{i}^{\varepsilon}$
is matched with some $x_{i}^{-\varepsilon}$. Indeed, by Claim \ref{claim: tr=00003D0 for non-balanced words}
below, $\trwl\left(n\right)$ vanishes unless the total number of
instances of $x_{i}^{-1}$ in $w_{1},\ldots,w_{\ell}$ is equal to
the total number of $x_{i}^{+1}$, for every $i\in\left[r\right]$.
The latter holds if and only if $w_{1}w_{2}\cdots w_{\ell}\in\left[\F_{r},\F_{r}\right]$,
and we sometimes say that in this case $w_{1},\ldots,w_{\ell}$ form
a \textbf{}\marginpar{balanced set of words}\textbf{balanced} set
of words. The set of matchings associated with $\wl$ is denoted $\match\left(\wl\right)$
and is formally described in Definition \ref{def:match(w) and B(sigma,tau)}.

\subsubsection*{Section \ref{sec:surface-from-matchings}: Constructing surfaces
from pairs of matchings and Theorems \ref{thm:leading-exponent} and
\ref{thm:leading-exponent-general}}

In Section \ref{sec:surface-from-matchings} we explain how to associate
an orientable surface $\Sigma_{(\sigma,\tau)}$ with every pair of
matchings $(\sigma,\tau)\in\match\left(\wl\right)\times\match\left(\wl\right)$.
The surface $\Sigma_{\left(\sigma,\tau\right)}$, which is basically
given in the form of a ribbon graph, has $\ell$ boundary components
and its Euler characteristic is denoted $\chi(\sigma,\tau)$. This
extends a construction of Culler \cite{CULLER} that deals with a
single matching $\sigma\in\match\left(w\right)$ of a single word:
his construction is the special case $\sigma=\tau$ in ours. The extension
to pairs of matchings (and to multiple words) seems to be new here. 

It so happens that in the formula for $\trwl(n)$ given by a sum over
pairs of matchings, the contribution of every pair $(\sigma,\tau)$
is of order $n^{\chi\left(\sigma,\tau\right)}$ (Proposition \ref{prop:order-of-contribution-of-(sigma,tau)}).
Hence the contributions to $\trwl\left(n\right)$ of largest order
come from pairs $\left(\sigma,\tau\right)$ of largest Euler characteristic. 

In addition, we associate with the pair $\left(\sigma,\tau\right)$
a map (defined up to homotopy) 
\[
f_{\left(\sigma,\tau\right)}\colon\Sigma_{\left(\sigma,\tau\right)}\to\wedger,
\]
and show that $\left(\Sigma_{\left(\sigma,\tau\right)},f_{\left(\sigma,\tau\right)}\right)$
is admissible for $\wl$. Moreover, in Lemma \ref{lem:every admissible and incompressible obtained from matchings}
we explain that for every $\left(\Sigma,f\right)$ admissible for
$\wl$ with $f$ incompressible, there is a pair of matchings $\left(\sigma,\tau\right)$
such that $\left(\Sigma_{\left(\sigma,\tau\right)},f_{\left(\sigma,\tau\right)}\right)\sim\left(\Sigma,f\right)$.
This gives a procedure for computing $\ch\left(\wl\right)$ which,
again, generalizes a procedure suggested in \cite{CULLER} to compute
$\cl\left(w\right)$.

Since every admissible $\left(\Sigma,f\right)$ with $\chi\left(\Sigma\right)=\ch\left(\wl\right)$
is incompressible, we deduce in Corollary \ref{cor:trw(n) leading exponent}
the content of Theorems \ref{thm:leading-exponent-general} and \ref{thm:leading-exponent},
namely, that $\trwl\left(n\right)=O\left(n^{\ch\left(\wl\right)}\right)$.
This result roughly summarizes the role in the current work of Theme
\textbf{A} from above (although we have not used the fine details
of the ribbon graphs so far, only the Euler characteristics of the
underlying surfaces).

\subsubsection*{Section \ref{sec:The-pairs-of-matchings-Poset}: Poset of pairs of
matchings for incompressible maps}

Our next goal is to study the leading coefficient of $\trwl\left(n\right)$,
namely, the coefficient of $n^{\ch\left(\wl\right)}$. For this sake,
we gather all pairs of matchings $\left(\sigma,\tau\right)$ that
are associated with the same class $\left[\left(\Sigma,f\right)\right]$
of admissible surfaces and maps, and denote the set $\pmp\left(\Sigma,f\right)$:
\[
\pmp\left(\Sigma,f\right)\overset{\mathrm{def}}{=}\left\{ \left(\sigma,\tau\right)\in\match\left(\wl\right)^{2}\,\middle|\,\left(\Sigma_{\left(\sigma,\tau\right)},f_{\left(\sigma,\tau\right)}\right)\sim\left(\Sigma,f\right)\right\} .
\]
We show that whenever $f$ is incompressible, there is a natural partial
order on the set $\pmp\left(\Sigma,f\right)$, which turns it into
a poset we call the Pairs of Matchings Poset of $\left(\Sigma,f\right)$
(Definition \ref{def:perm-poset}). This partial order is closely
related to the aforementioned partial order on non-crossing partitions
(e.g.~Proposition \ref{prop:rewiring-in-PP}). 

The pairs of matchings poset is important mainly because of the role
of its associated simplicial complex. This finite complex, the simplices
of which corresponding to chains in $\pmp\left(\Sigma,f\right)$,
is denoted $\left|\pmp\left(\Sigma,f\right)\right|$ -- see Definition
\ref{def:SC-of-poset}. Theorem \ref{thm:Euler-char-of-poset} shows
that the contributions to $\trwl\left(n\right)$ of all pairs of matchings
in $\pmp\left(\Sigma,f\right)$ sum to
\begin{equation}
\chi\left(\left|\pmp\left(\Sigma,f\right)\right|\right)\cdot n^{\chi\left(\Sigma\right)}+O\left(n^{\chi\left(\Sigma\right)-2}\right).\label{eq:total contrib of PMP}
\end{equation}
We remark again that the two instances of $\chi\left(\right)$ in
\eqref{eq:total contrib of PMP} are applied to very different topological
objects: on the one hand an orientable compact surface $\Sigma$,
and on the other hand a simplicial complex obtained from a poset whose
elements are related to $\Sigma$.

\subsubsection*{$\left|\pmp\left(\Sigma,f\right)\right|$ as $K\left(G,1\right)$-space}

Finally, we prove that the simplicial complex $\left|\pmp\left(\Sigma,f\right)\right|$
is a $K\left(G,1\right)$-space for\linebreak{}
$G=\mathrm{Stab}_{\mathrm{MCG}\left(\Sigma\right)}\left(\tilde{f}\right)$
whenever $f$ is incompressible. This is the content of Theorem \ref{thm:pmp is K(G,1)}
below, and it obviously yields Theorems \ref{thm:K(G,1) for incompressible}
and \ref{thm:stabilizers K(G,1)}, and together with \eqref{eq:total contrib of PMP}
implies our main result: Theorem \ref{thm:main - general} and its
special case, Theorem \ref{thm:main}. 

Proving that $\left|\pmp\left(\Sigma,f\right)\right|$ is a $K\left(G,1\right)$-space
boils down to showing the following three facts (see Footnote \ref{fn:An-Eilenberg-MacLane-space}):\\
$\left(i\right)$~~~$\left|\pmp\left(\Sigma,f\right)\right|$ is
path-connected.\\
$\left(ii\right)$~~The fundamental group of $\left|\pmp\left(\Sigma,f\right)\right|$
is isomorphic to $\mathrm{Stab}_{\mathrm{MCG}\left(\Sigma\right)}\left(\tilde{f}\right)$.
And,\\
$\left(iii\right)$~The universal cover of $\left|\pmp\left(\Sigma,f\right)\right|$
is contractible. \\

Establishing this result requires the most involved part of this work
and the introduction of yet another poset: the arc poset of $\left(\Sigma,f\right)$.

\subsubsection*{Section \ref{sec:core-of-result}: The arc poset of $\left(\Sigma,f\right)$}

Let $\Sigma$ be a compact surface and $f\colon\Sigma\to\wedger$
incompressible so that $\left(\Sigma,f\right)$ is admissible for
$\wl$. The arc poset of $\left(\Sigma,f\right)$, denoted $\ap\left(\Sigma,f\right)$,
is an infinite poset composed of {}``arc systems''. An arc system
consists of $\left|w_{1}\right|+\ldots+\left|w_{\ell}\right|$ disjoint
arcs in $\Sigma$ (defined up to isotopy). The boundary components
of $\Sigma$ are marked in a way that {}``spells out'' $\wl$ (via
the functions $f_{w_{i}}\circ\partial_{i}^{-1}$ for $i\in\left[\ell\right]$),
and the arcs represent a pair of matchings in $\match\left(\wl\right)$.
Thus, every arc system in $\ap\left(\Sigma,f\right)$ is a specific
geometric realization of a pair $\left(\sigma,\tau\right)$ in $\pmp\left(\Sigma,f\right)$.
However, every pair $\left(\sigma,\tau\right)$ has (infinitely) many
different geometric realizations. We endow the set of arc systems
with a partial ordering, analogous to the one we defined on $\pmp\left(\Sigma,f\right)$.
This order is too related to the order on non-crossing partitions.
The construction of the arc poset $\ap\left(\Sigma,f\right)$ is detailed
in Definition \ref{def:arc-poset}.

A major part of this work is devoted to the analysis of the arc poset.
As in the case of $\pmp\left(\Sigma,f\right)$, we can associate a
simplicial complex to $\ap\left(\Sigma,f\right)$, which we denote
$\left|\ap\left(\Sigma,f\right)\right|$. It is clear that the mapping
class group $\mathrm{MCG\left(\Sigma\right)}$ acts on arc systems,
and we show it preserves the order we defined, so we obtain an action
on the poset $\ap\left(\Sigma,f\right)$ (part of Theorem \ref{thm:MCG-action-on-AP}).
To establish our results we show the following properties of $\ap\left(\Sigma,f\right)$
and of the action of $\mathrm{MCG}\left(\Sigma\right)$ on it:
\begin{enumerate}
\item Theorem \ref{thm:MCG-action-on-AP}: the infinite simplicial complex
$\left|\ap\left(\Sigma,f\right)\right|$ is a topological covering
space of $\left|\pmp\left(\Sigma,f\right)\right|$. Moreover, the
action $\MCG\left(\Sigma\right)\curvearrowright\ap\left(\Sigma,f\right)$
extends to a covering space action $\MCG\left(\Sigma\right)\curvearrowright\left|\ap\left(\Sigma,f\right)\right|$
and 
\[
\nicefrac{\left|\ap\left(\Sigma,f\right)\right|}{\MCG\left(\Sigma\right)}\cong\left|\pmp\left(\Sigma,f\right)\right|
\]
is an isomorphism of simplicial complexes. 
\item Theorem \ref{thm:components-of-AP(Sigma,f)} (first part): there is
a one-to-one correspondence between connected components in $\left|\ap\left(\Sigma,f\right)\right|$
and homotopy classes of functions in $\left[\left(\Sigma,f\right)\right]$.
\item Theorem \ref{thm:components-of-AP(Sigma,f)} (second part): every
connected component of $\left|\ap\left(\Sigma,f\right)\right|$ is
contractible.
\end{enumerate}
The first and last item show that every connected component of $\left|\ap\left(\Sigma,f\right)\right|$
is a universal covering space for $\left|\pmp\left(\Sigma,f\right)\right|$.
The second item then shows that the fundamental group of $\left|\pmp\left(\Sigma,f\right)\right|$
is isomorphic to $\mathrm{Stab}_{\mathrm{MCG}\left(\Sigma\right)}\left(\tilde{f}\right)$.

The proof of contractability of the connected components of $\ap\left(\Sigma,f\right)$,
the content of Theorem \ref{thm:components-of-AP(Sigma,f)}, requires
the most technical proof of this paper, and we devote to it Section
\ref{sub:Proof-of-contractability}. The proof consists of a series
of (countably many) deformation retracts which we define for each
component of $\left|\ap\left(\Sigma,f\right)\right|$. This eventually
shows that every component contracts to a point. Each step is described
by a poset morphism which, by the content of Appendix \ref{sub:Homotopy-of-poset},
corresponds to a deformation retract on the associated simplicial
complex. 
\begin{rem}
\textcolor{black}{There is another }\textcolor{black}{\emph{arc complex
}}\textcolor{black}{that is similar to Penner's fat graph complex
but with fewer constraints on the arcs: in particular, without the
constraint that the arcs cut the surface into discs. In \cite{HATCHER},
Hatcher extends earlier work of Harer \cite{HARER} to prove under
certain conditions that this arc complex is contractible by a direct
combinatorial argument, in contrast to the proof of the contractability
of the fat graph complex via Teichmüller theory. This direct argument,
while less involved than our argument, is similar in flavor. We also
point out that while at the level of objects our arc poset is related
to Hatcher's arc complex from \cite{HATCHER}, the topological claims
we make are quite different}\textcolor{black}{\emph{. }}\textcolor{black}{Indeed,
a $k$-simplex in $\left|\ap(\Sigma,f)\right|$ is a chain of arc
systems all with the same number of arcs, whereas in Hatcher's arc
complex a $k$-simplex is a series of arc systems which are obtained
by a series of arc deletions.}
\end{rem}

\subsubsection*{Remaining sections in the paper}

Except for the sections mentioned above, the rest of the paper is
organized as follows. In Section \ref{sec:Background} we give some
background for the ideas and tools in this paper: some basic facts
about commutator length of words, Culler's construction and the correspondence
between algebraic and geometric objects (Section \ref{sub:cl-of-word});
some comments and open questions regarding word measures on groups
(Section \ref{sub:Word-measures}), and some words about stable commutator
length and the proof of Corollary \ref{cor:can hear scl} (Section
\ref{sub:Stable-commutator-length}). 

After the core of the paper in Sections \ref{sec:A-Rational-Expression}-\ref{sec:core-of-result},
Section \ref{sec:More-Consequences} elaborates some further results
derived from our analysis, especially regarding properties of the
stabilizers from Theorems \ref{thm:main} and \ref{thm:main - general},
and Section \ref{sec:Examples} contains some detailed examples. These
are followed by some related open questions in Section \ref{sec:Some-Open-Problems}
and a glossary of notation. The appendix contains some technical,
mostly known, lemmas regarding posets and complexes. These are used
in the proofs of Theorems \ref{thm:MCG-action-on-AP} and \ref{thm:components-of-AP(Sigma,f)}.

\subsection{Notations }

For the convenience of the readers, there is a glossary on Page \pageref{sec:Glossary}
listing most of the notations we use and where each one is defined.
We also mention here some of the notation we will use. We use $\partial\Sigma$
to denote the boundary of the surface $\Sigma$. The word measures
are coming from words in $\F_{r}$, and we denote the generators by
$x_{1},\ldots,x_{r}$. However, in examples we sometimes use $x,y,z,t$
instead. We may use capital letters for inverses and occasionally
enumerate the letters by their location in $w$. For example, we may
write $w=\left[x,y\right]^{2}$ as $x_{1}y_{2}X_{3}Y_{4}x_{5}y_{6}X_{7}Y_{8}$.
We use $a_{1},b_{1},\ldots,a_{g},b_{g}$ and their capital forms to
write elements in the fundamental groups of surfaces.

Standard asymptotic notation is used to describe some of our results.
This includes the big $O$ notation {}``$f\left(n\right)=O\left(g\left(n\right)\right)$''
meaning that the functions $f$ and $g$ satisfy that for large enough
$n$, $f\left(n\right)\le C\cdot g\left(n\right)$ for some constant
$C>0$. Likewise, {}``$f\left(n\right)=o\left(g\left(n\right)\right)$''
means that for large enough $n$, $g\left(n\right)\ne0$ and that
$\frac{f\left(n\right)}{g\left(n\right)}\underset{n\to\infty}{\to}0$.
Finally, {}``$f\left(n\right)=\theta\left(g\left(n\right)\right)$''
means that for large enough $n$, $C_{1}\cdot g\left(n\right)\le f\left(n\right)\le C_{2}\cdot g\left(n\right)$
for some constants $C_{1},C_{2}>0$.

\section{Background\label{sec:Background}}

\subsection{The geometric approach to commutator length\label{sub:cl-of-word}}

In this subsection we explain Culler's geometric interpretation of
commutator length which yields that Theorem \ref{thm:leading-exponent}
is indeed a special case of Theorem \ref{thm:leading-exponent-general}.
We also explain the other parallels mentioned in Section \ref{sub:Expected-product-of}
and Table \ref{tab:Algebra-geometry-dictionary} between algebraic
notions and geometric ones. In particular, we formulate the Dehn-Nielsen-Baer
Theorem showing that $\mathrm{Aut}_{\delta}\left(\F_{2g}\right)$
is isomorphic to the mapping class group of $\surface$, the genus
$g$ one boundary component orientable surface. This yields that Theorem
\ref{thm:main} is a special case of Theorem \ref{thm:main - general}.\\

We begin with an easy but useful characterization of maps from surfaces
that are homotopic relative the boundary. 
\begin{lem}
\label{lem:homotopy=00003Dsame_induced_map}Let $\Sigma$ be any orientable
surface with $\ell$ boundary components and $\ell$ marked points
as in Definition \ref{def: admissible maps}, and let $\gamma_{1},\ldots,\gamma_{t}$
be a set of disjoint oriented arcs with endpoints in $v_{1},\ldots,v_{\ell}$
which {}``fill $\Sigma$'', i.e., which cut $\Sigma$ into discs.
Then two maps $f_{1},f_{2}\colon\Sigma\to\wedger$ which coincide
on $\partial\Sigma$ and send all marked points $v_{1},\ldots,v_{\ell}$
to the basepoint $o$ are homotopic relative the boundary if and only
if $\left[f_{1}\left(\gamma_{j}\right)\right]=\left[f_{2}\left(\gamma_{j}\right)\right]$
for all $j\in\left[t\right]$.
\end{lem}
This is also equivalent to $f_{1}$ and $f_{2}$ inducing the same
map from the {}``fundamental groupoid'' of $\Sigma$ as a space
with several marked points to $\pi_{1}\left(\wedger,o\right)$.
\begin{proof}
It $f_{1}$ and $f_{2}$ are homotopic and $\gamma$ is any oriented
arc from $v_{i}$ to $v_{j}$, one can push forward this homotopy
to show homotopy between $f_{1}\left(\gamma\right)$ and $f_{2}\left(\gamma\right)$,
hence $\left[f_{1}\left(\gamma\right)\right]=\left[f_{2}\left(\gamma\right)\right]$.
Conversely, assume that $f_{1}$ and $f_{2}$ satisfy the property
with the arcs. We can then perturb $f_{2}$ so that it agrees with
$f_{1}$ on these arcs (without changing the homotopy class of $f_{2}$).
Then, on every disc $D$, $f_{1}$ and $f_{2}$ agree on the boundary,
and it is enough to show that $f_{1}\Big|_{D}\simeq f_{2}\Big|_{D}$
are homotopic relative $\partial D$. Now $f_{1}$ and $f_{2}$ can
be lifted to maps $\hat{f_{1}},\hat{f_{2}}\colon D\to\mathbb{T}_{2r}$
which coincide on $\partial D$, where the $2r$-regular tree $\mathbb{T}_{2r}$
is the universal covering space of $\wedger$. It is easy to see that
$\hat{f_{1}}$ and $\hat{f_{2}}$ are homotopic: for every $x\in D$,
let the homotopy move in a constant pace from $\hat{f_{1}}\left(x\right)$
to $\hat{f_{2}}\left(x\right)$ in $\mathbb{T}_{2r}$ along the sole
geodesic between them. This homotopy can then be projected to a homotopy
between $f_{1}$ and $f_{2}$.
\end{proof}
To give a precise formulation of the geometric analogue for $\mathrm{Hom}_{w}\left(\F_{2g},\F_{r}\right)$,
we first fix some more notation. Identify each circle in the wedge
$\wedger$ with a distinct generator $x_{i}$ of $\F_{r}$, orient
each of the circles, and use these labeling and orientation to fix
an isomorphism
\begin{equation}
\F_{r}\cong\pi_{1}\left(\wedger,o\right).\label{eq:iso F_r and wedge}
\end{equation}
Recall that for every $w\in\F_{r}$, the map $f_{w}\colon\left(S^{1},1\right)\to\left(\wedger,o\right)$
is a fixed representative of $w$. More concretely, 
\begin{defn}
\label{def:f_w}For $1\ne w\in\F_{r}$, let $f_{w}\colon\left(S^{1},1\right)\to\left(\wedger,o\right)$
be the sole non-backtracking closed path at $o$ representing $w$
(moving at arbitrary positive speed), so that $f_{w^{-1}}\left(z\right)=f_{w}\left(\overline{z}\right)$.
For $w=1$ fix $f_{1}$ to be the constant map to $o$. 
\end{defn}
To fix the isomorphism of $\F_{2g}$ with $\pi_{1}\left(\surface,v_{1}\right)$,
we need to fix $2g$ disjoint oriented arcs $\alpha_{1},\beta_{1},\ldots,\alpha_{g},\beta_{g}$
in $\surface$ with endpoints in $v_{1}$ that serve as representatives
for the basis\linebreak{}
$a_{1},b_{1},\ldots,a_{g},b_{g}$ of $\F_{2g}$. We do this using
the construction of $\surface$ from a $\left(4g+1\right)$-gon, as
in Figure \ref{fig:pentagon-example}: we identify $\alpha_{1},\ldots,\beta_{g}$
with the sides of this $\left(4g+1\right)$-gon that are being glued.
Then there is an isomorphism
\begin{equation}
\F_{2g}\cong\pi_{1}\left(\surface,v_{1}\right)\label{eq:iso F_2g and Sigma_g,1}
\end{equation}
mapping $a_{i}$ to $\left[\alpha_{i}\right]$, $b_{i}$ to $\left[\beta_{i}\right]$,
and $\delta_{g}$ to $\left[\partial_{1}\right]$.

Recall from Section \ref{sec:Introduction} that the commutator length
of a word $w\in\left[\F_{r},\F_{r}\right]$, denoted $\mathrm{cl}\left(w\right)$,
is the smallest $g$ such that there exist $u_{1},v_{1},\ldots,u_{g},v_{g}\in\F_{r}$
with $\left[u_{1},v_{1}\right]\ldots\left[u_{g},v_{g}\right]=w.$
Equivalently, $\mathrm{cl}\left(w\right)$ is the smallest $g$ for
which 
\[
\mathrm{Hom}_{w}\left(\F_{2g},\F_{r}\right)=\left\{ \phi\in\mathrm{Hom}\left(\F_{2g},\F_{r}\right)\,\middle|\,\phi\left(\delta_{g}\right)=w\right\} 
\]
is non-empty. The following proposition, basically due to \cite{CULLER},
explains why $\cl\left(w\right)$ is often called {}``the genus of
$w$'', and why Theorem \ref{thm:leading-exponent} is a special
case of Theorem \ref{thm:leading-exponent-general}.
\begin{prop}
\label{prop:genus(w)}Let $w\in\left[\F_{r},\F_{r}\right]$ be a balanced
word and $g\ge0$ a non-negative integer. With the fixed isomorphisms
\eqref{eq:iso F_r and wedge} and \eqref{eq:iso F_2g and Sigma_g,1},
there is a one-to-one correspondence
\begin{equation}
\mathrm{Hom}_{w}\left(\F_{2g},\F_{r}\right)\,\,\,\,\longleftrightarrow\,\,\,\,\left\{ \begin{gathered}\mathrm{Homotopy\,\, classes\,\,(relative\,}\,\partial\surface\mathrm{)\,\, of\,\, maps}\,\, f\colon\surface\to\wedger\\
\mathrm{such\,\, that}\,\,\left(\surface,f\right)\,\,\mathrm{is\,\, admissible\,\, for}\,\, w
\end{gathered}
\right\} .\label{eq:Hom_w correspondence with homotopy classes}
\end{equation}
In particular, $\cl\left(w\right)$ is equal to the smallest genus
$g$ of a surface $\surface$ with an admissible map for $w$.
\end{prop}
We note there are correspondences of the same spirit for maps admissible
for several words.\\

\begin{proof}
It is clear that if $\left(\surface,f\right)$ is admissible for $w$,
then $f_{*}\in\mathrm{Hom}_{w}\left(\F_{2g},\F_{r}\right)$, and $f_{*}$
only depends on the homotopy class of $f$. Conversely, given $\phi\in\mathrm{Hom}_{w}\left(\F_{2g},\F_{r}\right)$,
define $f:\surface\to\wedger$ as following. First, define $f\Big|_{\partial\surface}$
so that $f\circ\partial_{1}=f_{w}$. For every $i\in\left[g\right]$
define $f\Big|_{\alpha_{i}}$ so that $f\circ\alpha_{i}=f_{\phi\left(a_{i}\right)}$
and $f\Big|_{\beta_{i}}$ so that $f\circ\beta_{i}=f_{\phi\left(b_{i}\right)}$.
The arcs $\alpha_{1},\ldots,\beta_{g}$ cut $\surface$ to a single
polygon $P$, identical to the $\left(4g+1\right)$-gon used to construct
$\surface$ - see Figure \ref{fig:pentagon-example}. It therefore
remains to define $f$ on the interior of $P$.\FigBesBeg \\
\begin{figure}[h]
\centering{}\includegraphics[bb=0bp 280bp 400bp 600bp,scale=0.5]{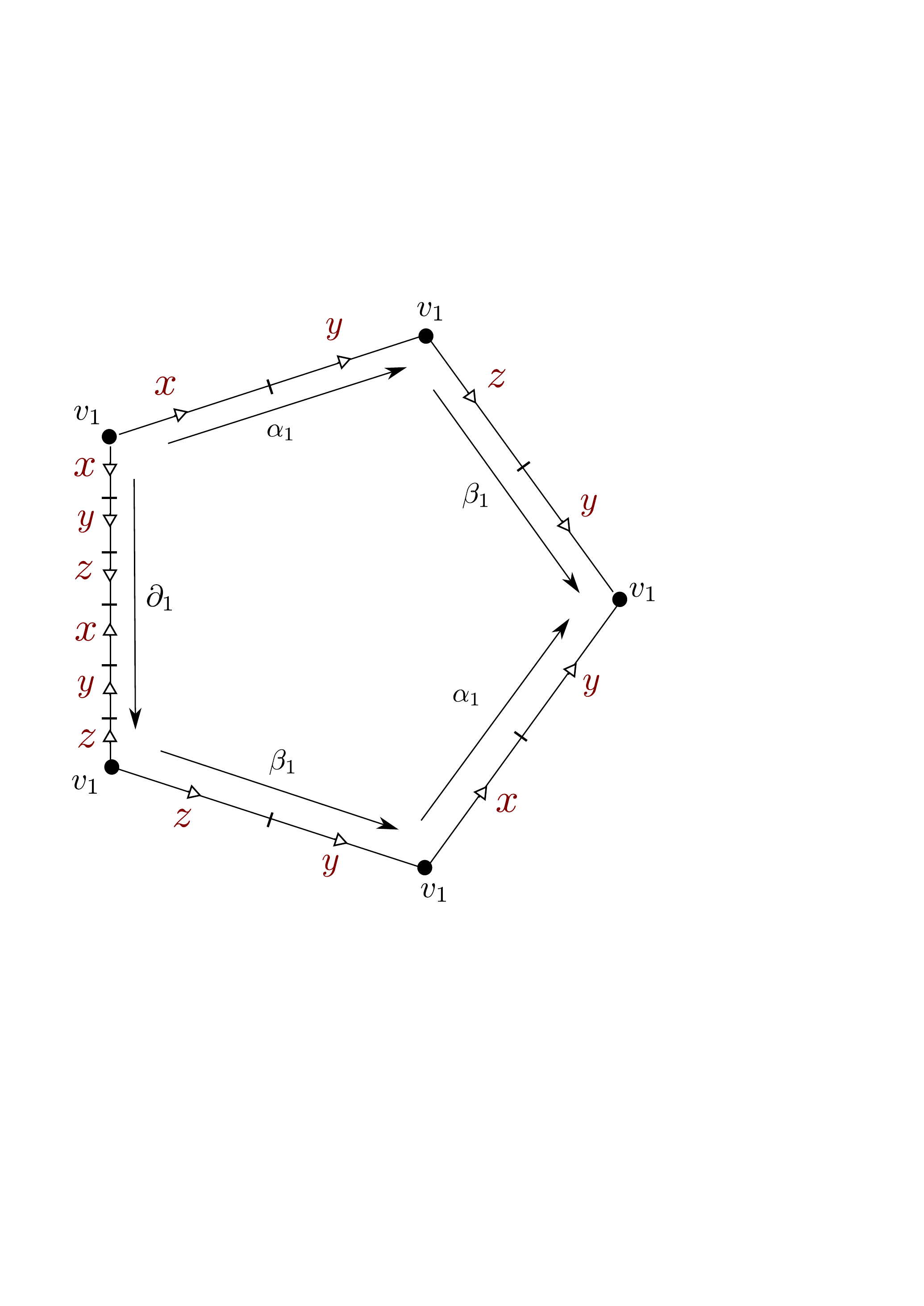}\caption{\label{fig:pentagon-example} The word $w=xyzXYZ$ has commutator
length $1$ as shown by the solution $w=\left[xy,zy\right]$. To construct
a corresponding map $f$ from $\Sigma_{1,1}$, we first define $f$
on the boundary of a $5$-gon $P$ described in the definition of
the arcs $\alpha_{1},\beta_{1},\ldots,\alpha_{g},\beta_{g}$. The
letters $x,y,z$ describe the image of $f\Big|_{\partial P}$ in $\wedger$.}
\end{figure}
\FigBesEnd \\
By the assumption on $\phi$, the boundary $\partial P$ is mapped
by $f$ to the trivial element of $\pi_{1}\left(\wedger,o\right)$.
So there is a homotopy $T\colon\partial P\times\left[0,1\right]\to\wedger$
such that $T\left(x,0\right)\equiv f\Big|_{\partial P}$ and $T\left(x,1\right)$
is constantly $o$. This map induces, therefore, a continuous map
$\overline{T}\colon\nicefrac{\partial P\times\left[0,1\right]}{\left(x,1\right)\sim\left(y,1\right)}\to\wedger$.
Since $\nicefrac{\partial P\times\left[0,1\right]}{\left(x,1\right)\sim\left(y,1\right)}$
is homeomorphic to $P$ in a way that identifies $\left(x,0\right)$
with $x$, we can use $\overline{T}$ to get the required map $f$
on all of $P$. Lemma \ref{lem:homotopy=00003Dsame_induced_map} shows
that the homotopy class of $f$ is well defined. It is also clear
that $f_{*}=\phi$. 
\end{proof}
Let us also mention a few facts about commutator length in free groups.
As mentioned in Section \ref{sub:More-related-work}, there are several
algorithms for computing the commutator length of a given word $w\in\left[\F_{r},\F_{r}\right]$.
One of this algorithms, due to Culler, follows from our discussion
in Section \ref{sec:surface-from-matchings} below --- see Remark
\ref{remark: algo-for-cl}. We also remark that the values taken by
$\cl$ on $\left[\F_{r},\F_{r}\right]$ ($r\ge2$) are all positive
integers. An illuminating example is given in \cite[Section 2.6]{CULLER}:
\[
\cl\left(\left[x,y\right]^{n}\right)=\left\lfloor \frac{n}{2}\right\rfloor +1.
\]
For instance, $\left[x,y\right]^{3}=\left[xyX,YxyX^{2}\right]\left[Yxy,y^{2}\right]$.
Moreover, in the same paper Culler shows that for every $1\ne w\in\left[\F_{r},\F_{r}\right]$,
$\cl\left(w^{n}\right)\underset{n\to\infty}{\to}\infty$. A tight
lower bound $\cl\left(w^{n}\right)>\frac{n}{2}$ is given in \cite[Theorem 4.111]{calegari2009scl}.\\

Finally, let us explain the last two lines of Table \ref{tab:Algebra-geometry-dictionary},
showing that Theorems \ref{thm:main} and \ref{thm:stabilizers K(G,1)}
are special cases of Theorems \ref{thm:main - general} and \ref{thm:K(G,1) for incompressible},
respectively. Recall that $\mathrm{Aut}_{\delta}\left(\F_{2g}\right)$
is the subgroup of $\mathrm{Aut}\left(\F_{2g}\right)$ fixing $\delta_{g}=\left[a_{1},b_{1}\right]\ldots\left[a_{g},b_{g}\right]$.
Via the isomorphism \eqref{eq:iso F_2g and Sigma_g,1}, we can view
$\mathrm{Aut}_{\delta}\left(\F_{2g}\right)$ as the group of automorphisms
of $\pi_{1}\left(\surface,v_{1}\right)$ fixing the element $\left[\partial_{1}\right]$.
\begin{thm}[Dehn-Nielsen-Baer]
\label{thm:Dehn-Nielsen-Baer} The map $\theta\colon\MCG\left(\surface\right)\to\mathrm{Aut}_{\delta}\left(\F_{2g}\right)$
defined by
\[
\left[\rho\right]\mapsto\rho_{*}
\]
is an isomorphism.
\end{thm}
A reference for the Dehn-Nielsen-Baer Theorem, including some historical
notes, can be found in \cite[Chapter 8]{FM}. However, the version
that appears in \cite{FM} and usually found in the literature is
slightly different and deals either with surfaces without boundary
or with homeomorphisms of surfaces with boundary that do not necessarily
fix the boundary. As we could not find any published reference for
the exact version we need here, let us say a few words about the proof
of Theorem \ref{thm:Dehn-Nielsen-Baer}.

That $\theta$ is a well-defined homomorphism of groups is trivial.
The surjectivity of $\theta$ is a special case of \cite[Theorem 5.7.1]{ZVC}.
Finally, the injectivity of $\theta$ follows from the fact that $\surface$
is a $\mathrm{K}\left(\F_{2g},1\right)$-complex: indeed, $\surface$
is a $\mathrm{K}\left(\F_{2g},1\right)$-space (for example, because
it deformation-retracts to a bouquet with $2g$ loops), which can
be given a CW-complex structure. A basic feature of every $\mathrm{K}\left(G,1\right)$-complex
$Y$ is that any homomorphism $\pi_{1}\left(Y,y_{0}\right)\to\pi_{1}\left(Y,y_{o}\right)$
is induced by some map $\left(Y,y_{0}\right)\to\left(Y,y_{0}\right)$,
which is unique up to homotopy fixing $y_{0}$ (e.g.~\cite[Theorem 1B.9]{hatcher2002algebraic}).
Since on surfaces homotopy of homeomorphisms is the same as isotopy
(\cite[Theorem 1.12]{FM}), we see that $\theta^{-1}\left(\mathrm{id}\right)$
is precisely $\mathrm{Homeo}_{0}\left(\surface\right)$. 

Another remark worth mentioning is that the group $\mathrm{Aut}_{\delta}\left(\F_{2g}\right)$
is torsion-free (e.g.,~\cite[Corollary 7.3]{FM}), and thus so are
the stabilizer subgroups in Theorem \ref{thm:stabilizers K(G,1)}.
This shows that a finite $\mathrm{K\left(G,1\right)}$-complex is
plausible.

Finally, note that if $\left[\rho\right]\in\mathrm{MCG}\left(\surface\right)$
and $f\colon\left(\surface,v_{1}\right)\to\left(\wedger,o\right)$
is admissible for $w$, then the action of $\left[\rho\right]$ on
$\tilde{f}$, the homotopy class of $f$, is given by $\widetilde{\rho\circ f}$.
On the other hand, the action of $\rho_{*}\in\mathrm{Aut}_{\delta}\left(\F_{2g}\right)$
on $f_{*}\in\mathrm{Hom}_{w}\left(\F_{2g},\F_{r}\right)$ is given
by $\rho_{*}\circ f_{*}=\left(\rho\circ f\right)_{*}$. This shows
that the action of $\mathrm{MCG}\left(\surface\right)$ on the homotopy
classes in \eqref{eq:Hom_w correspondence with homotopy classes}
is isomorphic to the action of $\mathrm{Aut}_{\delta}\left(\F_{2g}\right)$
on $\mathrm{Hom}_{w}\left(\F_{2g},\F_{r}\right)$.

\subsection{Word measures on compact groups\label{sub:Word-measures}}

Let $G$ be a compact group. As explained in Section \ref{sub:The-expected-trace},
every word $w\in\F_{r}$ induces a measure on $G$, which we call
the $w$-measure and denote in this subsection $\mu_{G}^{w}$\marginpar{$\mu_{G}^{w}$}.
This is the measure obtained by pushing forward the Haar measure on
$G^{r}=\underbrace{G\times\ldots\times G}_{r\,\mathrm{times}}$ through
the word map $w\colon G^{r}\to G$. Namely, to sample an element from
the $w$-measure on $G$, simply sample $r$ independent elements
$g_{1},\ldots,g_{r}$ according to the Haar measure on $G$, and evaluate
$w\left(g_{1},\ldots,g_{r}\right)$. An important special case is
when $G$ is finite and then the Haar measure is simply the uniform
distribution.

The following invariance of word measure motivates the theme that
$w$-measures on groups encode algebraic information about $w$:
\begin{fact}
\label{prop:Aut-F_r invariance of word measures}Word measures are
invariant under $\mathrm{Aut}\left(\F_{r}\right)$. Namely, if $w\in\F_{r}$
and $\phi\in\mathrm{Aut}\left(\F_{r}\right)$, then $w$ and $\phi\left(w\right)$
induce the same measure on every compact group.\end{fact}
\begin{proof}
Recall we denote the generators of $\F_{r}$ by $x_{1},\ldots,x_{r}$.
The automorphism group $\mathrm{Aut}\left(\F_{r}\right)$ is generated
by the following {}``elementary Nielsen transformations'' defined
on the generators (e.g.~\cite[Section I.4]{LyndonSchupp77}):
\begin{itemize}
\item The automorphism $\alpha_{\sigma}$ defined by a permutation $\sigma\in S_{r}$
on the generators
\item The automorphism $\beta$ defined by $x_{1}\mapsto x_{1}x_{2}$ and
$x_{i}\mapsto x_{i}$ for $i\ge2$
\item The automorphism $\gamma$ defined by $x_{1}\mapsto x_{1}^{-1}$ and
$x_{i}\mapsto x_{i}$ for $i\ge2$
\end{itemize}

Thus it is enough to show the word measures of a compact group $G$
are invariant under these transformations. This is obvious for the
automorphisms $\alpha_{\sigma}$. For $\beta$, it is enough to show
that if $g_{1},g_{2},\ldots,g_{r}\in G$ are $r$ independent Haar
random elements, then so are $g_{1}g_{2},g_{2},\ldots,g_{r}$. This
is true by right-invariance of the Haar measure on compact groups:
sample $g_{2}$ first. When sampling $g_{1}$, the measure on $g_{1}g_{2}$
is again the Haar measure. It also shows that $g_{1}g_{2}$ is independent
of $g_{2}$. As for automorphism $\gamma$, given $g_{1},\ldots,g_{r}$
as before, the independence of $g_{1}^{-1},g_{2},\ldots,g_{r}$ is
obvious. The transformation $g\mapsto g^{-1}$ turns a left Haar measure
into a right one, but these two are the same in compact groups.

\end{proof}
So two words in the same $\mathrm{Aut}\left(\F_{r}\right)$-orbit
in $\F_{r}$ induce the same measure on every compact group. But is
this the only reason for two words to have such a strong connection?
A version of the following conjecture appears, for example, in \cite[Question 2.2]{AV}
and in \cite[Conjecture 4.2]{Shalev2013}.
\begin{conjecture}
\label{conj:aner}If two words $w_{1},w_{2}\in\F_{r}$ induce the
same measure on every compact group, then there exists $\phi\in\mathrm{Aut}\left(\F_{r}\right)$
with $w_{2}=\phi\left(w_{1}\right)$.
\end{conjecture}
A special case of this conjecture, which attracted the attention of
several researchers, deals with the $\mathrm{Aut}\left(\F_{r}\right)$-orbit
of the single-letter word $x_{1}$, namely, with the set of primitive
words. It was asked whether words inducing the Haar measure on every
compact group are necessarily primitive. As mentioned in Section \ref{sub:More-related-work},
this was settled in \cite[Theorem 1.1]{PP15} using word measures
on symmetric groups:
\begin{thm}[Puder-Parzanchevski]
 A word inducing uniform measure on every finite group is necessarily
primitive.
\end{thm}
Still, even in this special case, open problems remain: for example,
can the symmetric groups be replaced in this result by, say, solvable
groups? or compact Lie groups? Is there a single compact Lie group
which suffices? We see our work here as a step towards answering these
questions and, especially, Conjecture \ref{conj:aner}.

To the very least, we hope to be able to show that only primitive
words induce the Haar measure on ${\cal U}\left(n\right)$ for every
$n$. To date, we can use the current work to show that two words
$w_{1}$ and $w_{2}$ with $\mathrm{scl}\left(w_{1}\right)\ne\mathrm{scl}\left(w_{2}\right)$
induce different measures on $\U\left(n\right)$ for every large enough
$n$ -- see Section \ref{sub:Stable-commutator-length}.

\medskip{}

The first result in our paper deals with $\trw\left(n\right)$, the
expected trace of a random matrix in $\U\left(n\right)$ sampled by
the $w$-measure. Let us explain why this particular projection of
the $w$-measure $\mu_{G}^{w}$ is a very natural first step.
\begin{fact}
For any compact group $G$, the word measure $\mu_{G}^{w}$ is determined
by the expected values of the irreducible characters $\left\{ \intop_{g\in G}\xi\left(g\right)d\mu_{G}^{w}\left(g\right)\right\} _{\xi\in\widehat{G}}$.
\end{fact}
Here $\widehat{G}$ marks the set of all irreducible characters of
$G$. 
\begin{proof}
The statement of the proposition holds for every conjugation-invariant
measure. First we show why $\mu_{G}^{w}$ has this property, and then
why this property yields the statement of the proposition. We ought
to show that for every $g\in G$ and every measurable set $A\subseteq G$,
we have $\mu_{G}^{w}\left(A\right)=\mu_{G}^{w}\left(gAg^{-1}\right)$.
This follows from the invariance of Haar measures under conjugation
and the equality 
\[
w^{-1}\left(gAg^{-1}\right)=g\left(w^{-1}\left(A\right)\right)g^{-1},
\]
the conjugation on the right hand side being the diagonal conjugation
on $G^{r}$.

To see that a conjugation-invariant measure $\mu$ on a compact group
$G$ is completely determined by the expectation of irreducible characters%
\footnote{For finite groups, this follows by viewing the measure as a function
and the fact that the irreducible characters form a basis for class
functions.%
}, consider any $\mu$-measurable function $f\colon G\to\mathbb{C}$
with finite expectation. Then, by conjugation-invariance, for every
$h\in G$, 
\[
\intop_{G}f\left(g\right)d\mu\left(g\right)=\intop_{G}f\left(hgh^{-1}\right)d\mu\left(g\right).
\]
Thus,

\[
\intop_{g\in G}f\left(g\right)d\mu\left(g\right)=\intop_{h\in G}\left[\intop_{g\in G}f\left(hgh^{-1}\right)d\mu\left(g\right)\right]d\mu\left(h\right)=\intop_{g\in G}\left[\intop_{h\in G}f\left(hgh^{-1}\right)d\mu\left(h\right)\right]d\mu\left(g\right),
\]
where we used Fubini's theorem. Defining the class function $\overline{f}\left(g\right)=\intop_{h\in G}f\left(hgh^{-1}\right)d\mu\left(h\right)$,
we obtain, as $\overline{f}=\sum_{\xi\in\hat{G}}\left\langle \overline{f},\xi\right\rangle \xi$,
that 
\[
\intop_{g\in G}f\left(g\right)d\mu\left(g\right)=\intop_{g\in G}\overline{f}\left(g\right)d\mu\left(g\right)=\sum_{\xi\in\hat{G}}\left\langle \overline{f},\xi\right\rangle \cdot\intop_{g\in G}\xi\left(g\right)d\mu\left(g\right).
\]

\end{proof}
Thus it makes sense to study word measures via the expectation of
irreducible characters. In this language, for example, Conjecture
\ref{conj:aner} says that if $w_{1}$ and $w_{2}$ do not belong
to the same $\mathrm{Aut}\left(\F_{r}\right)$-orbit, then there is
some compact group $G$ and some non-trivial character $1\ne\xi\in\widehat{G}$
so that $\xi$ has different expectations under $\mu_{G}^{w_{1}}$
and $\mu_{G}^{w_{2}}$. In the case of ${\cal U}\left(n\right)$,
it is fair to say the simplest irreducible character is the trace
of the standard representation, and its expected value under the $w$-measure
$\mu_{{\cal U}\left(n\right)}^{w}$ is, by definition, $\trw\left(n\right)$.
\begin{rem}
As hinted in Section \ref{sub:Expected-product-of}, our more general
results regarding $\trwl$ and finite sets of words give much more
information about word measures in ${\cal U}\left(n\right)$. In particular,
they give similar kind of control we get over $\trw\left(n\right)$
for many other irreducible characters of ${\cal U}\left(n\right)$.
For example, consider the irreducible character of ${\cal U}\left(n\right)$
which corresponds to the highest weight vector $\left(2,0,\ldots,0,-1\right)$.
It is given by 
\[
\xi\left(A\right)=\frac{\mathrm{tr}\left(A^{2}\right)+\mathrm{tr}^{2}\left(A\right)}{2}\cdot\mathrm{tr}\left(A^{-1}\right)-\mathrm{tr}\left(A\right).
\]
So the expected value of $\xi$ in the measure $\mu_{{\cal U}\left(n\right)}^{w}$
is 
\[
\mathbb{E}_{\mu_{{\cal U}\left(n\right)}^{w}}\left[\xi\right]=\frac{1}{2}\tr_{w^{2},w^{-1}}\left(n\right)+\frac{1}{2}\tr_{w,w,w^{-1}}\left(n\right)-\trw\left(n\right),
\]
and Theorem \ref{thm:main - general} gives information about the
leading term of this expression. The same is true for any {}``non-balanced''
irreducible character: a character the corresponding highest weight
vector of which sums to zero, or equivalently, a character which is
not invariant under multiplication by central elements of ${\cal U}\left(n\right)$.
In contrast, the character corresponding to $\left(1,0,\ldots,0,-1\right)$,
which is given by 
\[
\left|\mathrm{tr}\left(A\right)\right|^{2}-1,
\]
is balanced, and its expected value under $\mu_{{\cal U}\left(n\right)}^{w}$
is 
\begin{equation}
\tr_{w,w^{-1}}\left(n\right)-1.\label{eq:formula for balanced character}
\end{equation}
Because of the free term {}``$-1$'' in \eqref{eq:formula for balanced character},
Theorem \ref{thm:main - general} gives weaker information about the
leading coefficient of \eqref{eq:formula for balanced character},
and only determines the limit of the character as $n\to\infty$, rather
than its leading term.
\end{rem}
Finally, let us remark that many works in the area of word measures
focus on questions of slightly different flavor: the word measures
induced by a fixed word across all finite/compact groups; the support
of word measures; the probability, in word measures on finite groups,
of the identity, etc. A survey containing many references is \cite{Shalev2013}.

\subsection{Stable commutator length\label{sub:Stable-commutator-length}}

Recall that Corollary \ref{cor:can hear scl} states that the $w$-measures
on $\left\{ \U\left(n\right)\right\} _{n\in\mathbb{N}}$ determine
$\mathrm{scl}\left(w\right)$, the stable commutator length of $w\in\F_{r}$
(see \eqref{eq:scl}). In this subsection we explain how this result
follows from Theorem \ref{thm:main - general} and from Calegari's
rationality theorem. 

Calegari's theorem, which is the main result of \cite{CALRATIONAL},
says that $\mathrm{scl}\left(w\right)$ is rational for every $w\in\left[\F_{r},\F_{r}\right]$.
The proof goes through showing the existence of {}``extremal surfaces''
for $w$: an extremal surface for $w$ is an admissible $\left(\Sigma,f\right)$
for some set of powers of $w$, say $w^{j_{1}},\ldots,w^{j_{\ell}}$
with $j_{1},\ldots,j_{\ell}\in\mathbb{Z}$, so that $\frac{-\chi\left(\Sigma\right)}{2\left(j_{1}+\ldots+j_{\ell}\right)}$
achieves the infimum of the values of its kind. This infimum is $\mathrm{scl}\left(w\right)$,
the stable commutator length of $w$ \cite[Lemma 2.6]{CALRATIONAL}. 

The main theorem of \cite{CALRATIONAL} states that if $w\in\left[\F_{r},\F_{r}\right]$
then $w$ admits an extremal surface $\left(\Sigma,f\right)$. Moreover,
by \cite[Lemma 2.7]{CALRATIONAL}, this extremal surface can be taken
to be admissible for $w^{j_{1}},\ldots,w^{j_{\ell}}$ with $j_{1},\ldots,j_{\ell}>0$.
By definition of extremal surface, $\Sigma$ has maximal Euler characteristic
for $w^{j_{1}},\ldots,w^{j_{\ell}}$, namely, $\chi\left(\Sigma\right)=\ch\left(w^{j_{1}},\ldots,w^{j_{\ell}}\right)$.
Moreover, every surface which is admissible for $w^{j_{1}},\ldots,w^{j_{\ell}}$
with Euler characteristic $\ch\left(w^{j_{1}},\ldots,w^{j_{\ell}}\right)$
is extremal. By \cite[Lemma 2.9]{CALRATIONAL}, the maps associated
with extremal surfaces are $\pi_{1}$-injective, namely, if $\gamma$
is a non-nullhomotopic closed curve, then $f\left(\gamma\right)$
is not nullhomotopic. Note that this condition is stronger than incompressibility,
which only deals with \emph{simple} closed curves. The crux of the
matter is the following lemma, a special case of which is discussed
in Remark \ref{remark: free solutions}:
\begin{lem}
\label{lem:pi1-injective means trivial stab}If $\left(\Sigma,f\right)$
is $\pi_{1}$-injective, then $\mathrm{Stab}_{\mathrm{MCG}\left(\Sigma\right)}\left(\tilde{f}\right)$
is trivial. \end{lem}
\begin{proof}
Let $\left[\rho\right]\in\mathrm{MCG}\left(\Sigma\right)$ fix $\tilde{f}$,
so $f\circ\rho\simeq f$ are homotopic. Let $\gamma_{2},\ldots,\gamma_{\ell}$
be a set of $\ell-1$ disjoint arcs in $\Sigma$, where $\gamma_{i}$
leads from $v_{1}$ to $v_{i}$. The arc $\rho\left(\gamma_{2}\right)$
is homotopic to the concatenation $\beta*\gamma_{2}$ where $\beta$
is a closed loop at $v_{1}$, but 
\[
\left[f\left(\gamma_{2}\right)\right]=\left[f\left(\rho\left(\gamma_{2}\right)\right)\right]=\left[f\left(\beta*\gamma_{2}\right)\right]=\left[f\left(\beta\right)\right]\cdot\left[f\left(\gamma_{2}\right)\right]
\]
and so $\left[f\left(\beta\right)\right]=1$ and by $\pi_{1}$-injectivity,
$\left[\beta\right]=1\in\pi_{1}\left(\Sigma,v_{1}\right)$. Hence
$\rho\left(\gamma_{2}\right)\simeq\gamma_{2}$ and we may perturb
$\rho$ so that it fixes $\gamma_{2}$. We can do the same for $\gamma_{3}$
without modifying $\rho|_{\gamma_{2}}$ and so on, until $\rho$ fixes
$\gamma_{2}\cup\ldots\cup\gamma_{\ell}$ pointwise. Now we can cut
$\Sigma$ along $\gamma_{2},\ldots,\gamma_{\ell}$ and get a surface
$\Sigma'$ with one boundary component, a map $f'\colon\Sigma'\to\wedger$
and an induced homeomorphism $\rho'$ which fixes $\partial\Sigma'$
pointwise and such that $f'\circ\rho'\simeq\rho'$. By Theorem \ref{thm:Dehn-Nielsen-Baer},
$\left[\rho'\right]$ corresponds to some $\phi\in\mathrm{Aut}_{\delta}\left(\pi_{1}\left(\Sigma',v_{1}\right)\right)$.
As $f'$ is still $\pi_{1}$-injective, we see that $\left(f'\right)_{*}$
cannot be fixed by any non-trivial element of $\mathrm{Aut}\left(\pi_{1}\left(\Sigma',v_{1}\right)\right)$,
let alone of $\mathrm{Aut}_{\delta}\left(\pi_{1}\left(\Sigma',v_{1}\right)\right)$,
hence $\phi=1$ and $\left[\rho'\right]=\left[\mathrm{id}\right]$.
Thus $\left[\rho\right]=1$.
\end{proof}
We infer that if one of the extremal surfaces of $w$ is admissible
for $w^{j_{1}},\ldots,w^{j_{\ell}}$ with $j_{1},\ldots,j_{\ell}>0$,
then Theorem \ref{thm:main - general} translates in this case to
\begin{equation}
\tr_{w^{j_{1}},\ldots,w^{j_{\ell}}}\left(n\right)=n^{\ch\left(w^{j_{1}},\ldots,w^{j_{\ell}}\right)}\cdot\left|\sol\left(w^{j_{1}},\ldots,w^{j_{\ell}}\right)\right|+O\left(n^{\ch\left(w^{j_{1}},\ldots,w^{j_{\ell}}\right)-2}\right),\label{eq:trwl for extremal}
\end{equation}
which is strictly positive for large enough $n$. Hence,
\[
\frac{-\lim_{n\to\infty}\log_{n}\tr_{w^{j_{1}},\ldots,w^{j_{\ell}}}\left(n\right)}{2\left(j_{1}+\ldots+j_{\ell}\right)}=\frac{-\ch\left(w^{j_{1}},\ldots,w^{j_{\ell}}\right)}{2\left(j_{1}+\ldots+j_{\ell}\right)}=\mathrm{scl}\left(w\right).
\]
On the other hand, for an arbitrary $\ell>0$ and $j_{1},\ldots,j_{\ell}>0$
we have 
\[
\frac{-\lim_{n\to\infty}\log_{n}\tr_{w^{j_{1}},\ldots,w^{j_{\ell}}}\left(n\right)}{2\left(j_{1}+\ldots+j_{\ell}\right)}\ge\frac{-\ch\left(w^{j_{1}},\ldots,w^{j_{\ell}}\right)}{2\left(j_{1}+\ldots+j_{\ell}\right)}\ge\mathrm{scl}\left(w\right).
\]
This proves \eqref{eq:read scl} and Corollary \ref{cor:can hear scl}.
$\qed$

\medskip{}

\begin{cor}
If $\mathrm{scl}\left(w_{1}\right)\ne\mathrm{scl}\left(w_{2}\right)$
then for every large enough $n$, the $w_{1}$-measure on ${\cal U}\left(n\right)$
is different from the $w_{2}$-measure on ${\cal U}\left(n\right)$.
In particular, if $w_{1}\in\left[\F_{r},\F_{r}\right]$ and $w_{2}\notin\left[\F_{r},\F_{r}\right]$
then they induce different measures on ${\cal U}\left(n\right)$ for
almost all $n$.\end{cor}
\begin{proof}
Assume without loss of generality that $\mathrm{scl}\left(w_{1}\right)<\mathrm{scl}\left(w_{2}\right)$,
and let $j_{1},\ldots,j_{\ell}>0$ be so that $w_{1}^{j_{1}},\ldots,w_{1}^{j_{\ell}}$
admit an extremal surface. Then by the above discussion, $\tr_{w_{1}^{j_{1}},\ldots,w_{1}^{j_{\ell}}}\left(n\right)$
is strictly larger than $\tr_{w_{2}^{j_{1}},\ldots,w_{2}^{j_{\ell}}}\left(n\right)$
for any large enough $n$. In particular, if $w_{2}$ is not balanced,
i.e.~$w_{2}\notin\left[\F_{r},\F_{r}\right]$ and $\mathrm{scl}\left(w_{2}\right)=\infty$,
then nor is the set $w_{2}^{j_{1}},\ldots,w_{2}^{j_{\ell}}$ balanced
as we assume $j_{1},\ldots,j_{\ell}>0$. By Claim \ref{claim: tr=00003D0 for non-balanced words},
$\tr_{w_{2}^{j_{1}},\ldots,w_{2}^{j_{\ell}}}\left(n\right)\equiv0$
for every $n$. 
\end{proof}

\section{A Rational Expression for $\trwl\left(n\right)$\label{sec:A-Rational-Expression}}

In this section we prove that $\trwl\left(n\right)$ is a rational
function in $n$ (Theorem \ref{thm:trw-rational}). First, we prove
the observation mentioned above regarding non-balanced sets of words:
\begin{claim}
\label{claim: tr=00003D0 for non-balanced words}If $w_{1}w_{2}\cdots w_{\ell}\in\F_{r}\setminus\left[\F_{r},\F_{r}\right]$
then $\trwl\left(n\right)\equiv0$.\end{claim}
\begin{proof}
By the assumption, there is some $j\in\left[r\right]$ so that $\alpha_{j}$,
the sum of exponents of the letter $x_{j}$ in $\wl$, satisfies $\alpha_{j}\ne0$.
Recall that the Haar measure of a compact group is invariant under
left multiplication by any element. Since for $\theta\in\left[0,2\pi\right]$,
the diagonal central matrix $e^{i\theta}I_{n}$ is in ${\cal U}\left(n\right)$,
we obtain
\begin{eqnarray*}
 &  & \trwl\left(n\right)=\\
 &  & =\mathbb{E}_{{\cal U}\left(n\right)^{r}}\left[tr\left(w_{1}\left(U_{1}^{\left(n\right)},\ldots,U_{j}^{\left(n\right)},\ldots,U_{r}^{\left(n\right)}\right)\right)\cdots tr\left(w_{\ell}\left(U_{1}^{\left(n\right)},\ldots,U_{j}^{\left(n\right)},\ldots,U_{r}^{\left(n\right)}\right)\right)\right]\\
 &  & =\mathbb{E}_{{\cal U}\left(n\right)^{r}}\left[tr\left(w_{1}\left(U_{1}^{\left(n\right)},\ldots,e^{i\theta}U_{j}^{\left(n\right)},\ldots,U_{r}^{\left(n\right)}\right)\right)\cdots tr\left(w_{\ell}\left(U_{1}^{\left(n\right)},\ldots,e^{i\theta}U_{j}^{\left(n\right)},\ldots,U_{r}^{\left(n\right)}\right)\right)\right]\\
 &  & =e^{i\theta\alpha_{j}}\cdot\trwl\left(n\right).
\end{eqnarray*}
The claim follows as this equality holds for every $\theta\in\left[0,2\pi\right]$.
\end{proof}

\subsection{Weingarten function and integrals over ${\cal U}\left(n\right)$\label{sub:Weingarten-function-and-Collins-Sniady}}

The main tool used in this section is a formula, basically due to
Xu \cite{xu1997random} and, more neatly, to Collins and \'{S}niady
\cite{CS}, which expresses integrals with respect to $\left({\cal U}\left(n\right),\mu_{n}\right)$.
These integrals are expressed in terms of the \emph{Weingarten} function,
first studied in \cite{weingarten1978asymptotic} and formally defined
and named in \cite{collins2003moments}. Let $\mathbb{Q}\left(n\right)$
denote the field of rational functions with rational coefficients
in the variable $n$. Let $S_{L}$\marginpar{$S_{L}$} denote the
symmetric group on $L$ elements. The Weingarten function maps%
\footnote{More precisely, it is a function from the disjoint union $\bigcup_{L=1}^{\infty}S_{L}$
to $\mathbb{Q}\left(n\right)$.%
} $S_{L}$ to $\mathbb{Q}\left(n\right)$ (for every $L$). We think
of such functions as elements of the group ring $\mathbb{Q}\left(n\right)\left[S_{L}\right]$.
\begin{defn}
\label{def:weingarten}The \textbf{Weingarten function} \marginpar{$\wg$}$\wg:S_{L}\to\mathbb{Q}\left(n\right)$
is the inverse, in the group ring $\mathbb{Q}\left(n\right)\left[S_{L}\right]$,
of the function $\sigma\mapsto n^{\#\mathrm{cycles\left(\sigma\right)}}$. 
\end{defn}
That the function $\sigma\mapsto n^{\#\mathrm{cycles\left(\sigma\right)}}$
is invertible for every $L$ follows from \cite[Proposition 2.3]{CS}
and the discussion following it. Clearly, $\wg$ is constant on conjugacy
classes. For example, for $L=2$, the inverse of $\left(n^{2}\cdot\left(1\right)\left(2\right)+n\cdot\left(12\right)\right)\in\mathbb{Q}\left(n\right)\left[S_{2}\right]$
is $\left(\frac{1}{n^{2}-1}\cdot\left(1\right)\left(2\right)-\frac{1}{n\left(n^{2}-1\right)}\cdot\left(12\right)\right)$,
so $\wg\left(\left(1\right)\left(2\right)\right)=\frac{1}{n^{2}-1}$
while $\wg\left(\left(12\right)\right)=\frac{-1}{n\left(n^{2}-1\right)}$.
For $L=3$ the values of the Weingarten function are
\begin{eqnarray*}
 &  & \wg\left(\left(1\right)\left(2\right)\left(3\right)\right)=\frac{n^{2}-2}{n\left(n^{2}-1\right)\left(n^{2}-4\right)}\,\,\,\,\,\,\wg\left(\left(12\right)\left(3\right)\right)=\frac{-1}{\left(n^{2}-1\right)\left(n^{2}-4\right)}\\
 &  & \wg\left(\left(123\right)\right)=\frac{2}{n\left(n^{2}-1\right)\left(n^{2}-4\right)}.
\end{eqnarray*}
(We use here a non-standard cycle notation for permutations where
we write fixed points as well. This is to stress the dependency of
$\wg\left(\sigma\right)$, for $\sigma\in S_{L}$, on $L$. E.g.,
$\mathrm{Wg}\left(\left(12\right)\right)\ne\mathrm{Wg}\left(\left(12\right)\left(3\right)\right)$.)

Collins and \'{S}niady also provide an explicit formula for $\wg$
in terms of the irreducible characters of $S_{L}$ and Schur polynomials
\cite[Equation (13)]{CS}: for $\sigma\in S_{L}$,
\[
\wg\left(\sigma\right)=\frac{1}{\left(L!\right)^{2}}\sum_{\lambda\vdash L}\frac{\chi_{\lambda}\left(e\right)^{2}}{d_{\lambda}\left(n\right)}\chi_{\lambda}\left(\sigma\right),
\]
where $\lambda$ runs over all partitions of $ $$L$, $\chi_{\lambda}$
is the character of $S_{L}$ corresponding to $\lambda$, and $d_{\lambda}\left(n\right)$
is the number of semistandard Young tableaux with shape $\lambda$,
filled with numbers from $\left[n\right]$. A well known formula for
$d_{\lambda}\left(n\right)$ states $d_{\lambda}\left(n\right)=\frac{\chi_{\lambda}\left(e\right)}{L!}\prod_{\left(i,j\right)\in\lambda}\left(n+j-i\right)$,
where $\left(i,j\right)$ are the coordinates of cells in the Young
diagram with shape $\lambda$ (e.g.~\cite[Section 4.3, Equation (9)]{fulton1997young}).
Thus,
\begin{cor}
\label{cor:poles-of-Wg}For $\sigma\in S_{L}$, $\wg\left(\sigma\right)$
may have poles only at integers $n$ with $-L<n<L$.
\end{cor}
The key feature of $\wg$ that we need is the value of its leading
term. This is expressed in terms of a certain Möbius function which
we now define. For every permutation $\sigma\in S_{L}$ denote by
\label{norm-of-perm}$\left\Vert \sigma\right\Vert $\marginpar{$\left\Vert \sigma\right\Vert $}
its norm, defined as the length of the shortest product of transpositions
giving $\sigma$. Equivalently, $\left\Vert \sigma\right\Vert =L-\#\mathrm{cycles\left(\sigma\right)}$.
This norm can be used to define a poset structure on $S_{L}$: say
that \marginpar{$\sigma\preceq\tau$}$\sigma\preceq\tau$ if and only
if $\left\Vert \tau\right\Vert =\left\Vert \sigma\right\Vert +\left\Vert \sigma^{-1}\tau\right\Vert $.
That is, $\sigma\preceq\tau$ if and only if there is a product of
transpositions of minimal length giving $\tau$, such that some prefix
of this product is equal to $\sigma$. This poset is closely related
to that of non-crossing partitions --- see \cite[Lecture 23]{NICASPEICHER}.

Every locally finite poset%
\footnote{\label{fn:locally-finite-posets}A poset $\left(P,\le\right)$ is
said to be locally finite if for every $x\le y$ in $P$, the interval
$\left[x,y\right]=\left\{ z\,\middle|\, x\le z\le y\right\} $ is
finite.%
} gives rise to a Möbius function defined on comparable pairs of elements.
This is defined to be the only function $\mu:\left\{ \left(x,y\right)\,\middle|\, x\preceq y\right\} \to\mathbb{\mathbb{Z}}$
that satisfies 
\begin{equation}
\sum_{z:x\le z\le y}\mu\left(x,z\right)=\delta_{x,y}\label{eq:mobius-definition}
\end{equation}
for every $x,y$ in the poset with $x\preceq y$ (see \cite[Section 3.7]{Stanley-book}).

In the case of the poset $\left(S_{L},\preceq\right)$, the corresponding
Möbius function has a nice combinatorial description:
\begin{prop}
\label{prop:mobius function}\cite[Section 2.3]{CS} The Möbius function
of the poset $\left(S_{L},\preceq\right)$ is given by $\mu\left(\sigma,\tau\right)=\moeb\left(\sigma^{-1}\tau\right)$\marginpar{$\moeb\left(\sigma\right)$},
where 
\begin{equation}
\moeb\left(\sigma\right)=\mathrm{sgn}\left(\sigma\right)\prod_{i=1}^{k}c_{|C_{i}|-1},\label{eq:mobius}
\end{equation}
with $C_{1},\ldots,C_{k}$ the cycles composing $\sigma$, and 
\[
c_{m}=\frac{(2m)!}{m!(m+1)!}
\]
the $m$-th Catalan number. 
\end{prop}
The content of Proposition \ref{prop:mobius function} is that if
$\sigma\preceq\tau$ in $S_{L}$, then 
\[
\sum_{\pi\in S_{L}\,\mathrm{s.t.\:\sigma\preceq\pi\preceq\tau}}\moeb\left(\sigma^{-1}\pi\right)=\delta_{\sigma,\tau}.
\]

\begin{prop}
\label{prop:wg-mobius}\cite[Corollary 2.7]{CS} Let $\sigma\in S_{L}$.
The Weingarten function satisfies 
\[
\wg\left(\sigma\right)=\frac{\moeb\left(\sigma\right)}{n^{L+\left\Vert \sigma\right\Vert }}+O\left(\frac{1}{n^{L+\left\Vert \sigma\right\Vert +2}}\right).
\]

\end{prop}
Note the jump of $2$ in the exponent after the subtraction of the
leading term. In fact, this is shown to go on: in the Taylor expansion
of $\wg\left(\sigma\right)$ in $\frac{1}{n}$, every other term vanishes
\cite[Proposition 2.6]{CS}.

The formula of Collins and \'{S}niady evaluates integrals of monomials
in the entries $u_{i,j}$ and their conjugates $\overline{u_{i,j}}$
of a Haar distributed unitary matrix $u\in{\cal U}\left(n\right)$.
The simple argument in the proof of Claim \ref{claim: tr=00003D0 for non-balanced words}
shows that such an integral vanishes whenever the monomial is not
balanced, namely whenever the number of $u_{i,j}$'s is different
from the number of $\overline{u_{i,j}}$'s. The following formula
deals with the interesting case, where the monomial is balanced:
\begin{thm}
\label{thm:collins-sniady}\cite[Proposition 2.5]{CS} Let $m$ and
$n_{0}$ be positive integers and $\left(i_{1},\ldots,i_{m}\right)$,
$\left(j_{1},\ldots,j_{m}\right)$, $\left(i'_{1},\ldots,i'_{m}\right)$
and $\left(j'_{1},\ldots,j'_{m}\right)$ be $m$-tuples of indices
in $\left[n_{0}\right]$. Then
\[
\int_{{\cal U}(n)}u_{i_{1}.j_{1}}u_{i_{2},j_{2}}\ldots u_{i_{m},j_{m}}\overline{u_{i'_{1},j'_{1}}}\overline{u_{i'_{2},j'_{2}}}\ldots\overline{u_{i'_{m},j'_{m}}}d\mu_{n}
\]
is a rational function in $n$ (valid for $n\ge n_{0}$), which is
equal to 
\begin{equation}
\sum_{\sigma,\tau\in S_{m}}\delta_{i_{1}i'_{\sigma(1)}}\ldots\delta_{i_{m}i'_{\sigma(m)}}\delta_{j_{1}j'_{\tau(1)}}\ldots\delta_{j_{m}j'_{\tau(m)}}\wg\left(\sigma^{-1}\tau\right).\label{eq:moment}
\end{equation}

\end{thm}
Put differently, the rational function is given by $\sum_{\sigma,\tau}\wg\left(\sigma^{-1}\tau\right)$,
where $\sigma$ runs over all rearrangements of $\left(i'_{1},\ldots,i'_{m}\right)$
which make it identical to $\left(i_{1},\ldots,i_{m}\right)$, and
$\tau$ runs over all rearrangements of $\left(j'_{1},\ldots,j'_{m}\right)$
which make it identical to $\left(j_{1},\ldots,j_{m}\right)$. In
particular, the possible poles of the Weingarten function at $n$,
for every $n\ge n_{0}$, are guaranteed to cancel out in this summation
(see the example following Proposition 2.5 in \cite{CS}). We mention
that a result of the type of Theorem \ref{thm:collins-sniady}, where
integrals over $\U\left(n\right)$ are expressed as combinatorial
formulas involving permutations, is possible thanks to the Schur-Weyl
duality.

\subsection{Word integrals over ${\cal U}\left(n\right)$}

We use \eqref{eq:moment} to analyze 
\[
\trwl\left(n\right)=\int_{{\cal U}(n)\times{\cal U}\left(n\right)\times\ldots\times{\cal U}\left(n\right)}\mathrm{tr}\left(w_{1}\left(U_{1}^{\left(n\right)},\ldots,U_{r}^{\left(n\right)}\right)\right)\cdots\mathrm{tr}\left(w_{\ell}\left(U_{1}^{\left(n\right)},\ldots,U_{r}^{\left(n\right)}\right)\right)d\mu_{n}^{\,\, r}.
\]
We explain our approach by way of an example. Let $w=\left[x,y\right]^{2}=xyXYxyXY\in\F_{2}$.
Then, 
\begin{eqnarray}
\trw\left(n\right) & = & \intop_{\left(A,B\right)\in{\cal U}(n)\times{\cal U}\left(n\right)}\mathrm{tr}\left(ABA^{-1}B^{-1}ABA^{-1}B^{-1}\right)d\mu_{n}^{\,\,2}\nonumber \\
 & = & \intop_{\left(A,B\right)\in{\cal U}(n)\times{\cal U}\left(n\right)}\sum_{i,j,k,\ell,I,J,K,L\in\left[n\right]}A_{i,j}B_{j,k}A_{\,\, k,\ell}^{-1}B_{\,\,\ell,I}^{-1}A_{I,J}B_{J,K}A_{\,\, K,L}^{-1}B_{\,\, L,i}^{-1}d\mu_{n}^{\,\,2}\label{eq:two-indices-for-every-letter}\\
 & = & \sum_{i,j,k,\ell,I,J,K,L\in\left[n\right]}\intop_{\left(A,B\right)\in{\cal U}(n)\times{\cal U}\left(n\right)}A_{i,j}B_{j,k}\overline{A_{\ell,k}}\overline{B_{I,\ell}}A_{I,J}B_{J,K}\overline{A_{L,K}}\overline{B_{i,L}}d\mu_{n}^{\,\,2}\nonumber \\
 & = & \sum_{i,j,k,\ell,I,J,K,L\in\left[n\right]}\left[\int_{A\in{\cal U}(n)}A_{i,j}A_{I,J}\overline{A_{\ell,k}}\overline{A_{L,K}}d\mu_{n}\right]\cdot\left[\int_{B\in{\cal U}(n)}B_{j,k}B_{J,K}\overline{B_{I,\ell}}\overline{B_{i,L}}d\mu_{n}\right].\nonumber 
\end{eqnarray}
Now we use Theorem \ref{thm:collins-sniady} to replace each of the
two integrals inside the sum by a summation over pairs of permutations
in $S_{2}$. For the first integral we go over all bijections $\sigma_{a}\colon\left\{ i,I\right\} \overset{\sim}{\to}\left\{ \ell,L\right\} $
and $\tau_{a}\colon\left\{ j,J\right\} \overset{\sim}{\to}\left\{ k,K\right\} $,
and similarly over bijections $\sigma_{b}$ and $\tau_{b}$ for the
second integral. We think of these sets as ordered, so $\sigma_{a}=\left(12\right)$
means it maps $i\mapsto L$, $I\mapsto\ell$. We change the order
of summation, and sum first over $\sigma_{a}$, $\tau_{a}$, $\sigma_{b}$
and $\tau_{b}$, and only then over the indices $i,j,\ldots,L$. In
fact, for every set of permutations, we only need to count the number
of evaluations of $i,j,\ldots,L$ which {}``agree'' with the permutations.
For example, consider the case where 
\[
\begin{gathered}\underline{\sigma_{a}=\id}\\
i\mapsto\ell\\
I\mapsto L
\end{gathered}
\,\,\,\,\begin{gathered}\underline{\tau_{a}=\left(12\right)}\\
j\mapsto K\\
J\mapsto k
\end{gathered}
\,\,\,\,\begin{gathered}\underline{\sigma_{b}=\left(12\right)}\\
j\mapsto i\\
J\mapsto I
\end{gathered}
\,\,\,\,\begin{gathered}\underline{\tau_{b}=\left(12\right)}\\
k\mapsto L\\
K\mapsto\ell
\end{gathered}
.
\]
The summand corresponding to these permutations is
\[
\wg\left(\left(12\right)\right)\cdot\wg\left(\left(1\right)\left(2\right)\right)\cdot\sum_{i,j,k,\ell,I,J,K,L\in\left[n\right]}\delta_{i\ell}\delta_{IL}\delta_{jK}\delta_{Jk}\delta_{ji}\delta_{JI}\delta_{kL}\delta_{K\ell},
\]
and the product inside the last sum is 1 (and not 0) if and only if
$i=\ell=K=j$ and $I=L=k=J$. So there are exactly $n^{2}$ such sets
of indices and the total contribution of these particular 4 permutations
is 
\[
\wg\left(\left(12\right)\right)\cdot\wg\left(\left(1\right)\left(2\right)\right)\cdot n^{2}=\frac{-1}{n\left(n^{2}-1\right)}\cdot\frac{1}{n^{2}-1}\cdot n^{2}=\frac{-n}{\left(n^{2}-1\right)^{2}}.
\]
If we perform the same calculation for all 16 possible sets of permutations
and sum the contributions, we obtain that
\begin{equation}
\tr_{\left[x,y\right]^{2}}\left(n\right)=\frac{-4}{n^{3}-n}.\label{eq:rational-expression-for-[a,b]^2}
\end{equation}

Of course, similar analysis works for any word $w\in\left[\F_{r},\F_{r}\right]$
and any (balanced) finite sets $\wl\in\F_{r}$. As $\left\{ \wl\right\} $
is balanced, the total length of the words is even, and we denote
it by $2L\overset{\mathrm{def}}{=}\left|w_{1}\right|+\ldots+\left|w_{\ell}\right|$\marginpar{$L,L_{i}$}.
Let $L_{i}$ denote the number of appearances of $x_{i}$ in $\wl$
(appearances with positive exponent $+1$), so $\sum_{i=1}^{r}L_{i}=L$.
Let $\mathrm{BIJ}_{i}\left(\wl\right)$\marginpar{$\mathrm{BIJ}_{i}{\scriptstyle \left(\wl\right)}$}
denote the set of bijections from the appearances of $x_{i}^{+1}$
to those of $x_{i}^{-1}$, so $\left|\mathrm{BIJ}_{i}\left(\wl\right)\right|=L_{i}!$.
To compute $\trwl\left(n\right)$, we go over all $\left(2r\right)$-tuples
of bijections $\left(\sigma_{1},\tau_{1},\ldots,\sigma_{r},\tau_{r}\right)$,
with $\sigma_{i},\tau_{i}\in\mathrm{BIJ}_{i}\left(\wl\right)$. Note
that $\sigma_{i}^{-1}\tau_{i}$ can be thought of as a permutation
of the appearances of $x_{i}^{+1}$, and so $\sigma_{i}^{-1}\tau_{i}$
belongs to a well-defined conjugacy class in $S_{L_{i}}$.

As in the example, each tuple induces a partition on a set of $\left|w_{1}\right|+\ldots+\left|w_{\ell}\right|=2L$
indices, and we denote the number of blocks in this partition by $B\left(\sigma_{1},\tau_{1},\ldots,\sigma_{r},\tau_{r}\right)$\marginpar{${\scriptstyle B\left(\sigma_{1},\ldots,\tau_{r}\right)}$}.
The number of evaluations of the indices which agree with these bijections
is $n^{B\left(\sigma_{1},\ldots,\tau_{r}\right)}$. Hence,
\begin{thm}
\label{thm:trw-rational}In the notations of the previous paragraph,
for every $n\ge\max_{i}L_{i}$, 
\begin{equation}
\trwl\left(n\right)=\sum_{\sigma_{1},\tau_{1}\in\mathrm{BIJ}_{1}\left(\wl\right),\,\ldots\,,\sigma_{r},\tau_{r}\in\mathrm{BIJ}_{r}\left(\wl\right)}\wg\left(\sigma_{1}^{-1}\tau_{1}\right)\ldots\wg\left(\sigma_{r}^{-1}\tau_{r}\right)\cdot n^{B\left(\sigma_{1},\tau_{1},\ldots,\sigma_{r},\tau_{r}\right)}.\label{eq:first_rational_function-for-tr(w)}
\end{equation}
In particular, for $n\ge\max_{i}L_{i}$, $\trwl\left(n\right)$ is
given by a rational function in $n$.
\end{thm}
We have to restrict to $n\ge\max_{i}L_{i}$ because of possible poles
of the Weingarten function%
\footnote{Interestingly, very similar constraints on $n$ appear in a formula
for the trace of $w$ in $r$ uniform \emph{permutation }matrices
--- see \cite[Section 5]{PUDER14}.%
} (Corollary \ref{cor:poles-of-Wg}). When this function has no poles,
Theorem \ref{thm:collins-sniady} guarantees that the expression we
get gives the right answer.

\section{Constructing Surfaces from Pairs of Matchings\label{sec:surface-from-matchings}}

In this section we associate a surface for every $2r$-tuple of bijections
$\left(\sigma_{1},\ldots,\tau_{r}\right)$ appearing in Theorem \ref{thm:trw-rational}.
This allows a better understanding of the summation \eqref{eq:first_rational_function-for-tr(w)}
and the order of its terms, and leads to Theorem \ref{thm:leading-exponent}
about the leading exponent of $\trwl\left(n\right)$ (Corollary \ref{cor:trw(n) leading exponent}
below). Together with a suitable map, the surface we construct will
be admissible for $\wl$. As shown in Proposition \ref{prop:order-of-contribution-of-(sigma,tau)}
below, the order of the contribution of a $2r$-tuple of bijections
in \eqref{eq:first_rational_function-for-tr(w)} is given by the Euler
characteristic of its associated surface.

Notation-wise, instead of keeping track of $2r$ different bijections,
it is more convenient to regard them as a pair of matchings between
the letter with positive exponent in $\wl$ and the letters with negative
exponents. To formalize this, let $L$ and $L_{i}$ $\left(i\in\left[r\right]\right)$
be as in Section \ref{sec:A-Rational-Expression} above (we keep restricting
to the interesting case where $\wl$ is a balanced set). Let \marginpar{$E_{i}^{+},E_{i}^{-}$}$E_{i}^{+}$
be the set of appearances of $x_{i}^{+1}$ and $E_{i}^{-}$ be the
set of appearances of $x_{i}^{-1}$, so that $\left|E_{i}^{+}\right|=\left|E_{i}^{-}\right|=L_{i}$.
We also let $E^{+}=\bigcup_{i}E_{i}^{+}$ \marginpar{$E^{+},E^{-}$}and
$E^{-}=\bigcup_{i}E_{i}^{-}$, so $\left|E^{+}\right|=\left|E^{-}\right|=L$.
We then consider the set $\{\sigma_{i}:E_{i}^{+}\overset{\sim}{\to}E_{i}^{-}\}_{i\in\left[r\right]}$
encoded in a single bijection $\sigma:E^{+}\overset{\sim}{\to}E^{-}$.
Likewise, we encode $\{\tau_{i}:E_{i}^{+}\overset{\sim}{\to}E_{i}^{-}\}_{i\in\left[r\right]}$
in a single $\tau:E^{+}\overset{\sim}{\to}E^{-}$.
\begin{defn}
\label{def:match(w) and B(sigma,tau)}Denote by $\mathrm{\match}\left(\wl\right)$\marginpar{${\scriptscriptstyle {\scriptstyle \match\left(\wl\right)}}$}
the set of bijections $\sigma:E^{+}\overset{\sim}{\to}E^{-}$ which
are compatible with the colors of the edges. Namely,
\[
\match\left(\wl\right)=\left\{ \sigma\colon E^{+}\overset{\sim}{\to}E^{-}\,\middle|\,\,\sigma\left(E_{i}^{+}\right)=E_{i}^{-}\,\,\forall i\in\left[r\right]\right\} .
\]
For a pair of matchings $\left(\sigma,\tau\right)\in\match\left(\wl\right)^{2}$
we let\marginpar{$B_{\left(\sigma,\tau\right)}$}
\begin{align*}
B_{\left(\sigma,\tau\right)} & =B(\sigma\Big|_{E_{1}^{+}},\tau\Big|_{E_{1}^{+}},\ldots,\sigma\Big|_{E_{r}^{+}},\tau\Big|_{E_{r}^{+}})
\end{align*}
denote the number of blocks in the partition of $2L$ indices induced
by $\sigma$ and $\tau$.
\end{defn}
Clearly, for $\sigma,\tau\in\match\left(\wl\right)$, $\sigma^{-1}\tau$
is a permutation of $E^{+}$ which only mixes edges with the same
color, and belongs to a well-defined conjugacy class in $S_{L}$.\\

For every pair of matchings $\left(\sigma,\tau\right)$ we construct
a surface as a CW-complex. We begin by the $\ell$ boundary components
of the surface. These are, of course, merely $\ell$ pointed $1$-spheres,
but we want to mark some additional points on each of them. For this
sake, we first mark points on the wedge $\wedger$ (in addition to
the basepoint $o$): on the circle corresponding to the generator
$x_{i}$, we mark, in the order of the circle's orientation, distinct
points \marginpar{$p_{i},q_{i},z_{i}$}$p_{i}$, $z_{i}$ and $q_{i}$
that%
\footnote{Our immediate aim requires only the points $p_{i}$ and $q_{i}$.
The role of $z_{i}$ is explained in Claim \ref{claim:properties-of-perms-surface}
below.%
} are also distinct from $o$ -- this is illustrated in the right hand
side of Figure \ref{fig:S^1(w) and marked wedge}. 

Now, for every $w\in\F_{r}$, define $S^{1}\left(w\right)$\marginpar{$S^{1}\left(w\right)$}
to be the pointed $1$-sphere $\left(S^{1},1\right)$ with additional
$2\left|w\right|$ marked points with set of colors $\left\{ p_{i}^{+},p_{i}^{-},q_{i}^{+},q_{i}^{-}\,\middle|\, i\in\left[r\right]\right\} $\marginpar{$p_{i}^{+},p_{i}^{-},q_{i}^{+},q_{i}^{-}$}.
The marking is induced by the maps $f_{w}\colon\left(S^{1},1\right)\to\left(\wedger,o\right)$
from Definition \ref{def:f_w}: the marked points are $f_{w}^{-1}\left(\left\{ p_{1},q_{1},\ldots,p_{r},q_{r}\right\} \right)$.
The color of a marked point $p$ is determined by $f_{w}\left(p\right)$
and the orientation. For example, if $f_{w}\left(p\right)=q_{i}$
and $f_{w}$ advances at $p$ against the orientation of the circle
corresponding to $x_{i}$, then $p$ gets the color $q_{i}^{-}$.
This is illustrated in Figure \ref{fig:S^1(w) and marked wedge}.

\begin{figure}
\includegraphics[bb=0bp 35bp 448bp 270bp]{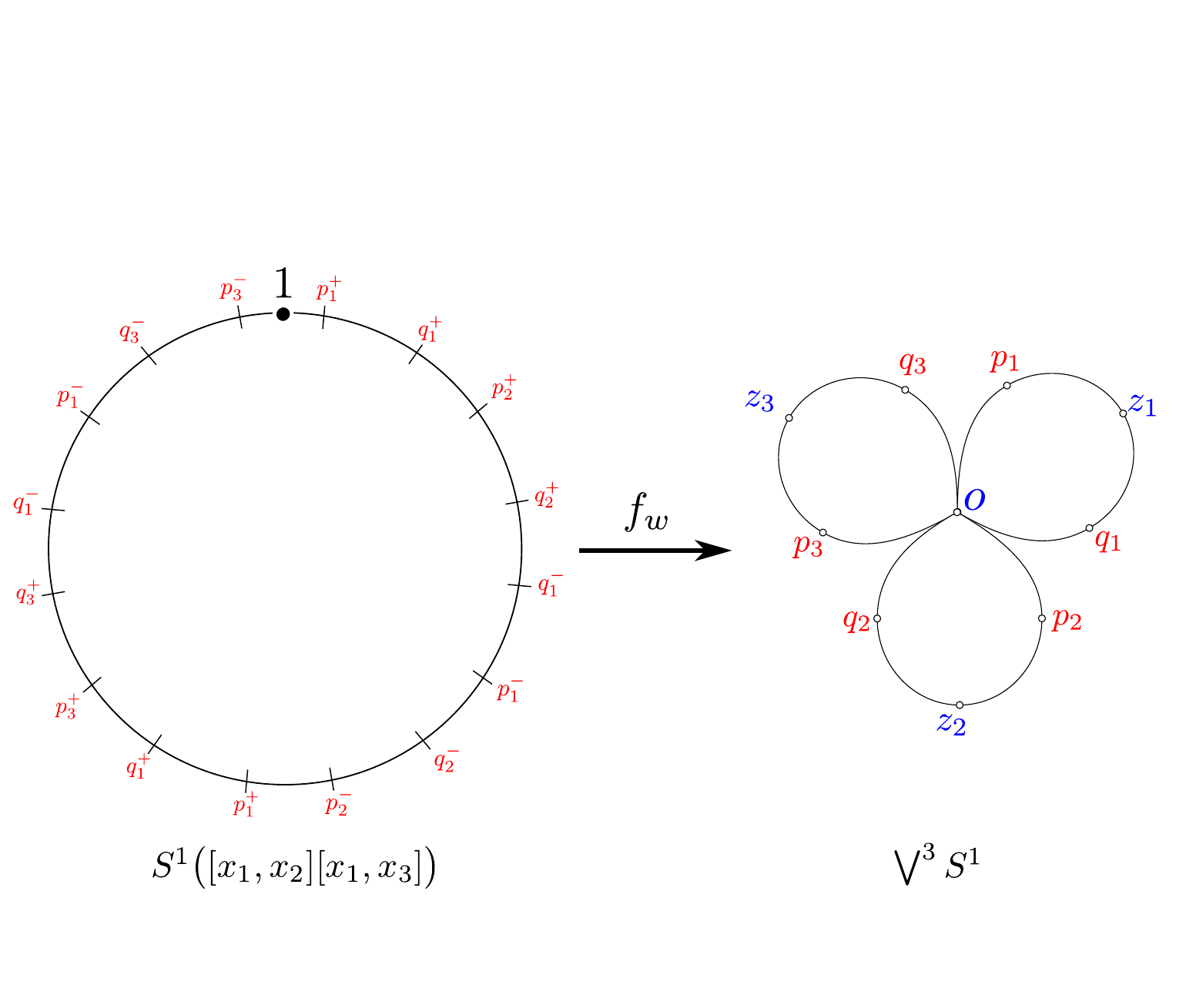}\caption{\label{fig:S^1(w) and marked wedge}The marked $1$-sphere $S^{1}\left(w\right)$
for $w=\left[x_{1},x_{2}\right]\left[x_{1},x_{3}\right]\in\F_{3}$
together with the marked wedge ${\textstyle \bigvee^{3}S^{1}}$.}
\end{figure}

We think of the points $p_{i}^{\pm}$ and $q_{i}^{\pm}$ in $S^{1}\left(w_{1}\right),\ldots,S^{1}\left(w_{\ell}\right)$
as representing the $2L$ indices associated with the different letters
of $\wl$ in the computation of $\trwl\left(n\right)$, as in \eqref{eq:two-indices-for-every-letter}.
By definition, the second index of every letter of $w_{i}$ must be
identical to the first index of the cyclically subsequent letter of
$w_{i}$. The other type of identifications of indices comes from
the fixed bijections $\sigma,\tau\in\match\left(\wl\right)$. Every
$p_{i}^{+}$-point is matched by $\sigma$ to a $p_{i}^{-}$-point.
Similarly, every $q_{i}^{+}$-point is matched by $\tau$ with a $q_{i}^{-}$-point. 
\begin{defn}
\label{def:surface-from-perms}Let $\wl$ be a balanced set of words
and let $\sigma,\tau\in\match\left(\wl\right)$. We associate with\textbf{
}the pair\textbf{ $\left(\sigma,\tau\right)$} a 2-dimensional CW-complex,
denoted $\Sigma_{\left(\sigma,\tau\right)}$\marginpar{$\Sigma_{\left(\sigma,\tau\right)}$}.
Its $1$-dimensional skeleton consists of $S^{1}\left(w_{1}\right),\ldots,S^{1}\left(w_{\ell}\right)$
together with edges (1-dimensional cells) depicting the matchings
$\sigma$ and $\tau$ as above. Namely, for every $i\in\left[r\right]$,
there is an edge connecting every $p_{i}^{+}$-point with its $\sigma$-image,
and an edge connecting every $q_{i}^{+}$-point with its $\tau$-image.
We call these edges \textbf{matching-edges}\emph{}\marginpar{matching-edges}\emph{. }

To define the $2$-dimensional cells, consider cycles in the $1$-skeleton
which are obtained by starting in some marked point on $S^{1}\left(w_{i}\right)$,
moving orientably along $S^{1}\left(w_{i}\right)$ until the next
marked point, then following the matching-edge emanating from this
point and arriving at some marked point in $S^{1}\left(w_{j}\right)$,
then moving orientably along $S^{1}\left(w_{j}\right)$ to the next
marked point, following a matching-edge and so forth, until a cycle
has been completed. A $2$-cell (a disc) is glued along every such
cycle.

Finally, we denote by $v_{1},\ldots,v_{\ell}$ the basepoints of $S^{1}\left(w_{1}\right),\ldots,S^{1}\left(w_{\ell}\right)$,
respectively, and for $i\in\left[\ell\right]$ define $\partial_{i}:\left(S^{1},1\right)\to\left(\Sigma_{\left(\st\right)},v_{i}\right)$
by the identification of $\left(S^{1},1\right)$ with $\left(S^{1}\left(w_{i}\right),1\right)\subset\partial\Sigma_{\left(\st\right)}$.
\end{defn}
We think of the $4L$ marked points as the vertices, or $0$-skeleton
of $\Sigma_{\left(\sigma,\tau\right)}$. Note the description of cycles
we gave in the definition does indeed yield cycles because the walks
on the 1-skeleton are invertible: to get the inverse walks use the
same instructions only with reversed orientation on $S^{1}\left(w_{1}\right),\ldots,S^{1}\left(w_{\ell}\right)$.
In Figures \ref{fig:1-skeleton} and \ref{fig:surface} we illustrate
the 1-skeleton and surface associated with a particular pair of matchings
for the word $w=\left[x,y\right]\left[x,z\right]$. \FigBesBeg \\
\begin{figure}[h]
\centering{}\includegraphics[bb=25bp 0bp 375bp 350bp,scale=0.7]{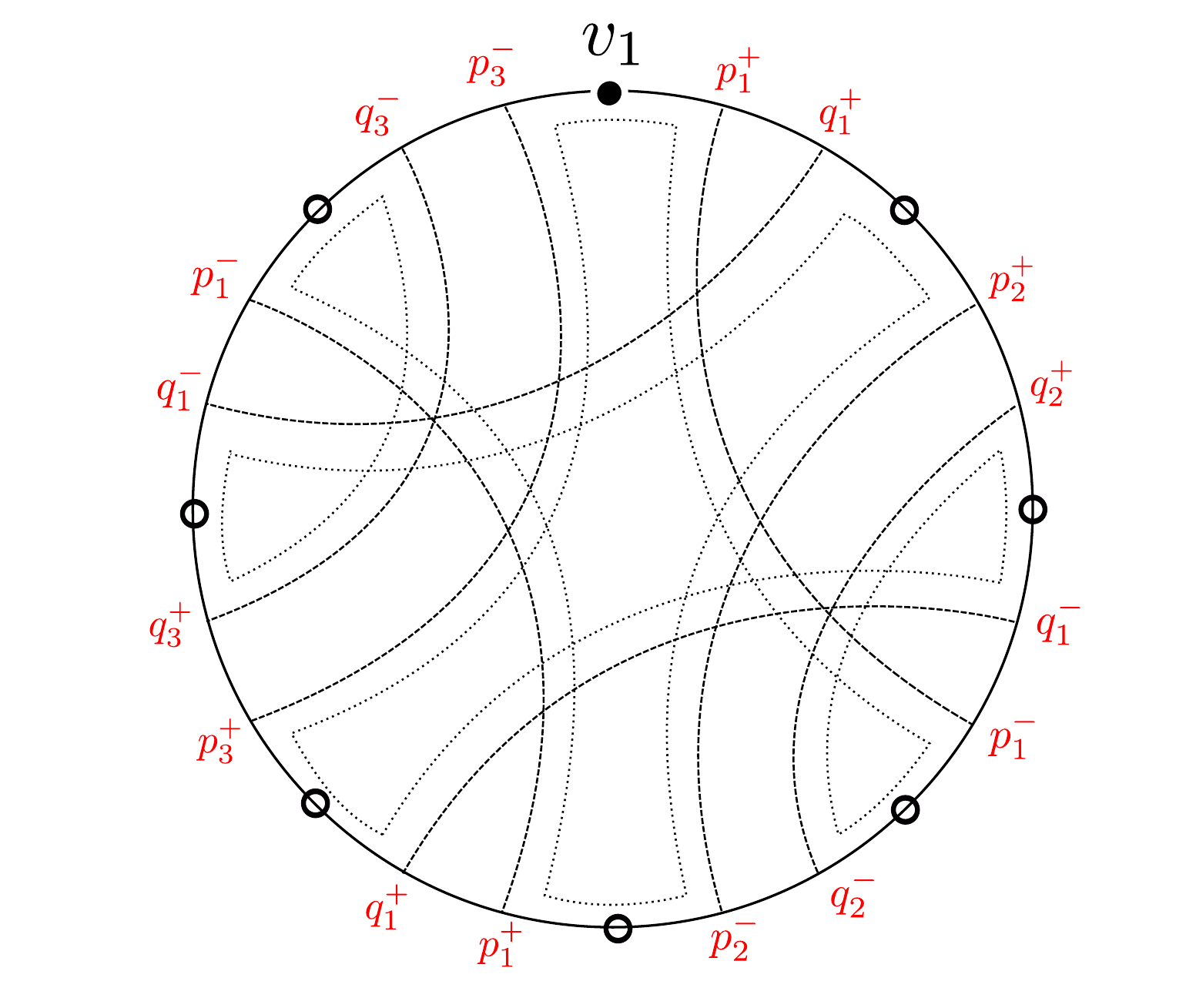}\caption{\label{fig:1-skeleton} The 1-skeleton of $\Sigma_{\left(\sigma,\tau\right)}$
for $w=\left[x_{1},x_{2}\right]\left[x_{1},x_{3}\right]=\left[x,y\right]\left[x,z\right]=x_{1}y_{2}X_{3}Y_{4}x_{5}z_{6}X_{7}Z_{8}$
and the matchings $\sigma=\left(\protect\begin{array}{cccc}
x_{1} & y_{2} & x_{5} & z_{6}\protect\\
X_{3} & Y_{4} & X_{7} & Z_{8}
\protect\end{array}\right)$ and $\tau=\left(\protect\begin{array}{cccc}
x_{1} & y_{2} & x_{5} & z_{6}\protect\\
X_{7} & Y_{4} & X_{3} & Z_{8}
\protect\end{array}\right)$. Dashed lines are matching-edges. The dotted lines trace the boundaries
of the two type-$o$ disc to be glued in (see Claim \ref{claim:properties-of-perms-surface}).
Three additional discs, one of type-$z_{1}$, one of type-$z_{2}$
and one of type-$z_{3}$, are glued in inside the other types of cycles
one can follow (unmarked). For convenience, we also mark here the
additional seven points of $f_{w}^{-1}\left(o\right)$ in $S^{1}\left(w\right)$,
along $v_{1}$, by black circles.}
\end{figure}
\FigBesEnd 
\begin{rem}
For completeness we need also describe what happens when some of\linebreak{}
$S^{1}\left(w_{1}\right),\ldots,S^{1}\left(w_{\ell}\right)$ have
no marked points, namely, when some of $w_{1},\ldots,w_{\ell}$ are
the empty word 1. In this case, whenever $w_{i}=1$, we simply glue
a disc along $S^{1}\left(w_{i}\right)$. To formally make it a CW-complex
we also need to specify a vertex at the boundary of such disc, say,
the basepoint $1$ of $S^{1}\left(w_{i}\right)$. All the results
below work just as well with this extension to trivial words, and
the adjustments required in the proofs are trivial. However, to keep
the writing slightly simpler, we ignore this case in what follows.
\end{rem}
\begin{figure}[h]
\centering{}\includegraphics[bb=50bp 0bp 400bp 320bp,scale=0.8]{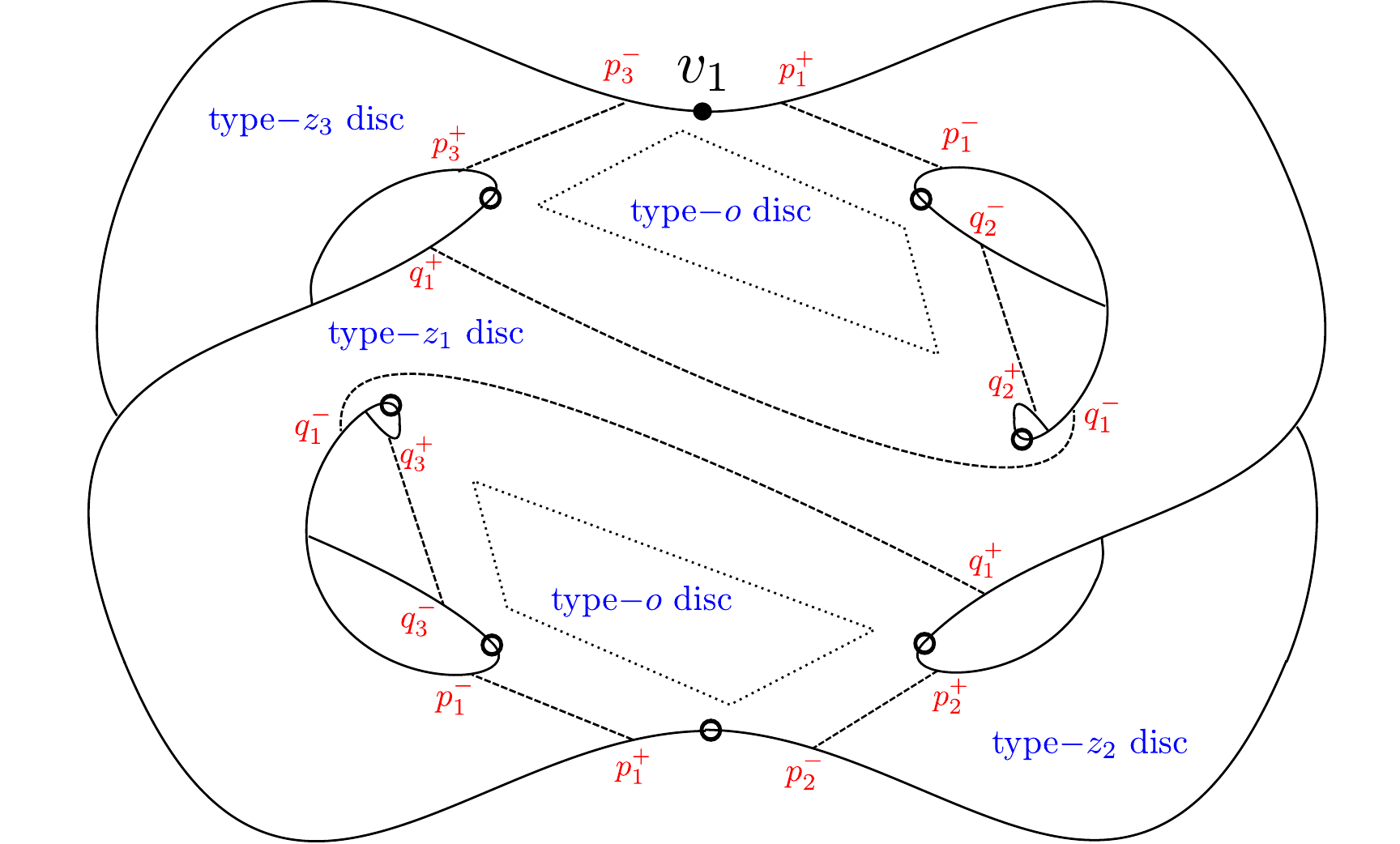}\caption{\label{fig:surface}The CW-complex $\Sigma_{\left(\sigma,\tau\right)}$
corresponding to the word and matchings from Figure \ref{fig:1-skeleton}.
Dashed and dotted lines correspond to those of Figure \ref{fig:1-skeleton}}
\end{figure}

\begin{claim}
\label{claim:properties-of-perms-surface}The CW-complex $\Sigma_{\left(\sigma,\tau\right)}$
has the following properties:
\begin{enumerate}
\item Topologically, it is an orientable surface with $\ell$ boundary components.
\item Each 2-cell $D$ is of one of two types: 

\begin{enumerate}
\item Either $\partial D\cap\left(S^{1}\left(w_{1}\right)\cup\ldots\cup S^{1}\left(w_{\ell}\right)\right)$
contains $o$-points (points from $f_{w_{1}}^{-1}\left(o\right)\cup\ldots\cup f_{w_{\ell}}^{-1}\left(o\right)$),
in which case we call it a \marginpar{type-$o$ disc}\textbf{type-$o$
disc},
\item Or $\partial D\cap\left(S^{1}\left(w_{1}\right)\cup\ldots\cup S^{1}\left(w_{\ell}\right)\right)$
contains $z_{i}$-points (points from $f_{w_{1}}^{-1}\left(z_{i}\right)\cup\ldots\cup f_{w_{\ell}}^{-1}\left(z_{i}\right)$)
for some unique $i$, in which case we call it a \marginpar{type-$z_{i}$ disc}\textbf{type-$z_{i}$
disc}.
\end{enumerate}
\item Every type-$o$ disc corresponds to a block of indices in the partition
induced by $\sigma$ and $\tau$, so that $B_{\left(\sigma,\tau\right)}$
is the number of type-$o$ discs.
\item \label{enu:type-zi-is-a-cycle}Every type-$z_{i}$ disc corresponds
to a cycle of the permutation $\left(\sigma^{-1}\tau\right)\Big|_{E_{i}^{+}}$.
\item Every matching-edge is contained in the boundaries of exactly one
type-$o$ disc and exactly one type-$z_{i}$ disc.
\end{enumerate}
\end{claim}
\begin{proof}
Every segment in $S^{1}\left(w_{j_{1}}\right)$ between two marked
points contains either an $o$-point or a $z_{i}$-point for some
unique $i$. If the boundary $\partial D$ of a $2$-cell $D$ follows
a segment containing an $o$-point, then $\partial D$ goes on to
follow a matching-edge emanating at the first marked point of a letter
in $E^{+}\cup E^{-}$, which, by construction, arrives at a second
marked point of some other letter in $E^{-}\cup E^{+}$. So it then
follows, again, a segment of $S^{1}\left(w_{j_{2}}\right)$ containing
an $o$-point. A similar argument shows that if $\partial D$ contains
a segment of $S^{1}\left(w_{j_{1}}\right)$ with a $z_{i}$-point,
then all the segments of $S^{1}\left(w_{1}\right)\cup\ldots\cup S^{1}\left(w_{\ell}\right)$
it contains have the same property. This shows item $\left(2\right)$.

Items $\left(3\right)$, $\left(4\right)$ and $\left(5\right)$ are
evident from the construction. Every segment of $S^{1}\left(w_{1}\right)\cup\ldots\cup S^{1}\left(w_{\ell}\right)$
between two adjacent marked points is contained in the boundary of
exactly one disc. This and item $\left(5\right)$ show that $\Sigma_{\left(\sigma,\tau\right)}$
is a surface with $S^{1}\left(w_{1}\right)\cup\ldots\cup S^{1}\left(w_{\ell}\right)$
its boundary, hence $\ell$ boundary components. We can orient every
disc according to the orientation of the $\left(S^{1}\left(w_{1}\right)\cup\ldots\cup S^{1}\left(w_{\ell}\right)\right)$-segments
at its boundary, which shows the global orientability and item $\left(1\right)$.
\end{proof}
We can now rewrite \eqref{eq:first_rational_function-for-tr(w)} as
\begin{eqnarray}
 &  & \trwl\left(n\right)=\label{eq:better-formula-for-tr(w)}\\
 &  & \sum_{\begin{gathered}\left(\sigma,\tau\right)\in\\
\match\left(\wl\right)^{2}
\end{gathered}
}\wg\left(\left(\sigma^{-1}\tau\right)\Big|_{E_{1}^{+}}\right)\cdot\ldots\cdot\wg\left(\left(\sigma^{-1}\tau\right)\Big|_{E_{r}^{+}}\right)\cdot n^{\#\left\{ \mathrm{type}-o\,\,\mathrm{discs\, in}\,\Sigma_{\left(\sigma,\tau\right)}\right\} }.\nonumber 
\end{eqnarray}
 
\begin{defn}
\label{def:genus(sigma,tau)}For $\left(\sigma,\tau\right)\in\match\left(\wl\right)^{2}$
denote by $\chi\left(\sigma,\tau\right)$\marginpar{$\chi\left(\sigma,\tau\right)$}
the Euler characteristic of $\Sigma_{\left(\sigma,\tau\right)}$.\end{defn}
\begin{prop}
\label{prop:order-of-contribution-of-(sigma,tau)}The contribution
of $\left(\sigma,\tau\right)\in\match\left(\wl\right)^{2}$ to the
summation \eqref{eq:better-formula-for-tr(w)} giving $\trwl\left(n\right)$
is
\[
\moeb\left(\sigma^{-1}\tau\right)\cdot n^{\chi\left(\sigma,\tau\right)}+O\left(n^{\chi\left(\sigma,\tau\right)-2}\right).
\]
\end{prop}
\begin{proof}
Although the Weingarten function of a permutation is \emph{not }the
product of the Weingarten functions of its disjoint cycles, the leading
term does have this property. More generally, if 
\[
\pi=\left(\pi_{1},\ldots,\pi_{r}\right)\in S_{L_{1}}\times\ldots\times S_{L_{r}}\le S_{L},
\]
then $\left\Vert \pi\right\Vert =\left\Vert \pi_{1}\right\Vert +\ldots+\left\Vert \pi_{r}\right\Vert $
and, by \eqref{eq:mobius}, $\moeb\left(\pi\right)=\moeb\left(\pi_{1}\right)\cdot\ldots\cdot\moeb\left(\pi_{r}\right)$.
Proposition \ref{prop:wg-mobius} therefore yields that
\begin{eqnarray*}
\wg\left(\pi_{1}\right)\cdot\ldots\cdot\wg\left(\pi_{r}\right) & = & \left(\frac{\moeb\left(\pi_{1}\right)}{n^{L_{1}+\left\Vert \pi_{1}\right\Vert }}+O\left(\frac{1}{n^{L_{1}+\left\Vert \pi_{1}\right\Vert +2}}\right)\right)\cdot\ldots\cdot\left(\frac{\moeb\left(\pi_{r}\right)}{n^{L_{r}+\left\Vert \pi_{r}\right\Vert }}+O\left(\frac{1}{n^{L_{r}+\left\Vert \pi_{r}\right\Vert +2}}\right)\right)\\
 & = & \frac{\moeb\left(\pi\right)}{n^{L+\left\Vert \pi\right\Vert }}+O\left(\frac{1}{n^{L+\left\Vert \pi\right\Vert +2}}\right).
\end{eqnarray*}
Since $\left\Vert \pi_{i}\right\Vert =L_{i}-\#\mathrm{cycles}\left(\pi_{i}\right)$,
Claim \ref{claim:properties-of-perms-surface}(\ref{enu:type-zi-is-a-cycle})
yields that 
\[
\left\Vert \sigma^{-1}\tau\right\Vert =L-\sum_{i}\#\left\{ \mathrm{type-}z_{i}\,\,\mathrm{discs\, in}\,\Sigma_{\left(\sigma,\tau\right)}\right\} ,
\]
so the term corresponding to $\left(\sigma,\tau\right)$ in \eqref{eq:better-formula-for-tr(w)}
is
\begin{eqnarray*}
 &  & \frac{\moeb\left(\sigma^{-1}\tau\right)}{n^{2L-\sum_{i}\#\left\{ \mathrm{type-}z_{i}\,\,\mathrm{discs\, in}\,\Sigma_{\left(\sigma,\tau\right)}\right\} }}\cdot n^{\#\left\{ \mathrm{type-}o\,\,\mathrm{discs\, in}\,\Sigma_{\left(\sigma,\tau\right)}\right\} }\cdot\left(1+O\left(\frac{1}{n^{2}}\right)\right)\\
 &  & =\moeb\left(\sigma^{-1}\tau\right)\cdot n^{\#\left\{ \mathrm{discs\, in}\,\Sigma_{\left(\sigma,\tau\right)}\right\} -2L}\cdot\left(1+O\left(\frac{1}{n^{2}}\right)\right).
\end{eqnarray*}
The statement of the proposition follows by noting that the $1$-skeleton
of $\Sigma_{\left(\sigma,\tau\right)}$ has $4L$ $0$-cells (2 marked
points associated with every letter of $\wl$), and $6L$ $1$-cells
($4L$ of them as segments of $S^{1}\left(w_{1}\right)\cup\ldots\cup S^{1}\left(w_{\ell}\right)$
and $2L$ matching-edges), so 
\[
\#\left\{ \mathrm{discs\, in}\,\Sigma_{\left(\sigma,\tau\right)}\right\} -2L=4L-6L+\#\left\{ \mathrm{discs\, in}\,\Sigma_{\left(\sigma,\tau\right)}\right\} =\chi\left(\Sigma_{\left(\sigma,\tau\right)}\right)=\chi\left(\sigma,\tau\right).
\]

\end{proof}
Next, we define (the homotopy class of) a function $f_{\left(\sigma,\tau\right)}\colon\Sigma_{\left(\st\right)}\to\wedger$
which makes $\left(\Sigma_{\left(\st\right)},f_{\left(\st\right)}\right)$
admissible for $\wl$.
\begin{defn}
\label{def:f_(sigma,tau)}Given $\st\in\match\left(\wl\right)$, define
the homotopy class (relative $\partial\Sigma_{\left(\st\right)}$)
of a map\marginpar{$f_{\left(\st\right)}$} $f_{\left(\sigma,\tau\right)}\colon\Sigma_{\left(\st\right)}\to\wedger$
as follows:
\begin{itemize}
\item Define $f_{\left(\st\right)}$ on $\partial\Sigma_{\left(\st\right)}$
by setting $f_{\left(\sigma,\tau\right)}\Big|_{S^{1}\left(w_{i}\right)}\equiv f_{w_{i}}\circ\partial_{i}^{-1}$
for every $i\in\left[\ell\right]$.
\item Extend $f_{\left(\st\right)}$ to the entire $1$-skeleton of $\Sigma_{\left(\st\right)}$
by setting $f_{\left(\st\right)}$ to be constant on every matching-edge,
namely, $f_{\left(\st\right)}\Big|_{e}\equiv p_{i}$ for every $p_{i}$-matching-edge
$e$ etc.
\item On every disc ($2$-cell) $D$, $f_{\left(\st\right)}$ now maps its
boundary to a nullhomotopic loop in $\wedger$, so there exists a
unique way, up to homotopy, to extend $f_{\left(\st\right)}$ to the
interior of $D$ (as in Lemma \ref{lem:homotopy=00003Dsame_induced_map}).
\end{itemize}
\end{defn}
From Definitions \ref{def:surface-from-perms} and \ref{def:f_(sigma,tau)}
and Claim \ref{claim:properties-of-perms-surface} we conclude:
\begin{cor}
\label{cor:(Sigma,f) of (sigma,tau) is admissible}For every $\st\in\match\left(\wl\right)$,
the pair $\left(\Sigma_{\left(\st\right)},f_{\left(\st\right)}\right)$
is admissible for $\wl$.
\end{cor}
It turns out that all admissible maps $\left(\Sigma,f\right)$ for
$\wl$ can be basically obtained this way, as long as $f$ is incompressible.
\begin{lem}
\label{lem:every admissible and incompressible obtained from matchings}If
$\left(\Sigma,f\right)$ is admissible for $\wl$ and $f$ is incompressible,
then there is a pair of matchings $\left(\st\right)\in\match\left(\wl\right)$
so that $\left(\Sigma,f\right)\sim\left(\Sigma_{\left(\st\right)},f_{\left(\st\right)}\right)$.\end{lem}
\begin{proof}
Let%
\footnote{A straight-forward argument is available when $f$ is smooth outside
$f^{-1}\left(o\right)$ and $p_{i}$ and $q_{i}$ are regular points
for each $i\in\left[r\right]$. In this case, the desired matchings
are obtained by the arc parts of $f^{-1}\left(p_{i}\right)$ and $f^{-1}\left(q_{i}\right)$.%
} $\left(\Sigma,f\right)$ be admissible for $\wl$ and $f$ incompressible.
As in Lemma \ref{lem:homotopy=00003Dsame_induced_map}, we find a
finite set of oriented disjoint arcs $\gamma_{1},\ldots,\gamma_{m}\colon\left[0,1\right]\to\Sigma$
with endpoints in $\left\{ v_{1},\ldots,v_{\ell}\right\} $ which
cut $\Sigma$ into discs. For every $j\in\left[m\right]$, denote
$u_{j}=\left[f\left(\gamma_{j}\right)\right]\in\F_{r}$. We now want
to mark the arc $\gamma_{j}$ with $2\left|u_{j}\right|$ points colored
with $\left\{ p_{i},q_{i}\,\middle|\, i\in\left[r\right]\right\} $
as we did in $S^{1}\left(w\right)$ in the beginning of this section,
only, for now, without the $\pm$ sign. Namely, using the function
$g\colon\left[0,1\right]\to S^{1}$ defined by $t\mapsto e^{2\pi it}$,
use $S^{1}\left(u_{j}\right)$ to mark and color points on $\gamma_{j}$. 

Now, in every disc $D$ which is cut from $\Sigma$ by the arcs $\gamma_{1},\ldots,\gamma_{m}$,
use the orientation on $D$ (induced from the one on $\Sigma$) to
orient each arc $\gamma_{j}$ at the boundary of $D$, and add accordingly
$\pm$ signs to the colors of the marked points on this arc. In particular,
every marked point on a $\gamma_{j}$ is signed {}``$+$'' for one
of the two discs it borders and signed {}``$-$'' for the other. 

Since the image of $\partial D$ through $f$ is nullhomotopic, the
sequence of marked points one reads along $\partial D$ can be reduced
to an empty sequence by successive deletions of pairs of the form
$p_{i}^{+}p_{i}^{-}$, $p_{i}^{-}p_{i}^{+}$, $q_{i}^{+}q_{i}^{-}$
or $q_{i}^{+}q_{i}^{-}$. We use one of these reduction processes
and, at each step, draw an arc inside $D$ between the two marked
points we delete at that step. A simple inductive argument shows that
at each step, the remaining unpaired points are all in the boundary
of the same disc bounded by parts of $\partial D$ and the existing
arcs (with no arcs inside the disc), so one can draw in its interior
a new arc connecting the next pair of points.

Next, use the new {}``reduction'' arcs to determine $\sigma$ and
$\tau$: for every marked point $t$ on $\partial\Sigma$, $t$ belongs
to some disc $D$, and follow the arc emanating from $t$ it to some
$t'\in\partial D$. If $t'$ is not in $\partial\Sigma$, but, say,
in $\gamma_{j}$, follow the arc from $t'$ inside the other disc
bordering $\gamma_{j}$. Continue in the same way until a point from
$\partial\Sigma$ is reached. It is easy to see that this induces
matchings $\sigma,\tau\in\match\left(\wl\right)$: for example, a
$q_{i}^{+}$-point in $\partial\Sigma$ is connected by an arc to
a $q_{i}^{-}$-point. If the latter is not on $\partial\Sigma$, it
is identified with a $q_{i}^{+}$-point on a neighboring disc, which
is then connected to another $q_{i}^{-}$-point, and so forth. Note
that some of the arcs may form cycles in the interior of $\Sigma$,
and simply disregard or delete these one. Let $A$ be the set of arcs
we used for determining $\sigma$ and $\tau$. This is illustrated
in Figure \ref{fig:pentagon2-1}.\FigBesBeg \\
\begin{figure}[h]
\centering{}\includegraphics[bb=0bp 0bp 300bp 280bp,scale=0.8]{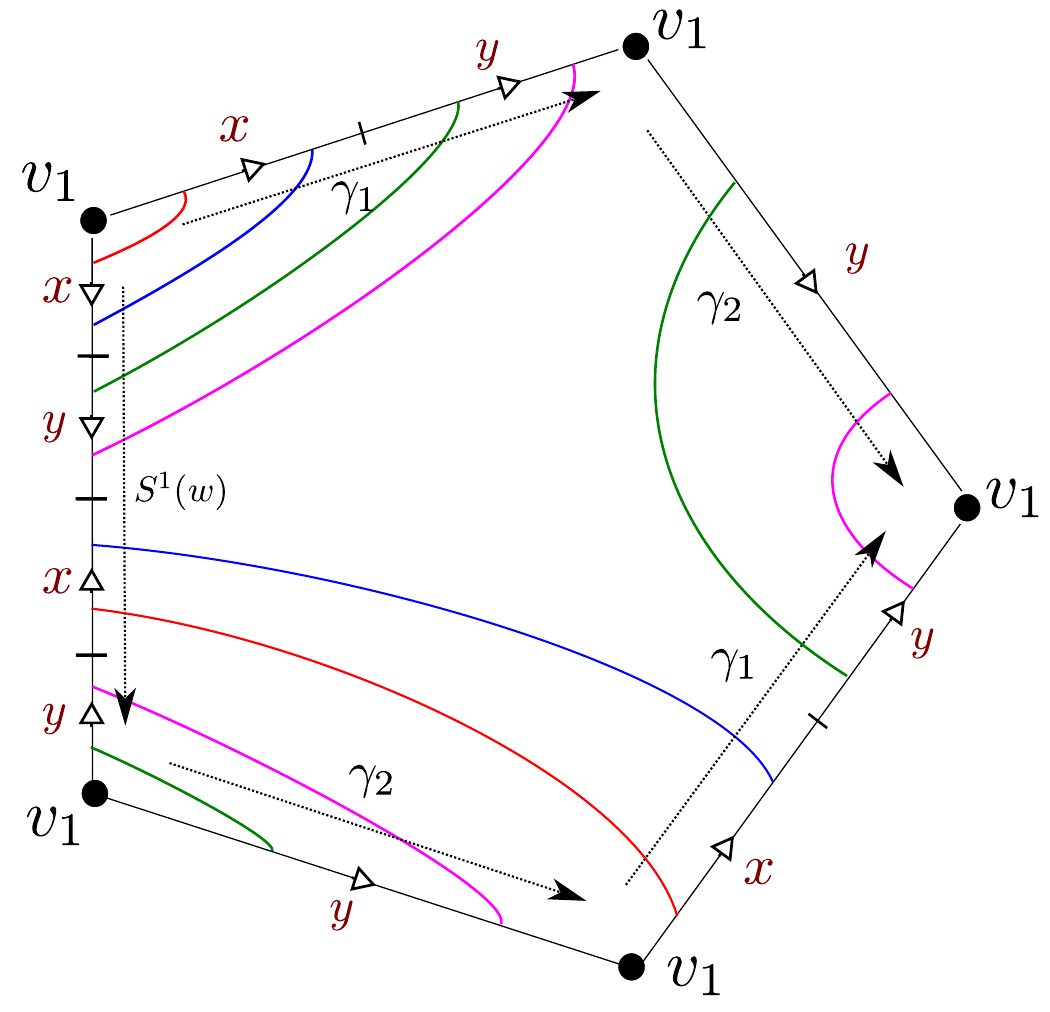}\caption{\label{fig:pentagon2-1} Let $w=\left[x,y\right]=x_{1}y_{2}X_{3}Y_{4}$
and $\left(\Sigma_{1,1},f\right)$ be admissible for $w$ and representing
the solution $w=\left[xy,y\right]$. Namely, if $\gamma_{1}$ and
$\gamma_{2}$ are arcs in $\Sigma_{1,1}$ representing two basis elements
in $\F_{2}$, then $\left[f\circ\gamma_{1}\right]=xy$ and $\left[f\circ\gamma_{2}\right]=y$.
The arcs $\gamma_{1}$ and $\gamma_{2}$ cut $\Sigma_{1,1}$ to a
sole disc $D$, and the colored arcs in the figure correspond to a
particular reduction process of the word read along $\partial D$,
as explained in the proof of Lemma \ref{lem:every admissible and incompressible obtained from matchings}.
The matchings we get here are $\sigma=\tau=\left(\protect\begin{array}{cc}
x_{1} & y_{2}\protect\\
X_{3} & Y_{4}
\protect\end{array}\right)$ (this is the only possible matching for this particular word).}
\end{figure}
\FigBesEnd 

We claim that $\Sigma\setminus\bigcup_{\alpha\in A}\alpha$ is a union
of discs. To see this, we first perturb $f$ so that it agrees with
$f_{u_{j}}$ on $\gamma_{j}$ for every $j\in\left[m\right]$. We
then perturb it so that it is constant on every arc drawn in the reduction
process: this only requires to change $f$ in the interior of every
disc $D$ which is cut from $\Sigma$ by the arcs $\gamma_{j}$. Now
regard the arcs $\alpha\in A$ as the matching-edges in Definition
\ref{def:surface-from-perms}, and follow the cycles along these arcs
and $\partial\Sigma$ described in the same definition. These are
precisely the boundaries of the connected components of $\Sigma\setminus\bigcup_{\alpha\in A}\alpha$.
As in Definition \ref{def:f_(sigma,tau)}, the image of $f$ through
each such cycle is easily seen to be nullhomotopic. But $f$ is incompressible,
hence each such circle must bound a disc. 

This shows that $\Sigma$ is homeomorphic to $\Sigma_{\left(\sigma,\tau\right)}$
with the arcs $\alpha\in A$ mapped to the matching-edges in $\Sigma_{\left(\sigma,\tau\right)}$.
Since $f$ and $f_{\left(\st\right)}$ agree on the $1$-skeleton,
they are homotopic (using, again, Lemma \ref{lem:homotopy=00003Dsame_induced_map}).
Hence $\left(\Sigma,f\right)\sim\left(\Sigma_{\left(\st\right)},f_{\left(\st\right)}\right)$.
\end{proof}
Since in every admissible $\left(\Sigma,f\right)$ for $\wl$ with
maximal Euler characteristic $f$ is incompressible, we deduce from
Corollary \ref{cor:(Sigma,f) of (sigma,tau) is admissible} and Lemma
\ref{lem:every admissible and incompressible obtained from matchings}
that,
\begin{cor}
\label{cor:max-Euler-of-matchings=00003Dchi}The highest Euler characteristic
of a pair of matchings is $\ch\left(\wl\right)$, namely, 
\[
\max_{\left(\sigma,\tau\right)\in\match\left(\wl\right)^{2}}\chi\left(\sigma,\tau\right)=\ch\left(\wl\right).
\]

\end{cor}
Moreover, we get an extension to a Theorem of Culler  \cite[Thm 4.1]{CULLER},
stating that the number of equivalence classes of solutions to $\left[u_{1},v_{1}\right]\cdots\left[u_{g},v_{g}\right]=w$
with $g=\cl\left(w\right)$ is finite:
\begin{cor}
\label{cor:Finitely many incompressible maps}For every $\wl\in\F_{r}$,
there are at most finitely many equivalence classes of $\left(\Sigma,f\right)$
which are admissible for $\wl$ and incompressible. In particular,
the set $\sol\left(\wl\right)$ is finite.\end{cor}
\begin{rem}
\label{remark: algo-for-cl}In the proof of Lemma \ref{lem:every admissible and incompressible obtained from matchings},
we could choose in each disc $D$ a reduction process that comes from
a reduction of the word we read along $\partial D$. This would mean
that whenever we pair two $p_{i}$-points, we also match their associated
two $q_{i}$-points. In other words, the bijections we obtain satisfy
$\sigma=\tau$. Thus,
\[
\max_{\sigma\in\match\left(\wl\right)}\chi\left(\sigma,\sigma\right)=\ch\left(\wl\right).
\]
This fact, in a slightly different language and for a single word,
appears already in Culler's work, where it is used as an algorithm
to compute $\cl\left(w\right)$ \cite[Theorem 2.1]{CULLER}. More
generally, Corollary \ref{cor:max-Euler-of-matchings=00003Dchi} provides
an algorithm to compute $\ch\left(\wl\right)$ for every $\wl\in\F_{r}$.
Furthermore, Lemma \ref{lem:every admissible and incompressible obtained from matchings}
shows that by going over all matchings, we can find representatives
for all admissible $\left(\Sigma,f\right)$ for $\wl$ with $f$ incompressible.
It is still not clear at this point how to tell apart the different
equivalence classes of admissible maps, but we face this challenge
in Section \ref{sec:More-Consequences} below.
\end{rem}
Finally, using \eqref{eq:better-formula-for-tr(w)} and Proposition
\ref{prop:order-of-contribution-of-(sigma,tau)}, we can now deduce
Theorem \ref{thm:leading-exponent-general}:
\begin{cor}
\label{cor:trw(n) leading exponent}For any $\wl\in\F_{r}$,
\begin{equation}
\trwl\left(n\right)=n^{\ch\left(\wl\right)}\left[\sum_{\begin{gathered}\left(\sigma,\tau\right)\in\match\left(\wl\right)^{2}\\
\mathrm{with}\,\chi\left(\sigma,\tau\right)=\ch\left(\wl\right)
\end{gathered}
}\moeb\left(\sigma^{-1}\tau\right)\right]+O\left(n^{\ch\left(\wl\right)-2}\right).\label{eq:leading-term-of-trw(n)-as-sum-over-min-genus-perms}
\end{equation}
In particular, $\trwl\left(n\right)=O\left(n^{\ch\left(\wl\right)}\right)$.\end{cor}
\begin{example}
\label{example:[x,y][x,z] leading term 0}As an example, consider
the word $w=\left[x,y\right]\left[x,z\right]=x_{1}y_{2}X_{3}Y_{4}x_{5}z_{6}X_{7}Z_{8}$.
The two possible matchings of $E^{+}$ and $E^{-}$ which preserve
the alphabet are $\left(\begin{array}{cccc}
x_{1} & y_{2} & x_{5} & z_{6}\\
X_{3} & Y_{4} & X_{7} & Z_{8}
\end{array}\right)$ and $\left(\begin{array}{cccc}
x_{1} & y_{2} & x_{5} & z_{6}\\
X_{7} & Y_{4} & X_{3} & Z_{8}
\end{array}\right)$, so there are exactly $4$ pairs of matchings in this case. A simple
computation shows all of them have Euler characteristic $-3$, which
shows that $\ch\left(w\right)=-3$ (and $\cl\left(w\right)=2$). For
two of the pairs, $\moeb\left(\sigma^{-1}\tau\right)=1$ and for the
other two $\moeb\left(\sigma^{-1}\tau\right)=-1$. Hence, by Corollary
\ref{cor:trw(n) leading exponent}, $\tr_{\left[x,y\right]\left[x,z\right]}\left(n\right)=n^{-3}\cdot0+O\left(n^{-5}\right)$.
In fact, the full computation in this case (by Theorem \ref{thm:trw-rational})
shows that $\tr_{\left[x,y\right]\left[x,z\right]}\left(n\right)$
is identically zero for every $n\ge2$. In particular, this example
shows that it is not true in general that $\trw\left(n\right)=\theta\left(\frac{1}{n^{2\cdot\cl\left(w\right)-1}}\right)$,
nor that $\trw\left(n\right)\not\equiv0$ for $w\in\left[\F_{r},\F_{r}\right]$.
\end{example}

\begin{example}
As another example, consider $w=\left[x,y\right]^{2}$. There are
four matchings in $\match\left(w\right)$, hence $16$ pairs. Among
them, twelve have $\chi=-3$ and four have $\chi=-5$. Of the twelve
with $\chi=-3$, four have $\moeb\left(\sigma^{-1}\tau\right)=1$
and eight have $\moeb\left(\sigma^{-1}\tau\right)=-1$. Corollary
\ref{cor:trw(n) leading exponent} thus gives $\tr_{\left[x,y\right]^{2}}=\frac{-4}{n^{3}}+O\left(\frac{1}{n^{5}}\right)$.
(Compare with the exact rational expression in \eqref{eq:rational-expression-for-[a,b]^2}).
\end{example}
We end this section with one more interesting property of $\trwl\left(n\right)$.
\begin{cor}
In the Laurent series in $n$ expressing $\trwl\left(n\right)$, the
coefficient of \emph{every other exponent }vanishes. If $\ell$ is
odd, only terms with odd exponents may not vanish, and if $\ell$
is even, only terms with even exponents may not vanish.\end{cor}
\begin{proof}
Actually, this is true for the contribution of every $\left(\sigma,\tau\right)\in\match\left(\wl\right)^{2}$
separately. That the leading exponent of every contribution has the
same parity as $\ell$ follows from the orientability of the surface
$\Sigma_{\left(\sigma,\tau\right)}$: we saw that this leading exponent
is $\chi\left(\st\right)=2-2\cdot\mathrm{genus}\left(\Sigma_{\left(\sigma,\tau\right)}\right)-\ell$.
The statement now follows from the property of the Weingarten function
that the coefficient of every other exponent vanishes (see the paragraph
right after Proposition \ref{prop:wg-mobius}).
\end{proof}

\section{The Pairs of Matchings Poset\label{sec:The-pairs-of-matchings-Poset}}

Corollary \ref{cor:trw(n) leading exponent} shows that in order to
prove Theorem \ref{thm:main - general}, it is enough to restrict
attention to pairs of matchings $\left(\st\right)\in\match\left(\wl\right)^{2}$
with $\chi\left(\st\right)=\ch\left(\wl\right)$, namely, with $\left(\st\right)$
so that $\left[\left(\Sigma_{\left(\st\right)},f_{\left(\st\right)}\right)\right]\in\sol\left(\wl\right)$.
However, all our proofs below regarding these matchings only use the
fact that $f_{\left(\st\right)}$ is incompressible. Therefore, we
continue analyzing pairs $\left(\st\right)$ with $\left(\Sigma_{\left(\st\right)},f_{\left(\st\right)}\right)$
incompressible. This also allows us to prove Theorem \ref{thm:K(G,1) for incompressible}
in its full generality.

To continue our analysis, we gather all pairs $\left(\st\right)$
which correspond to the same equivalence class of an admissible, incompressible
$\left(\Sigma,f\right)$. The main result of this section is that
there is a natural poset structure on every such set of pairs, and
that the leading coefficient of the contribution of this set to $\trwl\left(n\right)$
is the Euler characteristic of (the simplicial complex associated
with) this poset.

First, we introduce an order on pairs of permutations which is related
to the partial order on $S_{L}$ defined in Section \ref{sub:Weingarten-function-and-Collins-Sniady}:
for $\sigma,\tau,\sigma',\tau'\in S_{L}$, we write%
\footnote{This paper uses the same symbol $\preceq$ to denote different partial
orders. However, two different partial orders are always defined on
different types of elements, so it should be easy to realize which
partial order is referred to at any point in the text.%
} $\left(\sigma',\tau'\right)\preceq\left(\sigma,\tau\right)$\marginpar{${\scriptstyle \left(\sigma',\tau'\right)\preceq\left(\sigma,\tau\right)}$}
if 
\[
\left\Vert \sigma^{-1}\tau\right\Vert =\left\Vert \sigma^{-1}\sigma'\right\Vert +\left\Vert \left(\sigma'\right)^{-1}\tau'\right\Vert +\left\Vert \left(\tau'\right)^{-1}\tau\right\Vert .
\]
In other words, consider the Cayley graph of $S_{L}$ with respect
to all transpositions. We say that $\left(\sigma',\tau'\right)\preceq\left(\sigma,\tau\right)$
if and only if there is a geodesic in this Cayley graph from $\sigma$
to $\tau$ which goes through $\sigma'$ and then through $\tau'$.
\[
\xymatrix{\sigma\ar@{--}[r] & \sigma'\ar@{--}[r] & \tau'\ar@{--}[r] & \tau}
\]
Clearly, this order, with the same definition, can be applied just
as well to pairs of bijections $\sigma,\tau,\sigma',\tau'\colon E^{+}\overset{\sim}{\to}E^{-}$.
In fact, we can identify the set of bijections $E^{+}\overset{\sim}{\to}E^{-}$
with $S_{L}$ by declaring an arbitrary bijection as the identity
element. We can then think of $\match\left(\wl\right)$ as a set of
permutations in $S_{L}$. We shall use both points of views interchangeably.
\begin{defn}
\label{def:perm-poset}Let $\left(\Sigma,f\right)$ be admissible
for $\wl$ and incompressible. The \textbf{pairs of matchings poset
}of $\left(\Sigma,f\right)$, denoted \marginpar{$\pmp\left(\Sigma,f\right)$}$\pmp\left(\Sigma,f\right)$,
consists of pairs of matchings in\linebreak{}
$\match\left(\wl\right)^{2}$ which are associated, up to equivalence,
with $\left(\Sigma,f\right)$. Namely, 
\[
\pmp\left(\Sigma,f\right)\overset{\mathrm{def}}{=}\left\{ \left(\sigma,\tau\right)\in\match\left(\wl\right)^{2}\,\middle|\,\left(\Sigma_{\left(\st\right)},f_{\left(\st\right)}\right)\sim\left(\Sigma,f\right)\right\} .
\]
The partial order $\preceq$ on $\pmp\left(\Sigma,f\right)$ is induced
from the partial order on pairs of bijections $E^{+}\overset{\sim}{\to}E^{-}$.
\end{defn}
\medskip{}

The following property of pairs of matching associated with an incompressible
map is important in what follows.
\begin{lem}
\label{lem:joint-boundaries-of-discs-in-surface}If $f_{\left(\st\right)}$
is incompressible for some $\left(\sigma,\tau\right)\in\match\left(\wl\right)^{2}$,
then any two neighboring discs in $\Sigma_{\left(\sigma,\tau\right)}$,
which are necessarily of type-$o$ and of type-$z_{i}$ for some $i\in\left[r\right]$,
have at most two common matching-edges at their boundaries: at most
one $p_{i}$-edge and at most one $q_{i}$-edge.\end{lem}
\begin{proof}
Assume, to the contrary, that there are discs $D_{1}$ of type-$o$
and $D_{2}$ of type-$z_{i}$ so that $\partial D_{1}\cap\partial D_{2}$
contains two distinct matching-edges $e_{1}$ and $e_{2}$ of the
same color, say $q_{i}$. Let $\gamma$ be a simple closed curve that
traverses exactly two matching-edges -- $e_{1}$ and $e_{2}$ -- and
each one exactly once. It is easy to see that $f_{\left(\st\right)}\left(\gamma\right)$
is then nullhomotopic in $\wedger$, and so $\gamma$ bounds a disc
by the assumption. But this is impossible as there are points from
$\partial\Sigma_{\left(\st\right)}$ at both sides of $\gamma$ (e.g.~the
points at the endpoints of $e_{1}$ and $e_{2}$).
\end{proof}
The poset $\pmp\left(\Sigma,f\right)$ is a downward-closed sub-poset
of the poset of pairs of bijections. Namely,
\begin{lem}
\label{lem:bp(w)-closed-downwards}Assume that $\left(\sigma,\tau\right)\in\pmp\left(\Sigma,f\right)$,
that $\sigma',\tau'\colon E^{+}\overset{\sim}{\to}E^{-}$ are bijections
and that $\left(\sigma',\tau'\right)\preceq\left(\sigma,\tau\right)$.
Then $\left(\sigma',\tau'\right)\in\pmp\left(\Sigma,f\right)$.\end{lem}
\begin{proof}
First note the following observation: let $\pi\in S_{L}$ satisfy
$\left\Vert \pi\right\Vert =k$ and let $t_{1}t_{2}\ldots t_{k}$
be a product of transpositions giving $\pi$. Then for every $j$,
the two elements $x,y\in\left[L\right]$ swapped by $t_{j}$ must
be two elements which sit in two different cycles in $t_{1}t_{2}\ldots t_{j-1}$
but which belong to the same cycle in $\pi$. This follows from the
identity $\left\Vert \pi\right\Vert =L-\#\mathrm{cycles}\left(\pi\right)$
and from the fact that when a permutation is multiplied by a transposition
either two of its cycles are merged together or one of its cycles
is split into two. 

We claim that from this simple observation it follows that $\sigma',\tau'\in\match\left(\wl\right)$,
i.e.~that $\sigma'$ and $\tau'$ map $E_{i}^{+}$ to $E_{i}^{-}$
for every $i\in\left[r\right]$. Indeed, this is certainly true for
$\sigma$ and $\tau$ and thus $\sigma^{-1}\tau$ maps $E_{i}^{+}$
to $E_{i}^{+}$ for every $ $$i$. By assumption, there is a product
of transpositions in $\mathrm{Sym}\left(E^{+}\right)$ of minimal
length which gives $\sigma^{-1}\tau$ such that two of its prefixes
equal $\sigma^{-1}\sigma'$ and $\sigma^{-1}\tau'$. By the observation,
no transposition in the product can mix elements of $E_{i}^{+}$ and
$E_{j}^{+}$ with $i\ne j$, and thus this is also true for $\sigma^{-1}\sigma'$
and $\sigma^{-1}\tau'$, and indeed $\sigma',\tau'\in\match\left(\wl\right)$.

It is left to show that $\left(\Sigma_{\left(\sigma',\tau'\right)},f_{\left(\sigma',\tau'\right)}\right)\sim\left(\Sigma_{\left(\st\right)},f_{\left(\st\right)}\right)$.
It is enough to show this in the case when $\left(\sigma,\tau\right)$
covers $\left(\sigma',\tau'\right)$ (see Footnote \ref{fn:graded-poset}).
In this case, either $\sigma'=\sigma$ and $\tau^{-1}\tau'$ is a
transposition, or $\tau'=\tau$ and $\sigma^{-1}\sigma'$ is a transposition.
Assume the former case, the latter having the exact same proof. So
$\left(\sigma',\tau'\right)=\left(\sigma,\tau'\right)$ is the same
as $\left(\sigma,\tau\right)$, except for two $q_{i}^{+}$-points
$j$ and $k$, for some $ $$i$, with $\tau'\left(j\right)=\tau\left(k\right)$
and $\tau'\left(k\right)=\tau\left(j\right)$. If we abuse notation
and let $j$ and $k$ denote also the corresponding letters in $E^{+}$,
then $\sigma^{-1}\tau'\cdot\left(j\, k\right)=\sigma^{-1}\tau\in\mathrm{Sym}\left(E^{+}\right)$.
Because of the equality $\left\Vert \sigma^{-1}\tau'\right\Vert =\left\Vert \sigma^{-1}\tau\right\Vert -1$,
$j$ and $k$ must belong to different cycles of $\sigma^{-1}\tau'$
and to the same cycle of $\sigma^{-1}\tau$. Namely, the $q_{i}^{+}$-points
$j$ and $k$ are at the boundary of the same type-$z_{i}$ disc of
$\Sigma_{\left(\st\right)}$. 

Consider $\Sigma_{\left(\sigma,\tau\right)}$ and the two matching-edges
$e_{j}$ and $e_{k}$ emanating from $j$ and $k$, respectively,
and let $D$ denote the type-$z_{i}$ disc they both belong to. By
Lemma \ref{lem:joint-boundaries-of-discs-in-surface}, they belong
to two different type-$o$ discs. The change in these two edges is
the only change in the $1$-skeleton of the CW-complex when moving
from $\Sigma_{\left(\sigma,\tau\right)}$ to $\Sigma_{\left(\sigma,\tau'\right)}$.
In fact, to obtain $\Sigma_{\left(\sigma,\tau'\right)}$ from $\Sigma_{\left(\st\right)}$
we can do the following: $\left(i\right)$ draw two new disjoint edges
(arcs) inside $D$: $e_{j}'$ from $j$ to $\tau\left(k\right)$ and
$e_{k}'$ from $k$ to $\tau\left(j\right)$ -- this is always possible
because all $q_{i}$-matching-edges at the boundary of a type-$z_{i}$
disc are oriented. $\left(ii\right)$ Replace $e_{j}$ and $e_{k}$
by $e_{j}'$ and $e_{k}'$. The change results in splitting the joint
type-$z_{i}$ into two discs and merging the two type-$o$ discs into
one. We illustrate this in figure \ref{fig:disc-split}.
\begin{figure}[h]
\centering{}\includegraphics[bb=0bp 50bp 350bp 300bp,scale=0.8]{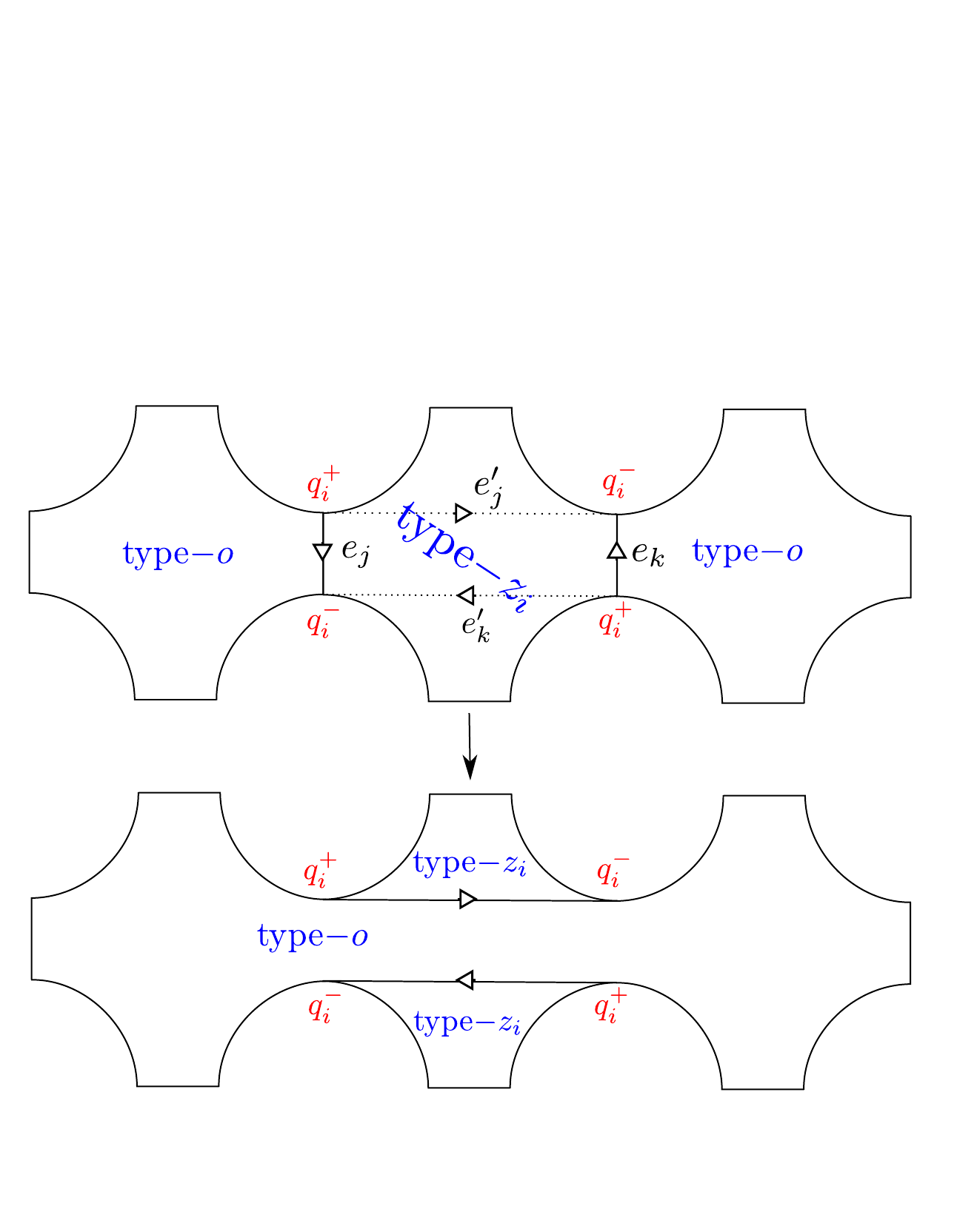}\caption{\label{fig:disc-split}Swapping two $q_{i}$-matching-edges in the
boundary of the same type-$z_{i}$ disc in $\Sigma_{\left(\sigma,\tau\right)}$
for some $\left(\sigma,\tau\right)\in\pmp\left(\Sigma,f\right)$ results
in $\Sigma_{\left(\sigma,\tau'\right)}$ for some other $\left(\sigma,\tau'\right)\in\pmp\left(\Sigma,f\right)$.
The number of type-$z_{i}$ discs increases by one, while the number
of type-$o$ discs decreases by one. This corresponds to moving one
step down, namely, to a covered element, in the poset $\pmp\left(\Sigma,f\right)$.}
\end{figure}

From this description of $\Sigma_{\left(\sigma,\tau'\right)}$ there
is a natural homeomorphism $\Sigma_{\left(\st\right)}\cong\Sigma_{\left(\sigma,\tau'\right)}$,
and $f_{\left(\sigma,\tau\right)}$ and $f_{\left(\sigma,\tau'\right)}$
agree on the entire common parts of the $1$-skeletons, i.e.~on all
boundary and matching-edges surfaces, except for, possibly, on $e_{j},e_{k}$
and $e_{j}',e_{k}'$. But within the freedom left in the definition
of these functions (Definition \ref{def:f_(sigma,tau)}), we can assume
that both are constant functions along all of $e_{j},e_{k},e_{j}'$
and $e_{k}'$, mapping all four matching edges to $q_{i}\in\wedger$.
Then, by Lemma \ref{lem:homotopy=00003Dsame_induced_map}, they are
homotopic to each other. Thus $\left(\Sigma_{\left(\st\right)},f_{\left(\st\right)}\right)\sim\left(\Sigma_{\left(\sigma,\tau'\right)},f_{\left(\sigma,\tau'\right)}\right)$.
\end{proof}
As an example, let $w=\left[x,y\right]\left[x,z\right]$.\label{[x,y][x,z] - PMP}
We already mentioned in Example \ref{example:[x,y][x,z] leading term 0}
above that there are four pairs of matchings, all of which with $\chi=-3$.
An easy application of Lemma \ref{lem:bp(w)-closed-downwards} shows
that all four belong to same class of admissible incompressible $\left[\left(\Sigma,f\right)\right]$.
Two of the four pairs satisfy $\sigma=\tau$, and both are smaller
($\prec$) than the other two pairs in which $\sigma^{-1}\tau$ is
a transposition.

From the last lemma we can deduce that the poset $\pmp\left(\Sigma,f\right)$
is a \emph{graded poset}%
\footnote{\emph{\label{fn:graded-poset}}A graded poset is a poset $\left(P,\le\right)$
together with a rank function $\mathrm{rk}:P\to\mathbb{Z}_{\ge0}$,
such that if $x<y$ then $\mathrm{rk}\left(x\right)<\mathrm{rk}\left(y\right)$,
and if $y$ covers $x$ (that is, $x<y$ and there is no $z$ with
$x<z<y$) then $\mathrm{rk}\left(y\right)=\mathrm{rk}\left(x\right)+1$.
We note that the definition in \cite[Section 3.1]{Stanley-book} is
slightly less general.%
}, with rank function $\pmp\left(\Sigma,f\right)\to\mathbb{Z}_{\ge0}$
given by $\left(\sigma,\tau\right)\mapsto\left\Vert \sigma^{-1}\tau\right\Vert $.
Moreover, recall from the proof of Proposition \ref{prop:order-of-contribution-of-(sigma,tau)}
that $\chi\left(\sigma,\tau\right)=\#\left\{ \mathrm{discs\, in}\,\Sigma_{\left(\sigma,\tau\right)}\right\} -2L$.
Among the pairs in $\pmp\left(\Sigma,f\right)$ the Euler characteristic
$\chi\left(\st\right)$ is constant, and thus so is the total number
of discs. The total number of type-$z_{i}$ discs is equal to $L-\left\Vert \sigma^{-1}\tau\right\Vert $,
hence we obtain:
\begin{claim}
\label{claim:num-type-o can serve as rank for surface}The poset $\pmp\left(\Sigma,f\right)$
is graded, with two possible, natural rank functions: either $\left\Vert \sigma^{-1}\tau\right\Vert $,
or the number of type-$o$ discs in $\Sigma_{\left(\sigma,\tau\right)}$.\end{claim}
\begin{rem}
\label{remark:simple rewiring}More generally, a similar argument
as in the proof of Lemma \ref{lem:bp(w)-closed-downwards} shows that
if $\left(\sigma,\tau\right),\left(\sigma',\tau'\right)\in\match\left(\wl\right)^{2}$
and $\left(\sigma',\tau'\right)\preceq\left(\sigma,\tau\right)$,
then $\chi\left(\sigma',\tau'\right)\ge\chi\left(\sigma,\tau\right)$.
If, moreover, $\left(\sigma,\tau\right)$ covers $\left(\sigma',\tau'\right)$,
then $\chi\left(\sigma',\tau'\right)-\chi\left(\st\right)\in\left\{ 0,2\right\} $. 
\end{rem}
The last argument in the proof of Lemma \ref{lem:bp(w)-closed-downwards},
where we made changes to matching-edges in $\Sigma_{\left(\sigma,\tau\right)}$,
can be generalized to the following definition which allows a more
geometric definition of the order $\preceq$ on $\pmp\left(\Sigma,f\right)$.
This equivalent definition will be of great importance in Section
\ref{sec:core-of-result}.
\begin{defn}
\label{def:colored-non-crossing-partition}A partition $P$ of the
matching-edges at the boundary of a disc of $\Sigma_{\left(\sigma,\tau\right)}$
is called a \textbf{colored non-crossing partition}, if 
\begin{itemize}
\item it is colored: every block of $P$ is monochromatic (contains matching-edges
of the same color), and 
\item it is non-crossing: there are no four matching-edges which in cyclic
order are $e_{1},e_{2},e_{3},e_{4}$ and such that $e_{1}$ and $e_{3}$
belong to one block and $e_{2}$ and $e_{4}$ to another.
\end{itemize}
\end{defn}
This is the same as the usual notion of non-crossing partitions (see
\cite[Lecture 9]{NICASPEICHER}), only with the additional constraint
of monochromatic blocks.
\begin{lem}
\label{lem: rewiring from colored-non-crossing partition}Given $\left(\sigma,\tau\right)\in\pmp\left(\Sigma,f\right)$
and a colored non-crossing partition $P$ of a disc (2-cell) $D$
of $\Sigma_{\left(\sigma,\tau\right)}$, we can obtain a new pair
of matchings $\left(\sigma',\tau'\right)\in\pmp\left(\Sigma,f\right)$
by the following procedure: using the orientation on $\partial D$,
match the second endpoint of a matching-edge with the first endpoint
of the following edge in the same block of $P$. Now replace the old
matching-edges along $\partial D$ with the new ones.\end{lem}
\begin{proof}
First, all matching-edges of a fixed color at the boundary of $D$
have the same orientation, so the instructions in the claim indeed
match marked points on $E^{+}$ with marked points on $E^{-}$, and
lead to a new pair $\left(\sigma',\tau'\right)\in\match\left(\wl\right)^{2}$.
It remains to show that $\left(\sigma',\tau'\right)\in\pmp\left(\Sigma,f\right)$,
and we now show this basically follows from the same argument as in
the proof of Lemma \ref{lem:bp(w)-closed-downwards}.

Note that the new matching-edges can be drawn as disjoint arcs inside
$D$: the disjointness can be achieved thanks to $P$ being non-crossing.
By Lemma \ref{lem:joint-boundaries-of-discs-in-surface}, the discs
on the other side of the matching-edges in the same block $B\in P$
are distinct. Thus, after replacing the matching-edges along $\partial D$
with the new ones, the surface is still cut to discs, and so the CW-complex
obtained that way from $\Sigma_{\left(\st\right)}$ is exactly $\Sigma_{\left(\sigma',\tau'\right)}$.
Finally, we can choose $f_{\left(\st\right)}$ so that it is constant
not only on all matching-edges of $\Sigma_{\left(\st\right)}$ but
also on the new matching-edges in $D$. Then, with Lemma \ref{lem:homotopy=00003Dsame_induced_map},
we get that $f_{\left(\st\right)}$ and $f_{\left(\sigma',\tau'\right)}$
are homotopic. We illustrate this in Figure \ref{fig:non-crossing-partition}.
\end{proof}
\begin{figure}[h]
\centering{}\includegraphics[bb=0bp 0bp 350bp 150bp]{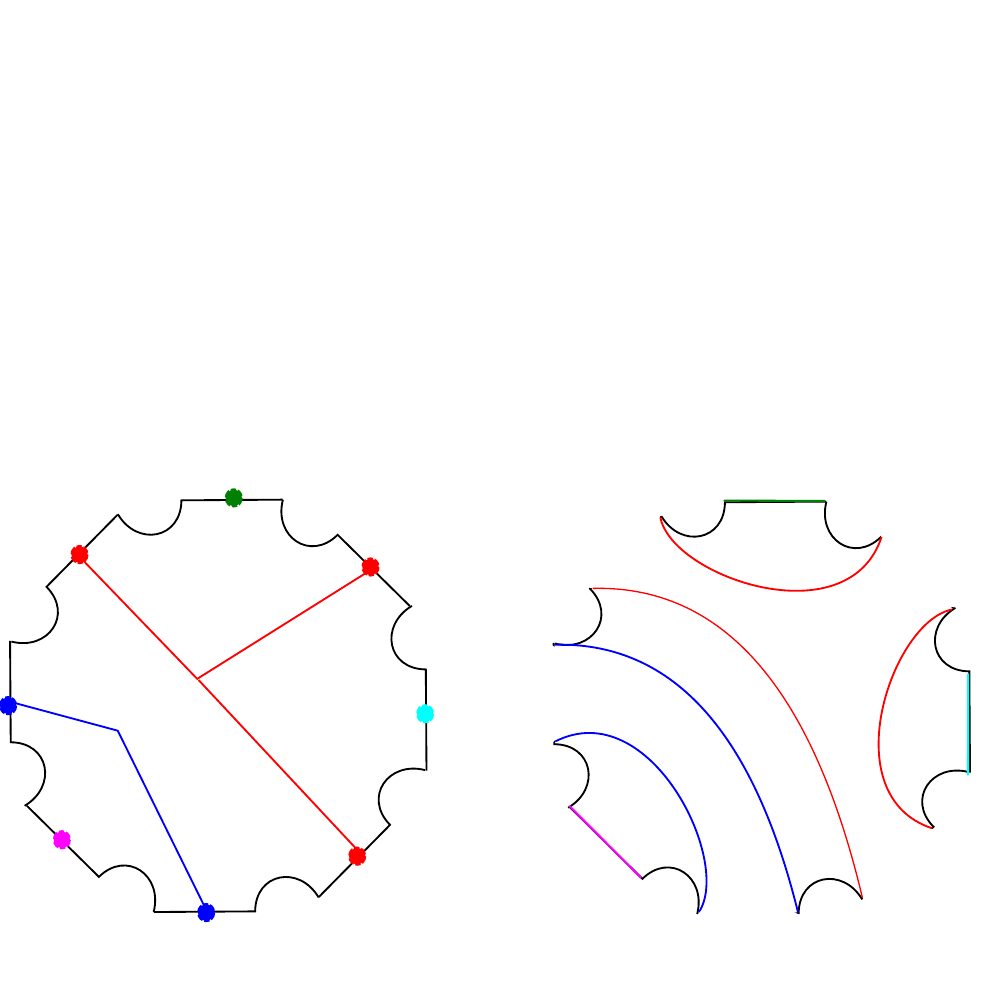}\caption{\label{fig:non-crossing-partition}The figure on the left shows a
non-crossing partition of the eight matching-edges along the boundary
of a disc $D$: every block is marked by a different color. (The matching-edges
in every block need be of the same color of $q_{i}$ or $p_{i}$,
but this is not shown in the figure.) Rewiring the matching-edges
according to this partition results in the figure on the right: the
disc $D$ is split into four smaller discs, and some of its area serves
as {}``corridors'' which merge neighboring discs.}
\end{figure}

\begin{prop}
\label{prop:rewiring-in-PP}Assume that $\left(\sigma',\tau'\right)$
and $\left(\sigma,\tau\right)$ are both in $\pmp\left(\Sigma,f\right)$.
Then the following are equivalent:
\begin{enumerate}
\item $\left(\sigma',\tau'\right)\preceq\left(\sigma,\tau\right)$
\item $\Sigma_{\left(\sigma,\tau\right)}$ can be obtained from $\Sigma_{\left(\sigma',\tau'\right)}$
by a rewiring of matching-edges according to colored non-crossing
partitions in type-$o$ discs. 
\item $\Sigma_{\left(\sigma',\tau'\right)}$ can be obtained from $\Sigma_{\left(\sigma,\tau\right)}$
by a rewiring of matching-edges according to colored non-crossing
partitions in type-$z_{i}$ discs (for all $i$ together). 
\end{enumerate}

Moreover, if indeed $\left(\sigma',\tau'\right)\preceq\left(\sigma,\tau\right)$,
then the set of colored non-crossing partitions in item $2$ (item
$3$) is unique.

\end{prop}
\begin{proof}
The uniqueness of the partitions is obvious. For example, in item
$\left(2\right)$ the partition in every type-$o$ disc can be read
from the pair $\left(\sigma,\tau\right)$, which is given. We now
prove $\left(1\right)\Longleftrightarrow\left(2\right)$, the equivalence
$\left(1\right)\Longleftrightarrow\left(3\right)$ being completely
analogous.

\textbf{$\left(1\right)\Longrightarrow\left(2\right)$:} We show that
if $\left(\sigma',\tau'\right)\preceq\left(\sigma,\tau\right)$ then
there is a rewiring of matching-edges inside type-$o$ discs of $\Sigma_{\left(\sigma',\tau'\right)}$
which results in $\Sigma_{\left(\sigma,\tau\right)}$. It is then
obvious that the rewiring in every type-$o$ disc corresponds to a
colored non-crossing partition of its matching-edges. We prove there
is such rewiring by induction on the difference in ranks $t=\left\Vert \sigma^{-1}\tau\right\Vert -\left\Vert \left(\sigma'\right)^{-1}\tau'\right\Vert $. 

If $t=1$, namely, if $\left(\sigma,\tau\right)$ covers $\left(\sigma',\tau'\right)$,
we repeat the argument in the proof of Lemma \ref{lem:bp(w)-closed-downwards}:
the difference in the $1$-skeletons is exactly in two matching-edges.
These two matching-edges in $\Sigma_{\left(\st\right)}$ must belong
to the same type-$z_{i}$ disc. Hence, by the proof of Lemma \ref{lem:bp(w)-closed-downwards}
and Figure \ref{fig:disc-split}, these the two matching-edges in
$\Sigma_{\left(\sigma',\tau'\right)}$ must belong to the same type-$o$
disc, and we can rewire both of them inside this disc to obtain $\Sigma_{\left(\sigma,\tau\right)}$.

If $t\ge2$, let $\left(\sigma'',\tau''\right)$ be an intermediate
pair which is covered by $\left(\sigma,\tau\right)$. Use the induction
hypothesis to find a rewiring inside type-$o$ discs of $\Sigma_{\left(\sigma',\tau'\right)}$
which gives $\Sigma_{\left(\sigma'',\tau''\right)}$. Of course, we
can now find a rewiring of two matching-edges inside a type-$o$ disc
of $\Sigma_{\left(\sigma'',\tau''\right)}$ which gives $\Sigma_{\left(\sigma,\tau\right)}$.
The crux of the argument is that type-$o$ discs of $\Sigma_{\left(\sigma'',\tau''\right)}$
are \emph{completely contained} inside type-$o$ discs of $\Sigma_{\left(\sigma',\tau'\right)}$,
so the whole rewiring takes places inside type-$o$ discs of $\Sigma_{\left(\sigma',\tau'\right)}$.

$\left(2\right)\Longrightarrow\left(1\right)$: By Lemma \ref{lem: rewiring from colored-non-crossing partition},
we can perform the rewiring at one type-$o$ disc at a time and obtain
a surface corresponding to some pair in $\pmp\left(\Sigma,f\right)$
at each step. Thus it is enough to show this implication if the rewiring
is in a single type-$o$ disc $D$, and by the colored non-crossing
partition $P$.

Let $P_{0},P_{1},\ldots,P_{m}=P$ be a sequence of partitions of the
matching-edges in $D$, each obtained from the former by merging together
two blocks, so that $P_{0}$ consists entirely of singletons. Denote
by $\left(\sigma_{j},\tau_{j}\right)$ the pair of matchings in $\pmp\left(\Sigma,f\right)$
corresponding to the rewiring by $P_{j}$. Now, $\Sigma_{\left(\sigma_{j},\tau_{j}\right)}$
can be obtained from $\Sigma_{\left(\sigma_{j-1},\tau_{j-1}\right)}$
by rewiring a single pair of matching-edges inside a type-$o$ disc.
Thus, it suffices to show that in this case we go up in the poset
$\pmp\left(\Sigma,f\right)$. Without loss of generality, assume that
this single pair of matching-edges is of color $q_{i}$. Thus, $\sigma_{j}=\sigma_{j-1}$
and $\tau_{j}^{-1}\tau_{j-1}$ is a transposition. So the pairs $\left(\sigma_{j-1},\tau_{j-1}\right)$
and $\left(\sigma_{j},\tau_{j}\right)$ are necessarily comparable,
and indeed $\left(\sigma_{j-1},\tau_{j-1}\right)\prec\left(\sigma_{j},\tau_{j}\right)$
because the number of type-$o$ discs increases in this rewiring.
\end{proof}
Before stating the main theorem of this section we need one more simple
lemma:
\begin{lem}
\label{lem:sum-of-mobius-over-a-pair-is-1}Let $\sigma_{0},\tau_{0}\in S_{L}$.
Then
\[
\sum_{\left(\sigma,\tau\right)\preceq\left(\sigma_{0},\tau_{0}\right)}\moeb\left(\sigma^{-1}\tau\right)=1.
\]
\end{lem}
\begin{proof}
By the definition of the order $\preceq$ on pairs, $\left(\sigma,\tau\right)\preceq\left(\sigma_{0},\tau_{0}\right)$
if and only if $\mathrm{id}\preceq\sigma_{0}^{-1}\sigma\preceq\sigma_{0}^{-1}\tau\preceq\sigma_{0}^{-1}\tau_{0}$
in $S_{L}$. By Proposition \ref{prop:mobius function} and the definition
\eqref{eq:mobius-definition} of the Möbius function $\mu$ of the
poset $\left(S_{L},\preceq\right)$,

\begin{eqnarray*}
\sum_{\left(\sigma,\tau\right)\preceq\left(\sigma_{0},\tau_{0}\right)}\moeb\left(\sigma^{-1}\tau\right) & = & \sum_{\sigma,\tau:\,\mathrm{id}\preceq\sigma_{0}^{-1}\sigma\preceq\sigma_{0}^{-1}\tau\preceq\sigma_{0}^{-1}\tau_{0}}\moeb\left(\sigma^{-1}\tau\right)\\
 & = & \sum_{\sigma,\tau:\,\mathrm{id}\preceq\sigma\preceq\tau\preceq\sigma_{0}^{-1}\tau_{0}}\moeb\left(\sigma^{-1}\tau\right)=\sum_{\sigma:\,\mathrm{id}\preceq\sigma\preceq\sigma_{0}^{-1}\tau_{0}}\left(\sum_{\tau:\,\sigma\preceq\tau\preceq\sigma_{0}^{-1}\tau_{0}}\moeb\left(\sigma^{-1}\tau\right)\right)\\
 & = & \sum_{\sigma:\,\mathrm{id}\preceq\sigma\preceq\sigma_{0}^{-1}\tau_{0}}\delta_{\sigma,\sigma_{0}^{-1}\tau_{0}}=1.
\end{eqnarray*}
\end{proof}
\begin{defn}
\label{def:SC-of-poset}\cite[Section 3.8]{Stanley-book} For every
locally finite poset%
\footnote{See footnote on Page \pageref{fn:locally-finite-posets}.%
} $\left(P,\le\right)$ there is an associated simplicial complex,
the vertices of which are the elements of $P$ and the simplices are
the chains. That is, $x_{1},\ldots,x_{k}\in P$ form a simplex if
and only if, after possible rearrangement, $x_{1}<x_{2}<\ldots<x_{k}$.
We let \marginpar{$\left|P\right|$}\textbf{$\left|P\right|$} denote
the \emph{geometric realization} of this simplicial complex%
\footnote{The space $\left|P\right|$ is a topological space with the following
topology: every simplex $s$ has the Euclidean topology. A general
set $A\subseteq\left|P\right|$ is closed if and only if $A\cap s$
is closed in $s$ for every simplex $s$.%
}.
\end{defn}
The following theorem shows that the Euler characteristic of the simplicial
complex\linebreak{}
$\left|\pmp\left(\Sigma,f\right)\right|$ captures the leading coefficient
of the contribution of the pairs of matchings in $\pmp\left(\Sigma,f\right)$
to $\trwl\left(n\right)$ from Corollary \ref{cor:trw(n) leading exponent}.
Recall that $\chi\left(\right)$ marks Euler characteristic.
\begin{thm}
\label{thm:Euler-char-of-poset}If $\left(\Sigma,f\right)$ is admissible
for $\wl$ and incompressible, then 
\[
\sum_{\left(\sigma,\tau\right)\in\pmp\left(\Sigma,f\right)}\moeb\left(\sigma^{-1}\tau\right)=\chi\left(\left|\pmp\left(\Sigma,f\right)\right|\right).
\]
In particular, 
\begin{equation}
\trwl\left(n\right)=n^{\ch\left(\wl\right)}\left[\sum_{\left[\left(\Sigma,f\right)\right]\in\sol\left(\wl\right)}\chi\left(\left|\pmp\left(\Sigma,f\right)\right|\right)\right]+O\left(n^{\ch\left(\wl\right)-2}\right).\label{eq:trwl as sum over chi of pmp}
\end{equation}
\end{thm}
\begin{proof}
Recall that for a simplicial complex $\Delta$, the Euler characteristic
is 
\[
\chi\left(\Delta\right)=\sum_{\emptyset\ne s}\left(-1\right)^{\dim s},
\]
the sum being over all non-empty simplices in $\Delta$, and $\dim s=\left|s\right|-1$.
We prove the statement for any poset $P$ of pairs of bijections with
the downward-closure property elaborated in Lemma \ref{lem:bp(w)-closed-downwards}.
It is enough to show that for every pair $\left(\sigma_{0},\tau_{0}\right)$
we have 
\begin{equation}
\moeb\left(\sigma_{0}^{-1}\tau_{0}\right)=\sum_{s\subseteq P:\,\max s=\left(\sigma_{0},\tau_{0}\right)}\left(-1\right)^{\dim s},\label{eq:mob=00003Dsum-of-faces-containing-elem-and-elems-above}
\end{equation}
the sum being over all chains in $P$ with maximal element $\left(\sigma_{0},\tau_{0}\right)$.
Indeed, if \eqref{eq:mob=00003Dsum-of-faces-containing-elem-and-elems-above}
holds, then 
\[
\sum_{\left(\sigma_{0},\tau_{0}\right)\in P}\moeb\left(\sigma_{0}^{-1}\tau_{0}\right)=\sum_{\left(\sigma_{0},\tau_{0}\right)\in P}\left[\sum_{s\subseteq P:\,\max s=\left(\sigma_{0},\tau_{0}\right)}\left(-1\right)^{\dim s}\right]=\sum_{\emptyset\ne s\subseteq P}\left(-1\right)^{\dim s}=\chi\left(\left|P\right|\right).
\]

So we only need to prove \eqref{eq:mob=00003Dsum-of-faces-containing-elem-and-elems-above}.
Denote by $(-\infty,\left(\sigma_{0},\tau_{0}\right)]{}_{\preceq}$
all pairs below (or equal to) $\left(\sigma_{0},\tau_{0}\right)$
according to $\preceq$. We prove \eqref{eq:mob=00003Dsum-of-faces-containing-elem-and-elems-above}
by induction on the size $t$ of $(-\infty,\left(\sigma_{0},\tau_{0}\right)]{}_{\preceq}$.
It clearly holds for $t=1$, in which case necessarily $\sigma_{0}=\tau_{0}$
by the downward-closeness property. For $t\ge2$, note the one-to-one
correspondence among the chains in $\big(-\infty,\left(\sigma_{0},\tau_{0}\right)\big]_{\preceq}$
between those containing $\left(\sigma_{0},\tau_{0}\right)$ and those
not containing it. This correspondence is given by $s\mapsto s\setminus\left\{ \left(\sigma_{0},\tau_{0}\right)\right\} $.
Now,

\begin{eqnarray*}
\sum_{s\subseteq P:\,\max s=\left(\sigma_{0},\tau_{0}\right)}\left(-1\right)^{\dim s} & = & \left(\sum_{s\subseteq P:\,\max s=\left(\sigma_{0},\tau_{0}\right)}\left[\left(-1\right)^{\dim s}+\left(-1\right)^{\dim\left(s\setminus\left\{ \left(\sigma_{0},\tau_{0}\right)\right\} \right)}\right]\right)\\
 &  & -\left(\left(-1\right)^{\dim\emptyset}+\sum_{\emptyset\ne s\subseteq P:\,\max s\prec\left(\sigma_{0},\tau_{0}\right)}\left(-1\right)^{\dim s}\right)\\
 & = & 0-\left(-1+\sum_{\left(\sigma,\tau\right)\prec\left(\sigma_{0},\tau_{0}\right)}\sum_{s\subseteq P:\,\max s=\left(\sigma,\tau\right)}\left(-1\right)^{\dim s}\right)\\
 & \overset{\left(1\right)}{=} & 1-\sum_{\left(\sigma,\tau\right)\prec\left(\sigma_{0},\tau_{0}\right)}\moeb\left(\sigma^{-1}\tau\right)\overset{\left(2\right)}{=}\moeb\left(\sigma_{0}^{-1}\tau_{0}\right),
\end{eqnarray*}
where in $\overset{\left(1\right)}{=}$ we used the induction hypothesis
for smaller values of $t$, and in $\overset{\left(2\right)}{=}$
we used Lemma \ref{lem:sum-of-mobius-over-a-pair-is-1}.
\end{proof}
As an example, consider again $w=\left[x,y\right]\left[x,z\right]$.\label{[x,y][x,z] - simplicial complex}
We already described above (in Page \pageref{[x,y][x,z] - PMP}) the
poset $\pmp\left(\Sigma,f\right)$ of the only equivalence class in
this case. The associated simplicial complex is one dimensional with
the shape of a $4$-cycle. Topologically, this is simply $S^{1}$,
and the Euler characteristic is $0$. This agrees, of course, with
the direct computation carried out in Example \ref{example:[x,y][x,z] leading term 0}.

In the next section we shall prove the following:
\begin{thm}
\label{thm:pmp is K(G,1)}Let $\left(\Sigma,f\right)$ be admissible
for $\wl$ and incompressible. As above, denote by $\tilde{f}$ the
homotopy class of $f$, relative $\partial\Sigma$. Then $\left|\pmp\left(\Sigma,f\right)\right|$
is a $\mathrm{K}\left(G,1\right)$-space for $G=\mathrm{Stab}_{\mathrm{MCG}\left(\Sigma\right)}\left(\tilde{f}\right)$.
\end{thm}
As explained in Section \ref{sec:Introduction}, in order to prove
this theorem we show in the next section that $\left(i\right)$ $\left|\pmp\left(\Sigma,f\right)\right|$
is (path) connected, $\left(ii\right)$ its fundamental group is isomorphic
to $\mathrm{Stab}_{\mathrm{MCG}\left(\Sigma\right)}\left(\tilde{f}\right)$,
and $\left(iii\right)$ its universal cover in contractible. 

Our main theorems now follow immediately from Theorem \ref{thm:pmp is K(G,1)}:
Theorem \ref{thm:K(G,1) for incompressible} follows as $\left|\pmp\left(\Sigma,f\right)\right|$
is a finite simplicial complex, and Theorem \ref{thm:main - general}
follows using \eqref{eq:trwl as sum over chi of pmp}.

\section{The Arc Poset\label{sec:core-of-result}}

In this section we construct yet another poset related to some $\left(\Sigma,f\right)$
(we assume $\left(\Sigma,f\right)$ is admissible for $\wl$ and incompressible
throughout this section). This poset is named the {}``arc poset''
of $\left(\Sigma,f\right)$, and its elements consist of sets of arcs
on the surface $\Sigma$. Each one of them looks like a specific geometric
realization of the matching-edges in $\Sigma_{\left(\st\right)}$
for some $\left(\st\right)\in\pmp\left(\Sigma,f\right)$. However,
in the arc poset we let $\mathrm{MCG}\left(\Sigma\right)$ act freely.
Namely, different sets of arcs representing the same pair of matchings
will constitute different elements in the arc poset as long as they
differ by the action of a non-trivial element of $\mathrm{MCG}\left(\Sigma\right)$.
As we show below, the connected components of the arc poset shall
serve as universal cover of $\left|\pmp\left(\Sigma,f\right)\right|$
and enable us to prove Theorem \ref{thm:pmp is K(G,1)}.

\subsection{Arc systems\label{sub:Arc-systems}}

Recall that if $\left(\Sigma,f\right)$ is admissible for $\wl$,
then the $\ell$ boundary components of $\Sigma$ are identified with
$S^{1}\left(w_{1}\right),\ldots,S^{1}\left(w_{\ell}\right)$ and have
$4L$ marked points on them which spell out $w_{1},\ldots,w_{\ell}$
(consult also the glossary on Page \pageref{sec:Glossary}).
\begin{defn}
\label{def:arc-system}Let $\left(\Sigma,f\right)$ be admissible
for $\wl$ and incompressible. An\textbf{\emph{ }}\textbf{arc system
}for $\left(\Sigma,f\right)$ is\emph{ }an ambient isotopy (relative
to the boundary $\partial\Sigma$) class of sets of $2L$ disjoint
arcs embedded in $\Sigma$, which meet $\partial\Sigma$ only at their
endpoints and so that the matching they induce on the $4L$ marked
points in $\partial\Sigma$ is identical to the one induced by some
$\left(\st\right)\in\pmp\left(\Sigma,f\right)$. 

We denote by $\left[\left\{ \alpha_{1},\ldots,\alpha_{2L}\right\} \right]$\marginpar{$\left[\left\{ \alpha_{1},\ldots,\alpha_{2L}\right\} \right]$}
the arc system with representative $\left\{ \alpha_{1},\ldots,\alpha_{2L}\right\} $.
We also denote by $\left(\sigma_{\aa},\tau_{\aa}\right)$\marginpar{$\sigma_{\aa},\tau_{\aa}$}
the pair of matchings in $\pmp\left(\Sigma,f\right)$ associated with
the arc system $\aa=[\{\alpha_{1},\ldots,\alpha_{2L}\}]$.
\end{defn}
Note, in particular, that an arc system for $\left(\Sigma,f\right)$
must connect $p_{i}^{+}$-points in $\partial\Sigma$ to $p_{i}^{-}$-points,
and $q_{i}^{+}$-points to $q_{i}^{-}$-points, for every $i\in\left[r\right]$.
We call an arc a \textbf{$p_{i}$-arc} (a \textbf{$q_{i}$-arc}, respectively)
if it connects a $p_{i}^{+}$-point with a $p_{i}^{-}$-point (a $q_{i}^{+}$-point
with a $q_{i}^{-}$-point, respectively). We think of the arcs as
colored by $\left\{ p_{i},q_{i}\,\middle|\, i\in\left[r\right]\right\} $.
\begin{claim}
An arc system $\aa$ for $\left(\Sigma,f\right)$ cuts $\Sigma$ into
discs.\end{claim}
\begin{proof}
By definition, the matching-edges in $\Sigma_{\left(\sigma_{\aa},\tau_{\aa}\right)}$
cut $\Sigma_{\left(\sigma_{\aa},\tau_{\aa}\right)}$ into discs. Since
$\Sigma_{\left(\sigma_{\aa},\tau_{\aa}\right)}\cong\Sigma$, a simple
Euler characteristic argument shows the arcs in $\aa$ must also cut
$\Sigma$ into discs: otherwise, the Euler characteristic is too small.
\end{proof}
We can therefore think of $\Sigma$ with the arc system $\aa$ as
a CW-complex which is isomorphic to the CW-complex $\Sigma_{\left(\sigma_{\aa},\tau_{\aa}\right)}$.
We let $\Sigma_{\aa}$\marginpar{$\Sigma_{\aa}$} denote this CW-complex.
We extend some of the notions we had for $\Sigma_{\left(\sigma_{\aa},\tau_{\aa}\right)}$
to $\Sigma_{\aa}$: As in Claim \ref{claim:properties-of-perms-surface},
every disc $D$ in $\Sigma_{\aa}$ is either a \textbf{type-$o$ disc}
(if $\partial D$ contains $o$-points, i.e.~points from $f_{w_{1}}^{-1}\left(o\right)\cup\ldots\cup f_{w_{\ell}}^{-1}\left(o\right)$)
or a \textbf{type-$z_{i}$ disc }(if $\partial D$ contains $z_{i}$-points
for some $i$, i.e.~points from $f_{w_{1}}^{-1}\left(z_{i}\right)\cup\ldots\cup f_{w_{\ell}}^{-1}\left(z_{i}\right)$
). This is illustrated in Figure \ref{fig:first-arc-systems}.

We also define a (homotopy class of a) map \marginpar{$f_{\aa}$}$f_{\aa}\colon\Sigma\to\wedger$
as in Definition \ref{def:f_(sigma,tau)}: we let $f_{\aa}$ extend
$f_{w_{1}},\ldots,f_{w_{\ell}}$ on $\partial\Sigma$, and be constant
on the arcs. There is then a unique way (up to homotopy) to extend
$f_{\aa}$ in the discs of $\Sigma_{\aa}$. Evidently, $\left(\Sigma,f_{\aa}\right)\sim\left(\Sigma_{\left(\sigma_{\aa},\tau_{\aa}\right)},f_{\left(\sigma_{\aa},\tau_{\aa}\right)}\right)\sim\left(\Sigma,f\right)$
and, in particular, $\left(\Sigma,f_{\aa}\right)$ is admissible for
$\wl$. Finally, as in Lemma \ref{lem:joint-boundaries-of-discs-in-surface},
two bordering discs of $\Sigma_{\aa}$, which must be one of type-$o$
and the other of type-$z_{i}$, have at most $2$ common arcs at their
boundaries: at most one $p_{i}$-arc and at most $q_{i}$-arc.

\begin{figure}[h]
\centering{}\includegraphics[bb=0bp 0bp 350bp 130bp,scale=1.2]{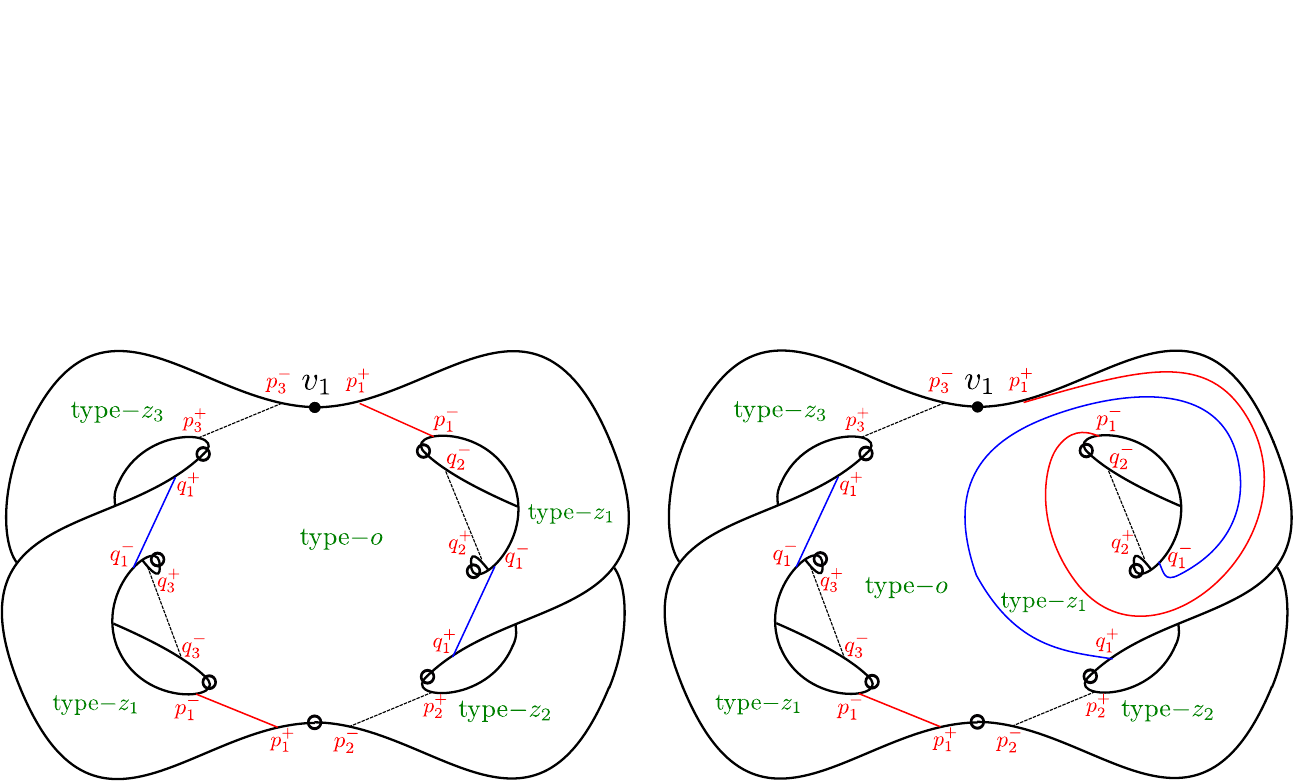}\caption{\label{fig:first-arc-systems}Two arc systems drawn on $\Sigma$ of
genus $2$ and with one boundary component identified with $S^{1}\left(w\right)$
for the word $w=\left[x,y\right]\left[x,z\right]=\left[x_{1},x_{2}\right]\left[x_{1},x_{3}\right]$.
The $p_{1}$-arcs are red, the $q_{1}$-arcs are blue and all the
others are drawn in black. These two arc systems are distinct yet
induce the same pair of matchings. Each of the associated CW-complexes
has five discs: one of type-$o$, two of type-$z_{1}$, one of type-$z_{2}$
and one of type-$z_{3}$. }
\end{figure}

The following useful claim is evident from the definition of $f_{\aa}$:
\begin{claim}
\label{claim:f_aa on arcs}Let $\aa$ be an arc system for $\left(\Sigma,f\right)$,
and let $\gamma$ be an oriented arc in $\Sigma$ with endpoints in
$v_{1},\ldots,v_{\ell}$. Then the word $\left[f_{\aa}\left(\gamma\right)\right]\in\F_{r}$
can be computed as follows: fix a representative $\left\{ \alpha_{1},\ldots,\alpha_{2L}\right\} $
of $\aa$ which meets $\gamma$ transversely. Now follow the intersections
of $\gamma$ with the $\alpha_{i}$'s:
\begin{itemize}
\item Whenever $\gamma$ enters a type-$z_{i}$ disc through a $p_{i}$-arc
and leaves through a $q_{i}$-arc, write $x_{i}$.
\item Whenever $\gamma$ enters a type-$z_{i}$ disc through a $q_{i}$-arc
and leaves through a $p_{i}$-arc, write $x_{i}^{-1}$.
\item Whenever $\gamma$ enters and leaves a type-$z_{i}$ disc through
$p{}_{i}$-arcs, or enter and leaves through $q_{i}$-arcs, write
nothing.
\end{itemize}

The final result is $\left[f_{\aa}\left(\gamma\right)\right]$, albeit
not necessarily in reduced form.

\end{claim}

\subsection{The Arc Poset of $\left(\Sigma,f\right)$\label{sub:The-Arc-Poset}}
\begin{defn}
\label{def:arc-poset}Let $\left(\Sigma,f\right)$ be admissible for
$\wl$ and incompressible. The \textbf{arc poset of $\mathbf{\left(\Sigma,f\right)}$},
denoted $\ap\left(\Sigma,f\right)$\marginpar{$\ap\left(\Sigma,f\right)$},
consists of the set of all arc systems for $\left(\Sigma,f\right)$
together with the partial order $\preceq$ \marginpar{$\preceq$}defined
by 
\[
\aa\preceq\bb
\]
whenever, for some representatives of $\aa$ and $\bb$, the arcs
of $\bb$ are embedded entirely \uline{inside type-$o$ discs}
of $\aa$.\end{defn}
\begin{rem}
\label{remark:about-order-in-arc-poset}
\begin{enumerate}
\item The type-$o$ discs in the definition can be taken to be either open
or closed (although the endpoints of the arcs, of course, are always
contained in their boundaries). However, using closed discs is more
convenient: some of the arcs can be left unchanged when moving from
$\aa$ to $\bb$.
\item Of course, if $\aa\preceq\bb$ then for every representative $\left\{ \alpha_{1},\ldots,\alpha_{2L}\right\} $
of $\aa$ there is a representative $\left\{ \beta_{1},\ldots,\beta_{2L}\right\} $
of $\bb$ with arcs embedded inside the type-$o$ discs defined by
$\left\{ \alpha_{1},\ldots,\alpha_{2L}\right\} $. 
\item This rewiring of arcs is completely analogous to the one in Proposition
\ref{prop:rewiring-in-PP}. As we explained there, if $\aa\preceq\bb$
then this rewiring corresponds to a unique set of colored non-crossing
partitions of the arcs of $\aa$ inside its type-$o$ discs.
\item An equivalent definition for the order $\preceq$ in $\ap\left(\Sigma,f\right)$
is the following: $\aa\preceq\bb$ if and only if for some representatives
of $\aa$ and $\bb$, the arcs of $\aa$ are embedded entirely \uline{inside
type-$z_{i}$ discs} of $\bb$ (union of type-$z_{i}$ discs for all
$i$).
\end{enumerate}
\end{rem}
The following claim says, in particular, that the partial order we
just defined is indeed an order:
\begin{claim}
\label{claim:AP-is-graded}
\begin{enumerate}
\item If $\aa\preceq\bb$ and $\aa\ne\bb$ then the number of type-$o$
discs in $\bb$ is strictly larger.
\item Moreover, the number of type-$o$ discs can serve as a rank for the
poset $\ap\left(\Sigma,f\right)$, which turns it into a \emph{graded
poset}%
\footnote{See footnote on Page \pageref{fn:graded-poset}.%
}\emph{.}
\item If $\aa\preceq\bb$ and $\bb\preceq\cc$ then $\aa\preceq\cc$.
\end{enumerate}
\end{claim}
\begin{proof}
$\left(1\right)$ Let $D$ be a type-$o$ disc of $\aa$ where new
arcs of $\bb$ are introduced (namely, where the non-crossing partition
is non-trivial). With the new arcs instead of the old ones, at least
two of the regions of $D$ are now disjoint type-$o$ discs of $\bb$,
thus strictly increasing the total number of type-$o$ discs. (The
other effect is that the other areas in $D$ now serve as {}``corridors'',
merging together several neighboring type-$z_{i}$ discs, as in Figure
\ref{fig:non-crossing-partition}.)

$\left(2\right)$ One needs to show that if $\aa$ is covered by $\bb$
(see footnote on Page \ref{fn:graded-poset}), then $\bb$ has exactly
one more type-$o$ disc than $\aa$. Let $\left\{ \alpha_{1},\ldots,\alpha_{2L}\right\} $
and $\left\{ \beta_{1},\ldots,\beta_{2L}\right\} $ be representatives
with the $\beta_{i}$'s contained in the type-$o$ discs defined by
the $\alpha_{i}$'s. Assume without loss of generality that $\beta_{1}$
is a genuine new arc (does not share the same two endpoints as any
of the $\alpha_{i}$'s), which is contained inside the type-$o$ disc
$D$ and meets at its two endpoints $\alpha_{1}$ and $\alpha_{2}$.
It is evident that we can draw an arc $\beta'$ embedded in $D$ and
disjoint from all the (interiors of) $\beta_{1},\ldots,\beta_{2L}$,
which connects the other endpoints of $\alpha_{1}$ and $\alpha_{2}$.
Then $\cc=\left[\left\{ \beta_{1},\beta',\alpha_{3},\ldots,\alpha_{2L}\right\} \right]$
clearly satisfies $\aa\prec\cc\preceq\bb$, and by the covering assumption,
$\cc=\bb$. The number of type-$o$ discs in $\cc$ is clearly one
larger than in $\aa$.

$\left(3\right)$ This is true by an argument similar to the one in
the proof of Proposition \ref{prop:rewiring-in-PP}: if $\beta_{1},\ldots,\beta_{2L}$
are contained inside type-$o$ discs defined by $\left\{ \alpha_{1},\ldots,\alpha_{2L}\right\} $,
then the union of type-$o$ discs associated with $\left\{ \beta_{1},\ldots,\beta_{2L}\right\} $
is contained in the union of type-$o$ discs associated with $\left\{ \alpha_{1},\ldots,\alpha_{2L}\right\} $.
Thus, if $\gamma_{1},\ldots,\gamma_{2L}$ are contained inside type-$o$
discs defined by $\left\{ \beta_{1},\ldots,\beta_{2L}\right\} $,
they are also contained inside type-$o$ discs defined by $\left\{ \alpha_{1},\ldots,\alpha_{2L}\right\} $.\end{proof}
\begin{prop}
\label{prop:graded-poset-morphism AP to PMP}The map $\Psi:\ap\left(\Sigma,f\right)\to\pmp\left(\Sigma,f\right)$
defined by $\aa\mapsto\left(\sigma_{\aa},\tau_{\aa}\right)$ is a
graded poset surjective morphism%
\footnote{For our cause, a map $\varphi\colon\left(P_{1},\le\right)\to\left(P_{2},\le\right)$
between two graded posets is a graded-poset morphism if it preserves
the order ($x\le y\,\Rightarrow\,\varphi\left(x\right)\le\varphi\left(y\right)$)
and preserves the rank up to a constant shift: $\mathrm{rank}\left(\varphi\left(x\right)\right)=\mathrm{rank}\left(x\right)+c_{0}$.%
}.\end{prop}
\begin{proof}
Let $\aa\preceq\bb$ in $\ap\left(\Sigma,f\right)$ and let $\left\{ \alpha_{1},\ldots,\alpha_{2L}\right\} $
and $\left\{ \beta_{1},\ldots,\beta_{2L}\right\} $ be representatives
so that $\beta_{1},\ldots,\beta_{2L}$ are embedded inside the type-$o$
discs defined by $\left\{ \alpha_{1},\ldots,\alpha_{2L}\right\} $.
Using the isomorphism of CW-complexes $\Sigma_{\aa}\cong\Sigma_{\left(\sigma_{\aa},\tau_{\aa}\right)}$
we can use the same rewiring of the arcs inside type-$o$ discs in
$\Sigma_{\aa}$, to get a rewiring of matching-edges inside type-$o$
discs of $\Sigma_{\left(\sigma_{\aa},\tau_{\aa}\right)}$. The resulting
CW-complex is $\Sigma_{(\sigma_{\bb},\tau_{\bb})}$. By Proposition
\ref{prop:rewiring-in-PP}, this means that $\left(\sigma_{\aa},\tau_{\aa}\right)\preceq(\sigma_{\bb},\tau_{\bb})$,
hence $\Psi$ is order preserving. Since the number of type-$o$ discs
can serve as a rank for both posets (Claims \ref{claim:num-type-o can serve as rank for surface}
and \ref{claim:AP-is-graded}), $\Psi$ is a graded-poset morphism.
It is surjective because given $\left(\st\right)\in\pmp\left(\Sigma,f\right)$,
the homeomorphism $\Sigma_{\left(\st\right)}\cong\Sigma$ which yields
the equivalence $\left(\Sigma_{\left(\st\right)},f_{\left(\st\right)}\right)\sim\left(\Sigma,f\right)$
can map the matching-edges in $\Sigma_{\left(\st\right)}$ to a valid
arc system for $\left(\Sigma,f\right)$, and this system is mapped
by $\Psi$ to $\left(\st\right)$.
\end{proof}
As before, we denote by \marginpar{$\left|\ap\left(\Sigma,f\right)\right|$}$\left|\ap\left(\Sigma,f\right)\right|$
the (geometric realization of the) simplicial complex associated with
$\ap\left(\Sigma,f\right)$ (see Definition \ref{def:SC-of-poset}). 

\medskip{}

\textit{\emph{Recall $\MCG\left(\Sigma\right)$}}, the mapping class
group of $\Sigma$ defined on Page \pageref{MCG(Sigma)}. Clearly,
the action of homeomorphisms of $\Sigma$ relative $\partial\Sigma$
on sets of arcs $\{\alpha_{1},\ldots,\alpha_{2L}\}$ as in Definition
\ref{def:arc-system} descends to an action of $\MCG(\Sigma)$ on
their isotopy classes, namely, on arc systems. In the following theorem
we analyze this action:
\begin{thm}
\label{thm:MCG-action-on-AP}
\begin{enumerate}
\item \label{enu:MCG action on AP}The action $\MCG\left(\Sigma\right)\curvearrowright\ap\left(\Sigma,f\right)$
is a graded-poset free action%
\footnote{A group action is said to be a graded-poset action if is order-preserving
and rank-preserving.%
}. The quotient is isomorphic to $\pmp\left(\Sigma,f\right)$ as a
graded poset.
\item \label{enu:MCG action on |AP|}The action $\MCG\left(\Sigma\right)\curvearrowright\left|\ap\left(\Sigma,f\right)\right|$
is a covering space action%
\footnote{Namely, every point in $\left|\ap\left(\Sigma,f\right)\right|$ has
a neighborhood $U$ so that $g.U\cap U=\emptyset$ for every $\mathrm{id}\ne g\in\MCG\left(\Sigma\right)$.%
}. The quotient is isomorphic to $\left|\pmp\left(\Sigma,f\right)\right|$
as a simplicial complex.
\end{enumerate}
\end{thm}
\begin{rem}
\label{remark:need-for-regular-action}Item \ref{enu:MCG action on |AP|}
of Theorem \ref{thm:MCG-action-on-AP} does not automatically follow
from item \ref{enu:MCG action on AP}. Consider, for example, the
poset $P=\left\{ x_{1},x_{2},y_{1},y_{2}\right\} $ with order $x_{i}\prec y_{j}$
for every $ $$i$ and $j$, and the action of $G=\nicefrac{\mathbb{Z}}{2\mathbb{Z}}$
on $P$ by swapping $x_{1}$ with $x_{2}$ and $y_{1}$ with $y_{2}$.
Whereas $\nicefrac{P}{G}$ is the poset $\left\{ x\prec y\right\} $
and $\left|\nicefrac{P}{G}\right|$ consists of two vertices and an
edge connecting them, the quotient $\nicefrac{\left|P\right|}{G}$
consists of two vertices with \emph{two} edges connecting them, and
is not even a simplicial complex. See Appendix \ref{sub:Regular G-complexes}
for more details.\end{rem}
\begin{proof}
\textbf{Item \ref{enu:MCG action on AP}: }It is clear that the action
of $\MCG\left(\Sigma\right)$ on $\ap\left(\Sigma,f\right)$ preserves
the number of discs of each type, which shows it preserves the rank
of the elements. It is also clear that the action commutes with rewiring
of arcs inside type-$o$ discs, which shows it is order-preserving.
Assume that $\left[\varphi\right]\in\MCG\left(\Sigma\right)$ fixes
$\aa=\left[\left\{ \alpha_{1},\ldots,\alpha_{2L}\right\} \right]\in\ap\left(\Sigma,f\right)$.
Since $\left[\varphi\right]$ and $\aa$ are defined up to $\mathrm{Homeo}_{0}\left(\Sigma\right)$,
we can assume $\varphi\in\mathrm{Homeo}_{\delta}\left(\Sigma\right)$
fixes $\partial\Sigma\cup\alpha_{1}\cup\ldots\cup\alpha_{2L}$ pointwise.
Because the boundary of every disc $D$ in $\Sigma$ contains segments
from $\partial\Sigma$, the homeomorphism $\varphi$ maps $D$ to
itself, and is the identity on $\partial D$. But $\MCG\left(D\right)$
is trivial (by the Alexander Lemma, e.g.~\cite[Lemma 2.1]{FM}),
and so $\varphi\Big|_{D}$ is isotopic (inside $D$, relative to $\partial D$)
to $\id\Big|_{D}$. Thus $\varphi$ is isotopic to the identity in
the whole of $\Sigma$, and so $\left[\varphi\right]$ is trivial.
This proves the action $\MCG\left(\Sigma\right)\curvearrowright\ap\left(\Sigma,f\right)$
is free. 

To see the quotient is $\pmp\left(\Sigma,f\right)$, we need to show
a correspondence between the orbits of the action and the elements
of $\pmp\left(\Sigma,f\right)$. Note first that $\Psi\left(\aa\right)=\left(\sigma_{\aa},\tau_{\aa}\right)$
only depends on the endpoints of the arcs which sit at the boundary
of $\Sigma$, and the elements of $\MCG\left(\Sigma\right)$ fix the
boundary pointwise. Thus the action commutes with $\Psi$. On the
other hand, if $\Psi\left(\aa\right)=\Psi(\bb)$, then the isomorphisms
of CW-complexes $\varphi_{\aa}\colon\Sigma_{\aa}\overset{\cong}{\to}\Sigma_{\left(\sigma_{\aa},\tau_{\aa}\right)}$
and $\varphi_{\bb}\colon\Sigma_{\bb}\overset{\cong}{\to}\Sigma_{(\sigma_{\bb},\tau_{\bb})}=\Sigma_{\left(\sigma_{\aa},\tau_{\aa}\right)}$
satisfy that $[\varphi_{\bb}^{-1}\circ\varphi_{\aa}]\in\MCG\left(\Sigma\right)$
maps $\aa$ to $\bb$. So, indeed, the orbits of the action $\MCG\left(\Sigma\right)\curvearrowright\ap\left(\Sigma,f\right)$
correspond to the elements of $\pmp\left(\Sigma,f\right)$. That $\nicefrac{\ap\left(\Sigma,f\right)}{\MCG\left(\Sigma\right)}\cong\pmp\left(\Sigma,f\right)$
is an isomorphism of graded-posets now follows from the fact that
$\Psi$ is a graded-poset morphism (which is the content of Proposition
\ref{prop:graded-poset-morphism AP to PMP}).

\medskip{}

\textbf{Item \ref{enu:MCG action on |AP|}:} A simplicial action of
a group $G$ on (the geometric realization of) a simplicial complex
$K$ is a covering space action if and only if the action is free:
there is clearly a neighborhood $U_{x}$ for every point $x$ such
that if $g.x\ne x$ then $g.U_{x}\cap U_{x}=\emptyset$ (take $U_{x}$
that does not intersect any closed simplices in the barycentric subdivision
of $K$ which do not contain $x$). In our case, the freeness of the
action $\MCG\left(\Sigma\right)\curvearrowright\left|\ap\left(\Sigma,f\right)\right|$
on the vertices is proved in item \ref{enu:MCG action on AP}. Since
the action preserves ranks, it cannot mix different vertices of the
same simplex, so if $g.s=s$ for some simplex $s$ and $g\in\MCG\left(\Sigma\right)$,
then necessarily $g$ fixes the vertices of $s$, hence $g=\id$.
So the action is free on all points. 

To see that $\nicefrac{\left|\ap\left(\Sigma,f\right)\right|}{\MCG\left(\Sigma\right)}\cong\left|\nicefrac{\ap\left(\Sigma,f\right)}{\MCG\left(\Sigma\right)}\right|$,
we use Corollary \ref{cor:regular-poset-actions} from the Appendix.
According to this corollary, it is enough to check that if $\aa_{0}\prec\ldots\prec\aa_{r}$
in $\ap\left(\Sigma,f\right)$ and $g_{0}.\aa_{0}\prec\ldots\prec g_{r}.\aa_{r}$
for some $g_{0},\ldots,g_{r}\in\MCG\left(\Sigma\right)$, then there
is a $g\in\MCG\left(\Sigma\right)$ with $g.\aa_{i}=g_{i}.\aa_{i}$
for every $i$. In fact, we show more: we show that in this case,
necessarily $g_{0}=g_{1}=\ldots=g_{r}$. To prove this stronger property,
it is enough to show it for a pair of elements, namely, that if $\mbox{\ensuremath{\aa\prec\bb}}$
and $g.\aa\prec g'.\bb$, then $g=g'$. By acting on the latter pair
by $g^{-1}$, we get that $\aa\prec\left(g^{-1}g'\right).\bb$. So,
replacing $g^{-1}g'$ with $g$, we reduce to showing that if $\aa\prec\bb$
and $\aa\prec g.\bb$ then $g=\id$. 

Consider again the isomorphism of CW-complexes $\varphi_{\aa}\colon\Sigma_{\aa}\overset{\cong}{\to}\Sigma_{\left(\sigma_{\aa},\tau_{\aa}\right)}$.
Let ${\cal P}_{1}$ and ${\cal P}_{2}$ be the unique sets of colored
non-crossing partitions of the arcs in type-$o$ discs of $\Sigma_{\aa}$
which yield $\bb$ and $g.\bb$, respectively. They both pass through
the homeomorphism induced by $\varphi_{\aa}$ to the unique set of
colored non-crossing partitions of type-$o$ discs in $\Sigma_{\left(\sigma_{\aa},\tau_{\aa}\right)}$
yielding $(\sigma_{\bb},\tau_{\bb})=\Psi(\bb)=\Psi(g.\bb)$. Thus,
${\cal P}_{1}={\cal P}_{2}$ and $\bb=g.\bb$. Using the freeness
from item \textbf{\ref{enu:MCG action on AP}}, we obtain that $g=\id$. \end{proof}
\begin{example}
\label{exa:a,b,a,c}We already analyzed above the pairs of matchings
poset $\pmp\left(\Sigma,f\right)$ and the simplicial complex $\left|\pmp\left(\Sigma,f\right)\right|$
of the sole incompressible $\left(\Sigma,f\right)$ which is admissible
for $w=\left[x,y\right]\left[x,z\right]$ (see example \ref{example:[x,y][x,z] leading term 0}
as well as Pages \pageref{[x,y][x,z] - PMP} and \pageref{[x,y][x,z] - simplicial complex}).
We saw that $\left|\pmp\left(\Sigma,f\right)\right|$ was a cycle
(composed of $4$ vertices and 4 edges). We already know that $\left|\ap\left(\Sigma,f\right)\right|$
is a covering space of $\left|\pmp\left(\Sigma,f\right)\right|$,
so every connected component of it is either a cycle or an infinite
line. In Figure \ref{fig:5 elems of AP} we show a piece of a connected
component of $\left|\ap\left(\Sigma,f\right)\right|$ made of three
elements of smallest rank together with two elements of one rank higher,
forming together a path of four edges. By carefully analyzing this
component, it is possible to see that it is actually homeomorphic
to an infinite line, and by Theorem \ref{thm:MCG-action-on-AP} it
follows that all components are of the same form. The fact it is a
line is an instance of Theorem \ref{thm:components-of-AP(Sigma,f)}
below. 

The middle element in Figure \ref{fig:5 elems of AP} is the same
as the left element in Figure \ref{fig:first-arc-systems}. The right
element in Figure \ref{fig:first-arc-systems} is yet another element
of the same poset $\ap\left(\Sigma,f\right)$. It is easy to see (by,
e.g., Claim \ref{claim:f_aa on arcs}) that this element induces a
different homotopy class of maps to $\wedger$. By Theorem \ref{thm:components-of-AP(Sigma,f)}
below this means it belongs to a different connected component of
$\left|\ap\left(\Sigma,f\right)\right|$. However, this element induces
the same bijections as the middle element in Figure \ref{fig:5 elems of AP}
and thus can be mapped to it by some mapping class in $\mathrm{MCG}\left(\Sigma\right)$
(a Dehn twist in this case).
\end{example}
\begin{figure}[h]
\centering{}\includegraphics[bb=50bp 0bp 350bp 150bp]{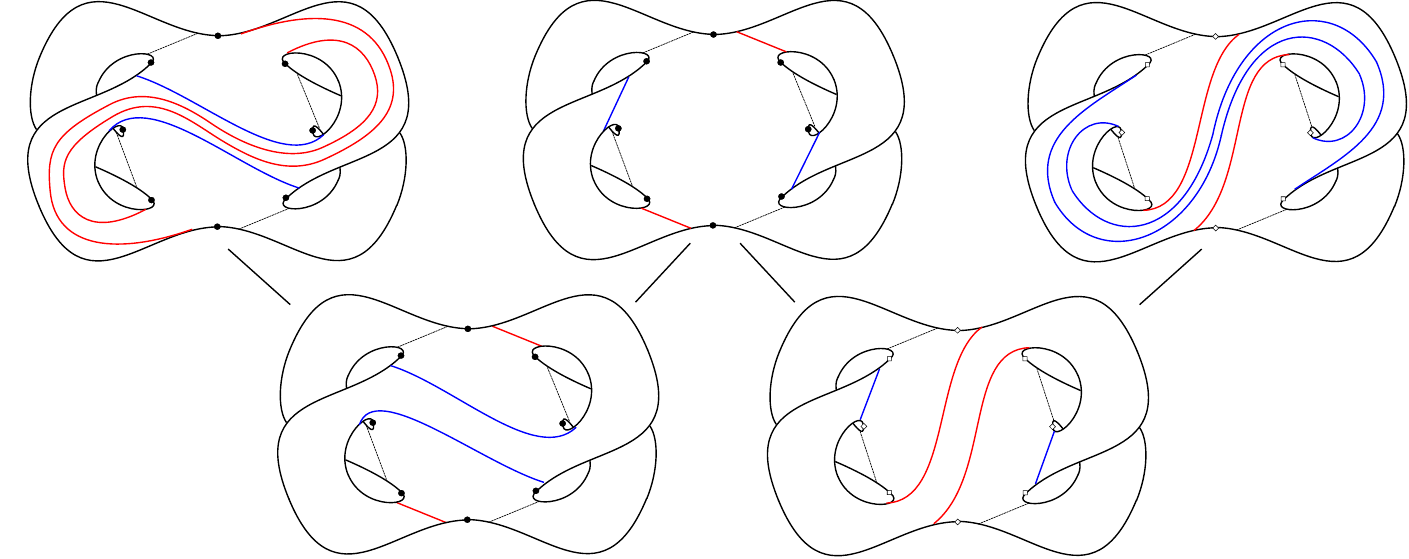}\caption{\label{fig:5 elems of AP}A series of five elements in the same connected
component of the arc poset $\ap\left(\Sigma,f\right)$ of the sole
equivalence class $\left[\left(\Sigma,f\right)\right]$ of incompressible
map admissible for the word $[x,y][x,z]$. The red lines are $p_{1}$-arcs
and the blue lines are $q_{1}$-arcs. The first and last elements
differ by an element of $\MCG(\Sigma)$ and induce the same matchings
$E^{+}\protect\overset{\sim}{\to}E^{-}$.}
\end{figure}

\begin{example}
\label{exa:aa,b}Now consider $w=\left[x^{2},y\right]=x_{1}x_{2}y_{3}X_{4}X_{5}Y_{6}$.
An easy computation yields that $\sol\left(w\right)$ consists of
exactly two equivalence classes. One $\left[\left(\Sigma_{1,1},f\right)\right]$
is represented by the pair of matchings $\sigma=\tau=\left(\begin{array}{ccc}
x_{1} & x_{2} & y_{3}\\
X_{5} & X_{4} & Y_{6}
\end{array}\right)$ and corresponds to the presentation of $w$ as the commutator $\left[x^{2},y\right]$;
the other equivalence class $\left[\left(\Sigma_{1,1},f'\right)\right]$
is represented by the pair of matchings $\sigma=\tau=\left(\begin{array}{ccc}
x_{1} & x_{2} & y_{3}\\
X_{4} & X_{5} & Y_{6}
\end{array}\right)$ and corresponds to the non-equivalent (under $\mathrm{Aut}_{\delta}\left(\F_{2}\right)$)
presentation as $\left[x^{2},yx\right]$. Both $\left|\pmp\left(\Sigma_{1,1},f\right)\right|$
and $\left|\pmp\left(\Sigma_{1,1},f'\right)\right|$ are each an isolated
point. It follows from Theorem \ref{thm:MCG-action-on-AP} that $\left|\ap\left(\Sigma_{1,1},f\right)\right|$
and $\left|\ap\left(\Sigma_{1,1},f'\right)\right|$ are also composed
of isolated points. In fact, there are infinitely countably many of
them in each of the two (this follows from Theorem \ref{thm:components-of-AP(Sigma,f)}
below). In Figure \ref{fig:three arc systems} we draw three elements
from these two arc posets. \FigBesBeg \\
\begin{figure}[h]
\centering{}\includegraphics[bb=0bp 0bp 375bp 300bp,scale=0.8]{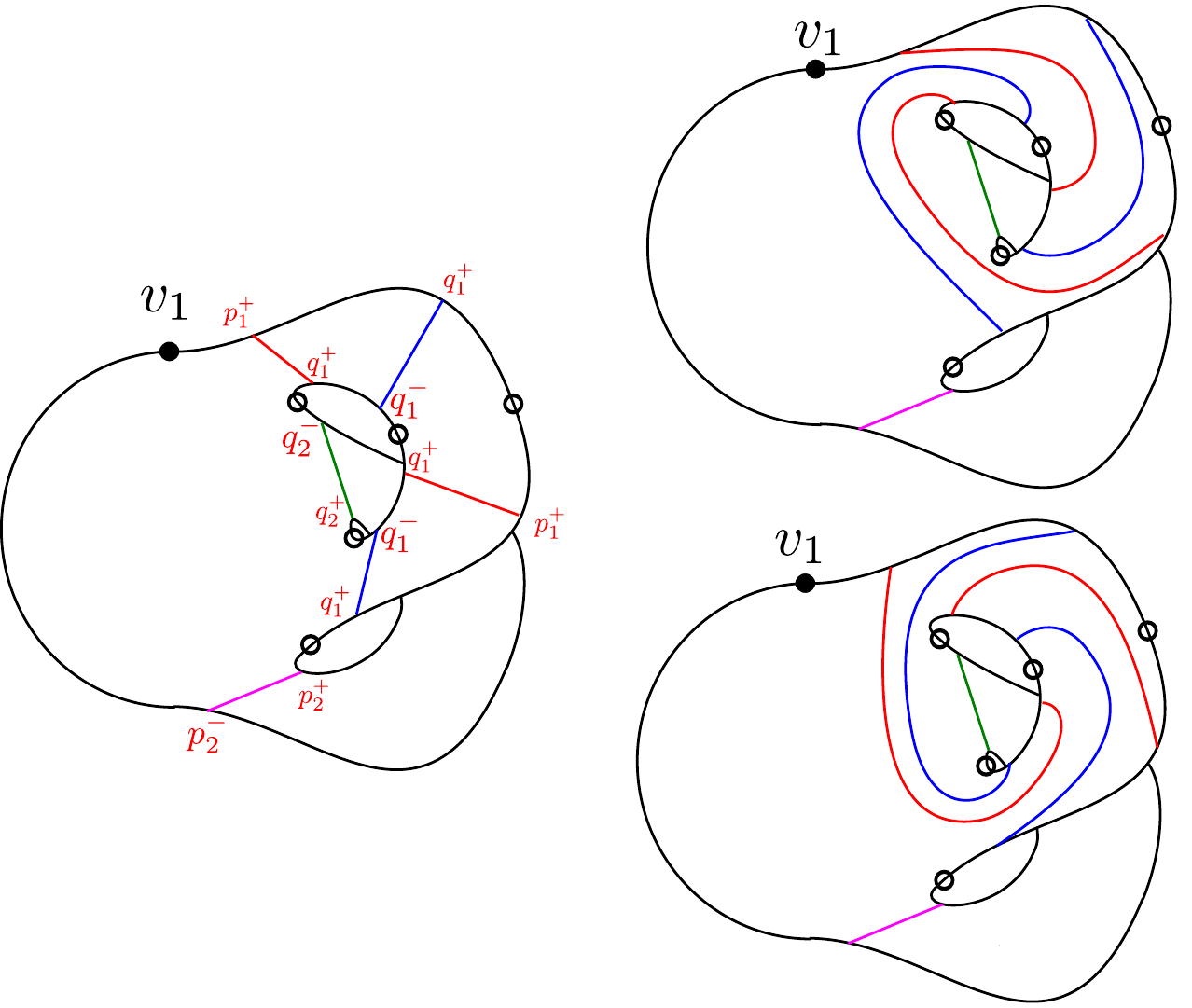}\caption{\label{fig:three arc systems} on the left: an element from $\ap\left(\Sigma_{1,1},f\right)$,
where $\left(\Sigma_{1,1},f\right)$ is admissible for $w=\left[x^{2},y\right]$
and corresponds to the solution $w=\left[x^{2},y\right]$. On the
right: two elements from $\ap\left(\Sigma_{1,1},f'\right)$, where
$\left(\Sigma_{1,1},f'\right)$ is admissible for $w=\left[x^{2},y\right]$
and corresponds to the solution $w=\left[x^{2},yx\right]$. Clearly,
the two arc systems on the right are in the same orbit of the action
of $\mathrm{MCG}\left(\Sigma_{1,1}\right)$.}
\end{figure}
\FigBesEnd 
\end{example}
In both examples the connected components of $\left|\ap\left(\Sigma,f\right)\right|$
are contractible: infinite lines in Example \ref{exa:a,b,a,c} and
isolated points in Example \ref{exa:aa,b}. In particular, in both
examples, every connected component is the universal covering space
of the corresponding connected component of $\left|\pmp\left(\Sigma,f\right)\right|$.
This turns out to be the general case:
\begin{thm}
\label{thm:components-of-AP(Sigma,f)}The map $\ap\left(\Sigma,f\right)\to\left\{ \Sigma\to\wedger\right\} $
given by $\aa\mapsto f_{\aa}$ induces a one-to-one correspondence
between the connected components of $\left|\ap\left(\Sigma,f\right)\right|$
and the homotopy classes (relative $\partial\Sigma$) of maps $\Sigma\to\wedger$
which are equivalent to $f$: 
\[
\pi_{0}\left(\left|\ap\left(\Sigma,f\right)\right|\right)\overset{\sim}{\to}\left\{ \begin{gathered}\mathrm{homotopy\,\, classes\,\, relative\,\,}\partial\Sigma\mathrm{\,\, of}\\
f'\colon\Sigma\to\wedger
\end{gathered}
\,\middle|\,\left(\Sigma,f'\right)\sim\left(\Sigma,f\right)\right\} .
\]
Moreover, every connected component of $ $$\left|\ap\left(\Sigma,f\right)\right|$
is contractible. 
\end{thm}
Recall that homotopy classes relative $\partial\Sigma$ of maps $\Sigma\to\wedger$
are in one-to-one correspondence with the homomorphisms of {}``fundamental
groupoid'' (Lemma \ref{lem:homotopy=00003Dsame_induced_map}). In
particular, when $\ell=1$, if $g=\cl\left(w_{1}\right)$, the correspondence
in Theorem \ref{thm:components-of-AP(Sigma,f)} can be interpreted
as a one-to-one correspondence between $\pi_{0}\left(\left|\ap\left(\Sigma,f\right)\right|\right)$
and the elements in the $\mathrm{Aut}_{\delta}\left(\F_{2g}\right)$-orbit
of $f_{*}$ in $\mathrm{Hom}_{w}\left(\F_{2g},\F_{r}\right)$.

The proof of Theorem \ref{thm:components-of-AP(Sigma,f)} is the most
technical in the paper, and we postpone it to Section \ref{sub:Proof-of-contractability}.
We first explain how it readily yields Theorem \ref{thm:pmp is K(G,1)}
and thus our main results.
\begin{proof}[Proof of Theorem \ref{thm:pmp is K(G,1)} given Theorem \ref{thm:components-of-AP(Sigma,f)}:]
\textbf{} Recall that to show that $\left|\pmp\left(\Sigma,f\right)\right|$
is a $\mathrm{K}\left(G,1\right)$-space for $G=\mathrm{Stab}_{\mathrm{MCG}\left(\Sigma\right)}(\tilde{f})$,
one needs to establish that $\left|\pmp\left(\Sigma,f\right)\right|$
is path-connected, that its fundamental group is isomorphic to $G$
and that its universal covering is contractible. 

By Theorem \ref{thm:MCG-action-on-AP}, $\left|\ap\left(\Sigma,f\right)\right|$
is a covering space of $\left|\pmp\left(\Sigma,f\right)\right|$.
In particular, so is every connected component of $\left|\ap\left(\Sigma,f\right)\right|$.
For instance, by Theorem \ref{thm:components-of-AP(Sigma,f)}, we
can take the connected component of $\left|\ap\left(\Sigma,f\right)\right|$
corresponding to $\tilde{f}$, the homotopy class of $f$. Denote
this component by $C$. By Theorem \ref{thm:components-of-AP(Sigma,f)}
again, $C$ is contractible and therefore the universal covering of
$\left|\pmp\left(\Sigma,f\right)\right|$. 

The subgroup of $\mathrm{MCG}\left(\Sigma\right)$ of elements mapping
$C$ to itself are precisely those preserving $\tilde{f}$, namely,
precisely $G=\mathrm{Stab}_{\mathrm{MCG}\left(\Sigma\right)}(\tilde{f})$.
Thus, the action of $\mathrm{MCG}\left(\Sigma\right)$ on $\left|\ap\left(\Sigma,f\right)\right|$
restricts to the action of $G$ on $C$. This action is precisely
the covering action, hence $G\cong\pi_{1}\left(\left|\pmp\left(\Sigma,f\right)\right|\right)$. 

Finally, to show $\left|\pmp\left(\Sigma,f\right)\right|$ is path-connected,
it is enough to show there is a path between any two of its vertices.
Let $\left(\st\right)\in\pmp\left(\Sigma,f\right)$ be a pair of matchings.
By definition, since $\left(\Sigma,f\right)\sim\left(\Sigma_{\left(\st\right)},f_{\left(\st\right)}\right)$,
there is a homeomorphism $\rho\colon\Sigma\to\Sigma_{\left(\st\right)}$
with $f\simeq f_{\left(\st\right)}\circ\rho$ homotopic. We can use
the image through $\rho^{-1}$ of the matching-edges in $\Sigma_{\left(\st\right)}$
to get an arc system $\aa\in\ap\left(\Sigma,f\right)$ with $f_{\aa}\simeq f$
homotopic, and so that $\Psi\left(\aa\right)=\left(\st\right)$ (see
the notation from Proposition \ref{prop:graded-poset-morphism AP to PMP}).
But $\left(\st\right)\in\pmp\left(\Sigma,f\right)$ was arbitrary,
and we can, likewise, obtain $\bb\in\ap\left(\Sigma,f\right)$ with
$f_{\bb}\simeq f$ and $\Psi(\bb)=\left(\sigma',\tau'\right)$ for
any $\left(\sigma',\tau'\right)\in\pmp\left(\Sigma,f\right)$. By
Theorem \ref{thm:components-of-AP(Sigma,f)}, $\aa$ and $\bb$ belong
to the same connected component of $\left|\ap\left(\Sigma,f\right)\right|$
(specifically, to $C$, the one corresponding to $\tilde{f}$). We
can now take any path between them in $C$ and project it to a path
between $\left(\st\right)$ and $\left(\sigma',\tau'\right)$ in $\left|\pmp\left(\Sigma,f\right)\right|$.
\end{proof}

\subsection{\label{sub:Proof-of-contractability}Contractability of connected
components}

We now come to prove Theorem \ref{thm:components-of-AP(Sigma,f)},
regarding the connected components of $\ap\left(\Sigma,f\right)$.
Let 
\[
\Upsilon\colon\pi_{0}\left(\left|\ap\left(\Sigma,f\right)\right|\right)\to\left\{ \begin{gathered}\mathrm{homotopy\,\, classes\,\, relative\,\,}\partial\Sigma\mathrm{\,\, of}\\
f'\colon\Sigma\to\wedger
\end{gathered}
\,\middle|\,\left(\Sigma,f'\right)\sim\left(\Sigma,f\right)\right\} 
\]
be the map defined on every connected component $C$ by taking an
arbitrary vertex $\aa\in C$ and mapping $C$ to $f_{\aa}$. We need
to show that $\Upsilon$ is a well-defined bijection, and that every
such $C$ is contractible.
\begin{lem}
\label{lem:Upsilon well defined}$\Upsilon$ is well-defined.\end{lem}
\begin{proof}
To see that $\Upsilon$ is well-defined, it is enough to show that
if $\bb$ covers $\aa$ in $\ap\left(\Sigma,f\right)$, then $f_{\bb}\simeq f_{\aa}$
are homotopic. This is shown by an argument we already used in Section
\ref{sec:The-pairs-of-matchings-Poset}: in this case, there is a
particular type-$z_{i}$ disc $D$ defined by $\bb$, and two equally-colored
arcs at its boundary, say $\beta_{1}$ and $\beta_{2}$, which are
replaced by $\alpha_{1}$ and $\alpha_{2}$ to obtain $\aa=\left[\left\{ \alpha_{1},\alpha_{2},\beta_{3},\ldots,\beta_{2L}\right\} \right]$.
We can take both $f_{\aa}$ and $f_{\bb}$ to be constant (and identical)
on all arcs $\alpha_{1},\alpha_{2},\beta_{1},\beta_{2},\beta_{3},\ldots,\beta_{2L}$.
Since these arcs cut $\Sigma$ to discs, Lemma \ref{lem:homotopy=00003Dsame_induced_map}
shows $f_{\aa}\simeq f_{\bb}$. \end{proof}
\begin{lem}
\label{lem:Upsilon is onto}$\Upsilon$ is onto.\end{lem}
\begin{proof}
Let $f'\colon\Sigma\to\wedger$ satisfy $\left(\Sigma,f'\right)\sim\left(\Sigma,f\right)$.
We want to show there is an arc system $\bb\in\ap\left(\Sigma,f\right)$
with $f_{\bb}\simeq f'$. First, note we have already seen that $\pmp\left(\Sigma,f\right)$
is non-empty (Lemma \ref{lem:every admissible and incompressible obtained from matchings}),
and thus nor is $\ap\left(\Sigma,f\right)$ (Proposition \ref{prop:graded-poset-morphism AP to PMP}).
So there is some $\aa\in\ap\left(\Sigma,f\right)$. Now we can repeat
an argument we used in the very end of Section \ref{sub:The-Arc-Poset}:
by definition, $\left(\Sigma,f'\right)\sim\left(\Sigma,f_{\aa}\right)$,
so there is a homeomorphism $\rho\colon\Sigma\to\Sigma$ with $f'\simeq f_{\aa}\circ\rho$
homotopic. The arc system $\bb=\rho^{-1}\left(\aa\right)$ now satisfies
$f_{\bb}\simeq f'$.
\end{proof}
We are left to show that $\Upsilon$ is injective and that every connected
component of $\left|\ap\left(\Sigma,f\right)\right|$ is contractible.
Although the former is easier than the latter, we prove both at once.
Consider the subposet\marginpar{$\P\left(f\right)$}

\[
\P\left(f\right)\overset{\mathrm{def}}{=}\left\{ \aa\in\ap\left(\Sigma,f\right)\,\middle|\, f_{\aa}\simeq f\right\} \subseteq\ap\left(\Sigma,f\right).
\]
We show that $\left|\P\left(f\right)\right|$ is connected and, moreover,
contractible. Since $f$ is arbitrary (if $\left(\Sigma,f'\right)\sim\left(\Sigma,f\right)$,
then $\ap\left(\Sigma,f\right)=\ap\left(\Sigma,f'\right)$ and we
could work just as well with $f'$), this yields that the same is
true for any $f'\colon\Sigma\to\wedger$ with $\left(\Sigma,f'\right)\sim\left(\Sigma,f\right)$,
and thus proves Theorem \ref{thm:components-of-AP(Sigma,f)}.

It already follows from Lemmas \ref{lem:Upsilon well defined} and
\ref{lem:Upsilon is onto} that $\left|\P\left(f\right)\right|$ is
a non-empty collection of connected components of $\left|\ap\left(\Sigma,f\right)\right|$.
It is left to show it consists of a single component, and that this
component is contractible.

\subsubsection*{Guide-arcs}

Fix $\aa_{0}\in{\cal P}\left(f\right)$ (so $f_{\aa_{0}}\simeq f$).
Let $\left\{ \alpha_{1},\ldots,\alpha_{2L}\right\} $ be a representative
of $\aa_{0}$. 
\begin{defn}
\label{def:guide-arcs}A finite set of arcs $\gamma_{1},\ldots,\gamma_{M}$
embedded in $\Sigma$ is said to be \textbf{a set of guide-arcs }for
$\left\{ \alpha_{1},\ldots,\alpha_{2L}\right\} $ if
\begin{itemize}
\item the $\gamma_{m}$'s are disjoint from each other and from the $\alpha_{i}$'s,
and
\item the only arc system in $\ap\left(\Sigma,f\right)$ with a representative
which is disjoint from $\gamma_{1}\cup\ldots\cup\gamma_{M}$ is $\aa$.
\end{itemize}
\end{defn}
Every (representative of an) arc system has a set of guide-arcs: for
example, for every arc $\alpha$ in the system take two guide arcs
which follow $\alpha$ very closely, one from each side, in a parallel
fashion. Figure \ref{fig:guide-arcs} illustrates a set of guide-arcs
of size five for an element of $\ap\left(\Sigma,f\right)$ where $\left[\left(\Sigma,f\right)\right]\in\sol\left(\left[x,y\right]\left[x,z\right]\right)$. 

\begin{figure}[h]
\centering{}\includegraphics[bb=0bp 0bp 350bp 225bp,scale=0.8]{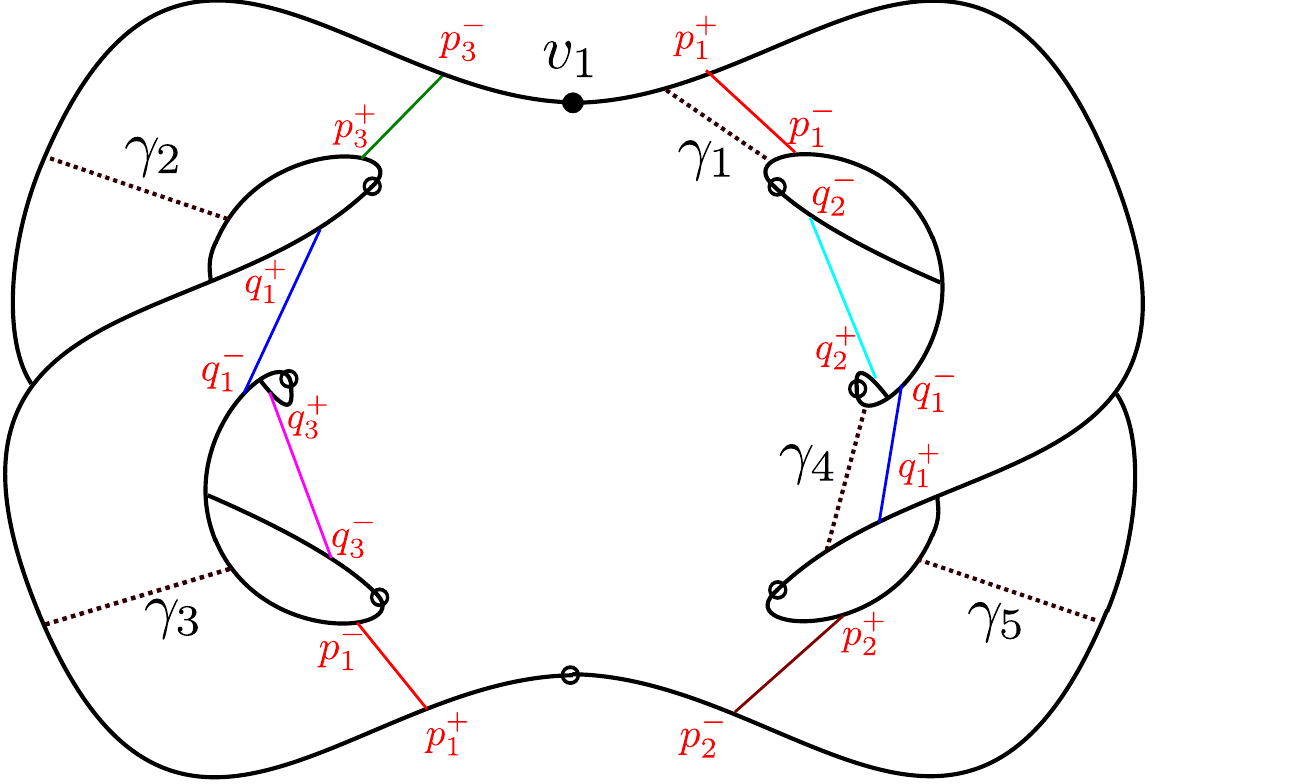}\caption{\label{fig:guide-arcs}A set of guide-arcs (marked in dotted lines)
for an element of $\ap(\Sigma,f)$ with $\left[\left(\Sigma,f\right)\right]\in\sol\left(\left[x,y\right]\left[x,z\right]\right)$.
This element is the central one in Figure \ref{fig:5 elems of AP}.}
\end{figure}

Given a set of guide-arcs $\gamma_{1},\ldots,\gamma_{M}$, let \marginpar{${\cal P}_{m}$}$\P_{m}$,
$0\le m\le M$, denote the subposet of $\P\left(f\right)$ consisting
of arc systems which have a representative which does not cross $\gamma_{m+1},\ldots,\gamma_{M}$
(but may cross $\gamma_{1},\ldots,\gamma_{m}$). So
\[
\left\{ \aa_{0}\right\} =\P_{0}\subseteq\P_{1}\subseteq\ldots\subseteq\P_{M}=\P\left(f\right)
\]
is an increasing sequence of posets. Consider, for example, the set
of guide arcs given in Figure \ref{fig:guide-arcs}, and denote the
five elements in Figure \ref{fig:5 elems of AP}, from left to right,
by $\aa_{-2}$, $\aa_{-1}$, $\aa_{0}$, $\aa_{1}$ and $\aa_{2}$.
Then, ${\cal P}_{0}=\left\{ \aa_{0}\right\} $, ${\cal P}_{1}=\P_{2}=\left\{ \aa_{0},\aa_{1}\right\} $,
${\cal P}_{3}=\left\{ \aa_{0},\aa_{1},\aa_{2}\right\} $ and ${\cal P}_{4}={\cal P}_{5}={\cal P}\left(f\right)$
contain the entire connected component of $\left|\ap\left(\Sigma,f\right)\right|$
(the component a piece of which is given in Figure \ref{fig:5 elems of AP}).
We stress that there may be many more elements in $\ap\left(\Sigma,f\right)$
with representatives which do not cross subsets of the guide-arcs
(for instance, the arc system in the right hand side of Figure \ref{fig:first-arc-systems}
does not cross $\gamma_{2}$, $\gamma_{3}$ nor $\gamma_{5}$), but
they do not belong to ${\cal P}\left(f\right)$, and thus nor to the
${\cal P}_{m}$'s.

We shall prove the contractability of $\left|\P\left(f\right)\right|$
by showing that each $\left|\P_{m}\right|$ deformation retracts to
$\left|\P_{m-1}\right|$.

\subsubsection*{Depth of words along guide-arcs}

Fix an arbitrary orientation for each guide-arc $\gamma_{m}$. For
every $\aa\in\P\left(f\right)$ find a representative which meets
the guide-arcs transversely and in minimal position (so that no arc
of $\aa$ crosses twice in a row the same guide-arc). Define $u_{m}\left(\aa\right)$\marginpar{$u_{m}\left(\aa\right)$}
to be a word in the alphabet $\left\{ p_{1},q_{1},\ldots,p_{r},q_{r}\right\} $
which describes the sequence of crossings between $\gamma_{m}$ and
$\aa$: simply follow $\gamma_{m}$ according to the given orientation,
whenever it crosses a $p_{i}$-arc write $p_{i}$, and whenever it
crosses a $q_{i}$-arc, write $q_{i}$. In this language,
\[
\P_{m}=\left\{ \aa\in\P\left(f\right)\,\middle|\, u_{m+1}\left(\aa\right)=u_{m+2}\left(\aa\right)=\ldots=u_{M}\left(\aa\right)\,\,\mathrm{are\, all\, the\, empty\, word}\right\} .
\]

\begin{lem}
\label{lem:u_i null-homotopic}For every $\aa\in\P\left(f\right)$
and every $1\le m\le M$, the formal word $u_{m}\left(\aa\right)$
can be reduced to the empty word by a series of deletions of subwords
$p_{i}p_{i}$ and $q_{i}q_{i}$. \end{lem}
\begin{proof}
Recall that for any path $\overline{\gamma}$ in $\Sigma$ from $v_{i}$
to $v_{j}$ ($i,j\in\left[\ell\right]$) which meets the arcs of $\aa$
transversely, the value of $\left[f_{\aa}\left(\overline{\gamma}\right)\right]$
is determined by the sequence of crossings between $\overline{\gamma}$
and the arcs, as detailed in Claim \ref{claim:f_aa on arcs}. It is
easy to see that an equivalent way to define $\left[f_{\aa}\left(\overline{\gamma}\right)\right]$
is the following: write a word in $\left\{ p_{1},q_{1},\ldots,p_{r},q_{r}\right\} $
which depicts the sequence of crossings of $\overline{\gamma}$ with
the arcs of $\aa$ (as in the definition of $u_{m}\left(\aa\right)$),
then reduce this word by deleting subwords of the form $p_{i}p_{i}$
or $q_{i}q_{i}$, and eventually scan the word from beginning to end
and replace every $p_{i}q_{i}$ with $x_{i}$ and every $q_{i}p_{i}$
with $x_{i}^{-1}$. (It is standard that the order of reductions does
not effect the final result.) 

Now let $\gamma=\gamma_{m}$ for some $m$ starting at the boundary
component $i$ and arriving at the boundary component $j$. Let $\overline{\gamma}$
be a path in $\Sigma$ which begins at $v_{i}$, then goes along $\partial\Sigma$
from $v_{i}$ to the beginning of $\gamma$, then goes along $\gamma$,
and then arrives to $v_{j}$ through $\partial\Sigma$ (the parts
trough $\partial\Sigma$ can be chosen arbitrarily). Since the arcs
of an arc system always meet the boundary only at their endpoints,
the sequence of crossings along the pieces of $\overline{\gamma}$
at the boundary are the same for all elements of $\ap\left(\Sigma,f\right)$.
Since $\left[f_{\aa}\left(\overline{\gamma}\right)\right]=\left[f_{\aa_{0}}\left(\overline{\gamma}\right)\right]$,
the words $u_{m}\left(\aa\right)$ and $u_{m}\left(\aa_{0}\right)$
must be equivalent (through reductions). We are done as $u_{m}\left(\aa_{0}\right)$
is empty by the definition of guide-arcs.
\end{proof}
For example, for $\aa\in\ap\left(\Sigma,f\right)$ the element in
Figure \ref{fig:first-arc-systems} on the right and the element $\aa_{0}$
and guide arcs in Figure \ref{fig:guide-arcs}, $u_{1}\left(\aa\right)=q_{1}p_{1}$
and thus $\aa\notin{\cal P}\left(f_{\aa_{0}}\right)$.

Next, we define the depth of $u_{m}\left(\aa\right)$. Let $\mathbb{T}_{2r,2}$\marginpar{$\mathbb{T}_{2r,2}$}
be the infinite $\left(2r,2\right)$-biregular tree%
\footnote{A $\left(2r,2\right)$-biregular tree has vertices of degrees $2r$
and $2$. Every vertex of degree $2r$ is connected only with vertices
of degree $2$, and vice-versa.%
}. We think of it as the universal cover of the graph $\wedger$, where
the point $o$ and the points $\left\{ z_{i}\,\middle|\, i\in\left[r\right]\right\} $
are vertices. We also label every vertex of $\mathbb{T}_{2r,2}$ by
$o$ or $z_{i}$ according to the vertex it covers, and every edge
of $\mathbb{T}_{2r,2}$ by $p_{i}$ or $q_{i}$, according to the
marked point contained in the edge of $\wedger$ it covers.

Since $\gamma_{m}$ is disjoint from the arcs of $\aa_{0}$, it is
completely embedded in a (closed, type-$o$ or type-$z_{i}$) disc
of $\aa_{0}$. If this disc is type-$o$ (type-$z_{i}$), then $\gamma_{m}$
begins and ends in a type-$o$ (type-$z_{i}$, respectively) disc
in any arc system in $\ap\left(\Sigma,f\right)$. If it begins and
ends in a type-$o$ (type-$z_{i}$) disc, we choose a basepoint \marginpar{$\otimes_{m}$}$\otimes_{m}$
for $\mathbb{T}_{2r,2}$ in some $o$-vertex ($z_{i}$-vertex, respectively).
We can think of $u_{m}\left(\aa\right)$ as a path in the tree: we
begin at the basepoint $\otimes_{m}$, whenever we write $p_{i}$,
we traverse a $p_{i}$-edge, and whenever we write $q_{i}$ we traverse
a $q_{i}$-edge. It it easy to verify that we never get stuck (if
our walk reaches a $z_{i}$-vertex, the following step will necessarily
be a $p_{i}$ or a $q_{i}$ with the same $i$). Moreover, $u_{m}\left(\aa\right)$
reduces to the empty word if and only if the associated walk in the
tree is closed. 

We define \textbf{the depth of $u_{m}\left(\aa\right)$}, denoted
\marginpar{${\scriptstyle \mathrm{depth}\left(u_{m}\left(\aa\right)\right)}$}$\mathrm{depth}\left(u_{m}\left(\aa\right)\right)$,\textbf{
}to be the largest distance from the basepoint $\otimes_{m}$ of a
vertex in $\mathbb{T}_{2r,2}$ visited in the walk of $u_{m}\left(\aa\right)$.
For example, in the following word we write the distance from the
basepoint to the vertex visited after every step: 
\[
\overset{0}{\,\,\,}p_{1}\overset{1}{\,\,\,}q_{1}\overset{2}{\,\,\,}p_{2}\overset{3}{\,\,\,}p_{2}\overset{2}{\,\,\,}q_{3}\overset{3}{\,\,\,}q_{3}\overset{2}{\,\,\,}q_{4}\overset{3}{\,\,\,}p_{4}\overset{4}{\,\,\,}q_{4}\overset{5}{\,\,\,}q_{4}\overset{4}{\,\,\,}p_{4}\overset{3}{\,\,\,}q_{4}\overset{2}{\,\,\,}q_{1}\overset{1}{\,\,\,}q_{1}\overset{2}{\,\,\,}q_{1}\overset{1}{\,\,\,}p_{1}\overset{0}{\,\,\,}
\]
hence the depth of this word is $5$.

This notion of depth allows us to define a finer sequence of nested
subposets $\P_{m,n}$ ($1\le m\le M$ and $n\in\mathbb{Z}_{\ge0}$)
as follows:\marginpar{$\P_{m,n}$}
\[
\P_{m,n}\overset{\mathrm{def}}{=}\left\{ \aa\in\P\left(f\right)\,\middle|\,\begin{gathered}\mathrm{depth}\left(u_{m}\left(\aa\right)\right)\le n,\,\,\mathrm{and}\\
u_{m+1}\left(\aa\right)=u_{m+2}\left(\aa\right)=\ldots=u_{M}\left(\aa\right)\,\,\mathrm{are\, all\, the\, empty\, word}
\end{gathered}
\right\} .
\]
So 
\[
\P_{m-1}=\P_{m,0}\subseteq\P_{m,1}\subseteq\ldots\subseteq\P_{m,n}\subseteq\ldots\subseteq\P_{m},
\]
and 
\[
\bigcup_{n=0}^{\infty}\P_{m,n}=\P_{m}.
\]
For instance, if we continue with the example of $w=\left[x,y\right]\left[x,z\right]$,
the five guide-arcs drawn in Figure \ref{fig:guide-arcs} and the
five elements $\aa_{-2},\ldots,\aa_{2}$ in Figure \ref{fig:5 elems of AP},
then $\P_{0}=\P_{1,0}=\left\{ \aa_{0}\right\} $ and $\P_{1,1}=\P_{1,2}=\ldots={\cal P}_{1}=\P_{2,0}=\left\{ \aa_{0},\aa_{1}\right\} $.
{}``Opening'' $\gamma_{2}$ does not add elements so $\P_{2,n}=\P_{2}=\P_{3,0}=\left\{ \aa_{0},\aa_{1}\right\} $
for every $n$. When we allow words of depth $1$ on $\gamma_{3}$
we get ${\cal P}_{3,1}=\left\{ \aa_{0},\aa_{1},\aa_{2}\right\} $,
but allowing bigger depth there without {}``opening'' $\gamma_{4}$
does not add any elements, so $\P_{3,n}=\P_{3}=\P_{4,0}$ for every
$n\ge1$. The subposet $\P_{4,1}$ already contains, in addition,
$\aa_{-1}$ as well as the element to the right of $\aa_{2}$ which
we may denote by $\aa_{3}$. The leftmost element in Figure \ref{fig:5 elems of AP},
$\aa_{-2}$, is contained only in $\P_{4,2}$, and so does {}``$\aa_{4}$''.
This goes on: $\P_{4,n}$ consists of $\P_{4,n-1}$ together with
one more element to the right and one more element to the left in
the component a piece of which is given in Figure \ref{fig:5 elems of AP}.
Finally, $\P_{4}=\P_{5,n}=\P_{5}=\P\left(f\right)$ for every $n$.

Using Corollary \ref{cor:deformation-retract}, we now show that $\left|\P_{m,n}\right|$
deformation retracts to $\left|\P_{m,n-1}\right|$. Namely, we show
there is a map $\left|\P_{m,n}\right|\to\left|\P_{m,n-1}\right|$
which restricts to the identity in $\left|\P_{m,n-1}\right|$ and
is homotopic to the identity in $\left|\P_{m,n}\right|$, through
an homotopy that fixes $\left|\P_{m,n-1}\right|$ pointwise. Showing
this means that $\P_{m}$ deformation retracts to $\P_{m-1}$, and
thus completes the proof. (To be sure: we can let the deformation
retract $\left|\P_{m,n}\right|\to\left|\P_{m,n-1}\right|$ take place
at time $\left[\frac{1}{2^{n}},\frac{1}{2^{n-1}}\right]$. This is
a well-defined deformation retract $\left|\P_{m}\right|\to\left|\P_{m-1}\right|$
since every point in $\P_{m}$ belongs to some $\P_{m,n}$, and the
retracts of $\left|\P_{m,n+1}\right|,\left|\P_{m,n+2}\right|,\ldots$
leave $\left|\P_{m,n}\right|$ fixed pointwise.)

\subsubsection*{A deformation retract $\left|\P_{m,n}\right|\to\left|\P_{m,n-1}\right|$ }

The retract $\left|\P_{m,n}\right|\to\left|\P_{m,n-1}\right|$ is
defined by a map \marginpar{$h_{m,n}$}$h_{m,n}:\P_{m,n}\to\P_{m,n-1}$
which prunes all leaves of depth $n$ in the walk $u_{m}\left(\aa\right)$
for every $\aa\in\P_{m,n}$. The basic idea is that if $\aa\in\P_{m,n}\setminus\P_{m,n-1}$
then $u_{m}\left(\aa\right)$ has at least one leaf of depth $n$.
Every such leaf means that $\gamma_{m}$ crosses two equally-colored
arcs of $\aa$ in a row, and we can {}``rewire'' these two arcs
locally to prune the leaf, as in Figure \ref{fig:pruning}. We remark
that in every such step, $\aa$ is modified to some \emph{comparable
}$\bb$, so $\bb$ is in the same connected component of $\left|\ap\left(\Sigma,f\right)\right|$
as $\aa$. By successive steps of this kind we can decrease the depth
of all $u_{m}\left(\aa\right)$ until they are all empty and we arrive
at $\aa_{0}$. This alone suffices to show the connectivity of $\left|\P\left(f\right)\right|$. 

\begin{figure}[h]
\centering{}\includegraphics[bb=0bp 0bp 350bp 135bp]{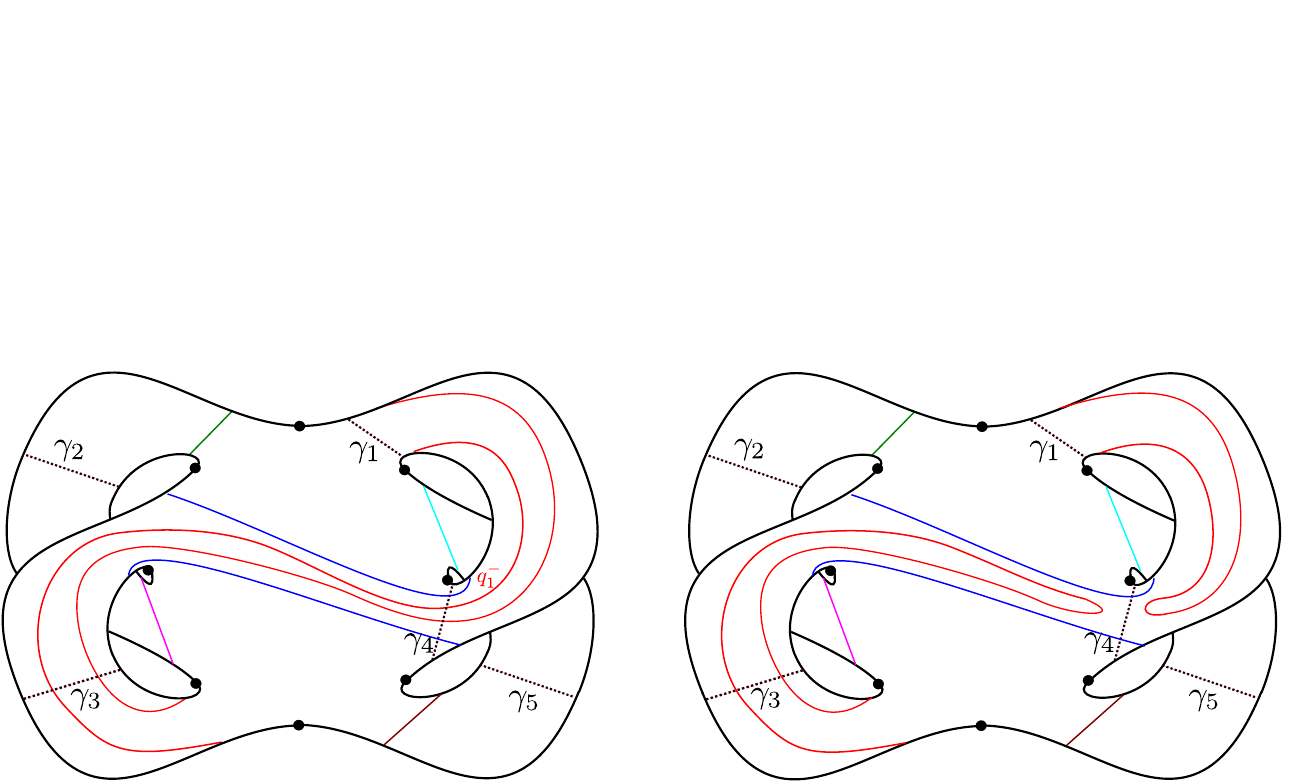}\caption{\label{fig:pruning}Pruning (from left to right) a leaf node of depth
$2$ in $u_{4}\left(\aa_{-2}\right)$, where $\aa_{-2}$ is the left
most arc system in Figure \ref{fig:5 elems of AP}. The result is
$\aa_{-1}$, the left of center arc system in Figure \ref{fig:5 elems of AP}.
We use the guide-arcs from Figure \ref{fig:guide-arcs}. This pruning
is the resulting of applying $h_{4,2}$ on $\aa_{-2}$.}
\end{figure}

More formally, fix $m$ and $n$ and consider all leaves of depth
$n$ in $u_{m}\left(\aa\right)$ (every visit of the walk to a vertex
of distance $n$ from $\otimes_{m}$ is considered a leaf). Every
such leaf corresponds to some backtracking move $p_{i}p_{i}$ or $q_{i}q_{i}$,
and we consider the segment of $\gamma_{m}$ which lies between these
two crossings (between the two crossings with $p_{i}$-arcs of $\aa$,
or two crossings with $q_{i}$-arcs of $\aa$). From the point of
view of the arc system $\aa$, these segments of $\gamma_{m}$ correspond
to disjoint arcs, which we call \textbf{$\gamma$-arcs}\marginpar{$\gamma$-arcs},
inside the discs of $\aa$. Each $\gamma$-arc meets the boundary
of the disc only at its endpoints, and at two equally-colored $\aa$-arcs.
Moreover, the $\gamma$-arcs never cross each other as $\gamma_{m}$
is embedded in $\Sigma$ (and does not self-intersect). In addition,
all vertices at distance $n$ from the basepoint $\otimes_{m}$ in
$\mathbb{T}_{2r,2}$ are of type-$o$, or all are of type-$z_{i}$
(not necessarily the same $i$ for all vertices), depending solely
on the parity of $n$. In the former case, all $\gamma$-arcs are
contained in type-$o$ discs; in the latter in type-$z_{i}$ discs.
From now on we assume that $n$ is such that the $\gamma$-arcs are
all contained in type-$o$ discs, the other case being completely
analogous.

For every type-$o$ disc $D$ of $\aa$ ($\aa\in\P_{m,n}$), the $\gamma$-arcs
determine a partition $P_{D}$ of the arcs in (the boundary of) $D$:
this is the finest partition such that any two arcs connected by a
$\gamma$-arc belong to the same block. We claim that $P_{D}$ is
colored and non-crossing. The monochromaticity of blocks stems from
the fact that the $\gamma$-arcs correspond to subwords of the form
$p_{i}p_{i}$ or $q_{i}q_{i}$ for some $i$. The partition $P_{D}$
is non-crossing because the $\gamma$-arcs are disjoint. We define
$h_{m,n}\left(\aa\right)$\marginpar{$h_{m,n}\left(\aa\right)$} to
be the arc system obtained from $\aa$ by the set of partitions $P_{D}$
of its type-$o$ discs (see Definition \ref{def:arc-poset} and Remark
\ref{remark:about-order-in-arc-poset}).

It is evident that $h_{m,n}\Big|_{\P_{m,n-1}}$ is the identity, and
that $\aa\preceq h_{m,n}\left(\aa\right)$ for every $\aa\in\P_{m,n}$.
Moreover, we claim that indeed $h_{m,n}\left(\P_{m,n}\right)\subseteq\P_{m,n-1}$:
to see this, we show that the modification we made to obtain $h_{m,n}\left(\aa\right)$
from $\aa$ prunes all backtracking steps of $u_{m}\left(\aa\right)$
which correspond to leaves at depth $n$ and does not introduce any
new steps in $u_{m}\left(\aa\right)$ or in $u_{m'}\left(\aa\right)$
for any $m'$. (In contrast, $h_{m,n}$ may prune backtracking steps
at depth smaller than $n$ in $u_{m}\left(\aa\right)$ or at any depth
in $u_{m'}\left(\aa\right)$ for $m'<m$). First, if $\eta$ is any
$\gamma$-arc in $\Sigma_{\aa}$ corresponding to a backtracking step
at distance $n$, it necessarily enters and exists $D$ through two
arcs in the same block of $P_{D}$ and these two crossings disappear
in $h_{m,n}\left(\aa\right)$, hence this leaf is indeed pruned. Second,
any piece $\eta$ of the arc $\gamma_{m'}$ for some $m'\le m$ which
is allocated by two successive crossings of $\aa$-arcs in $\Sigma_{\aa}$
and which is contained in a type-$o$ disc $D$ of $\aa$ satisfies
the following:
\begin{itemize}
\item If $\eta$ enters and exists $D$ through two arcs in the same block
of $P_{D}$, then these two crossings disappear in $h_{m,n}\left(\aa\right)$,
and the corresponding subword $p_{i}p_{i}$ (or $q_{i}q_{i}$) of
$u_{m'}\left(\aa\right)$ is reduced.
\item If $\eta$ enters and exists $D$ through two arcs $e_{1}$ and $e_{2}$
in two different blocks (or {}``$z_{i}$-corridors'') $B_{1}$ and
$B_{2}$, respectively, of $P_{D}$, then it necessarily does not
cross any other block. I.e., there cannot be two other arcs, $e_{3}$
and $e_{4}$ at the same block $B_{3}$ of $P_{D}$, $B_{3}\ne B_{1},B_{2}$,
with the cyclic order of the four being $e_{1},e_{3},e_{2},e_{4}$,
because $\eta$ does not intersect the $\gamma$-arcs. Thus, in minimal
position, the only crossings of $\eta$ with arcs in $h_{m,n}\left(\aa\right)$
are with the arc through which it leaves $B_{1}$ and then through
the arc through which it enters $B_{2}$. By definition of $P_{D}$,
the first arc has the same color as $e_{1}$, and the second arc has
the same color as $e_{2}$. Thus, in this case, there is no change
to the part of $u_{m'}\left(\aa\right)$ corresponding to $\eta$,
when moving from $\aa$ to $h_{m,n}\left(\aa\right)$. 
\end{itemize}
There are also pieces of $\gamma_{m'}$ at its very beginning or very
end which may be contained in type-$o$ discs of $\aa$. The same
argument shows there is no change in the subword of $u_{m'}\left(\aa\right)$
read along such segments when applying $h_{m,n}\left(\aa\right)$.

By Corollary \ref{cor:deformation-retract}, if we want to show that
$h_{m,n}$ induces a deformation retract $\left|h_{m,n}\right|\colon\left|\P_{m,n}\right|\to\left|\P_{m,n-1}\right|$,
we have left to show that $h_{m,n}$ is order-preserving. So assume
$\aa\preceq\bb$, and both are in $\P_{m,n}$. We need to show that
$h_{m,n}\left(\aa\right)\preceq h_{m,n}(\bb)$. This follows from
two properties expressed in the following two lemmas:
\begin{lem}
Let $\aa\preceq\bb$ in $\P\left(f\right)$. We divide the word $u_{m}\left(\aa\right)$
to subwords $x_{1},\ldots,x_{t}$ by grouping together successive
crossings with arcs at the boundary of the same type-$o$ disc. So
$u_{m}\left(\aa\right)=x_{1}*x_{2}*\ldots*x_{t}$, with $*$ denoting
concatenation and each $x_{j}$ of length $2$ except for, possibly,
$x_{1}$ and $x_{t}$, which may be of length 1. Each $x_{j}$ corresponds
to a segment $\eta_{j}$ of $\gamma_{m}$ (allocated by the two crossings).
Since the type-$o$ discs of $\bb$ can be thought of as being contained
inside the type-$o$ discs of $\aa$, we let $y_{j}$ ($1\le j\le t$)
be the subword of $u_{m}(\bb)$ which corresponds to $\eta_{i}$ and
then $u_{m}(\bb)=y_{1}*\ldots*y_{t}$. We claim that for every $j$,
the vertex in $\mathbb{T}_{2r,2}$ that the walk of $u_{m}(\bb)$
visits at the beginning of $y_{j}$, is the same as the vertex visited
by $u_{m}\left(\aa\right)$ at the beginning of $x_{j}$.\end{lem}
\begin{proof}
It is enough to show that $x_{j}$ and $y_{j}$ are equivalent through
reduction for every $j$. Indeed, assume that $x_{j}$ corresponds
to the type-$o$ disc $D$ of $\aa$ and that the partition of this
disc inside the set of partitions leading from $\aa$ to $\bb$ is
$P_{D}$. Because $P_{D}$ is non-crossing, there is a clear order
on the set of blocks of $P_{D}$ (or {}``$z_{i}$-corridors'') crossed
by $\eta_{j}$ ($\eta_{j}$ has to exit a block immediately after
entering it, before entering the next block). We are done as entering
and exiting a block of $P_{D}$ corresponds to a pair of backtracking
steps in $y_{j}$.\end{proof}
\begin{lem}
\label{lem:gamma-arcs in beta}Assume that $\aa\preceq\bb$ in $\P_{m,n}$.
Let $\eta$ be a $\gamma$-arc in $\aa$. Assume that the $\bb$-arcs
intersected by $\eta$ are $\beta_{1},\beta_{2},\ldots,\beta_{2\ell}$.
Then they are all of the same color and represent $\ell$ leaves of
depth $n$ in $\bb$. 
\end{lem}
Note that since $\eta$ begins and ends in (the boundary of) type-$z_{i}$
discs of $\aa$, it must indeed intersect an even number of arcs of
$\bb$ (recall that the type-$o$ discs of $\bb$ can be assumed to
be contained in type-$o$ discs of $\aa$). Of course, $\ell=0$ is
possible.
\begin{proof}
Assume $\ell>0$ (otherwise the statement is trivial). Let $D$ be
the type-$o$ disc of $\aa$ in which $\eta$ is embedded. Since $\eta$
represents a leaf in $u_{m}\left(\aa\right)$, it enters and exits
$D$ through equally-colored arcs $\alpha_{1}$ and $\alpha_{2}$,
and assume w.l.o.g.~these are $p_{1}$-arcs. Now consider the partition
$P_{D}$ of the arcs of $D$ which is part of the set of partitions
yielding $\bb$ from $\aa$. By construction, the two arcs $\beta_{2i}$
and $\beta_{2i+1}$ ($1\le i\le\ell-1$) are formed by rewiring of
the $\aa$-arcs in the same block of $P_{D}$, and thus are of the
same color. 

Since $\eta$ is a $\gamma$-arc, then, by definition, the piece of
walk in $u_{m}\left(\aa\right)$ it corresponds to moves from a vertex
at distance $n-1$ from $\otimes_{m}$ to a vertex of distance $n$
and back. By the previous lemma, the piece of walk represented by
$\eta$ in $u_{m}(\bb)$ also starts at the same vertex of $\mathbb{T}_{2r,2}$,
at distance $n-1$ from $\otimes_{m}$. The arc $\beta_{1}$ is formed
by the rewiring of the block containing $\alpha_{1}$, and thus has
also color $p_{1}$. Thus, after the intersection of $\eta$ with
$\beta_{1}$, the walk $u_{m}(\bb)$ is at distance $n$ from $\otimes_{m}$.
But, and this is the crux of this lemma, $\bb\in\P_{m,n}$ so $\mathrm{depth}(u_{m}(\bb))\le n$.
So the next step of $u_{m}(\bb)$ must backtrack, hence $\beta_{2}$
is also of color $p_{1}$. We already know that $\beta_{2}$ and $\beta_{3}$
have the same color, so $\beta_{3}$ is also of color $p_{1}$ and
represents a step to the vertex at distance $n$. The same argument
as before now shows that $\beta_{4}$ must also be a $p_{1}$-arc
and represents a backtracking step. Repeating these arguments proves
the lemma. 
\end{proof}
We now reach the endgame. Assume that $\aa\preceq\bb$ and both are
in $\P_{m,n}$. We already know that $\aa\preceq h_{m,n}\left(\aa\right)$,
and that $\aa\preceq\bb\preceq h_{m,n}(\bb)$ so $\aa\preceq h_{m,n}(\bb)$.
We need to show that $h_{m,n}\left(\aa\right)\preceq h_{m,n}(\bb)$,
namely, that the partitions in type-$o$ discs of $\aa$ yielding
$h_{m,n}(\bb)$ are \emph{coarser }than those yielding $h_{m,n}\left(\aa\right)$.
To see this, it is convenient to think of these partitions at the
type-$o$ disc $D$ as partitions of the neighboring type-$z_{i}$
discs: each arc at the boundary of $D$ separates it from some type-$z_{i}$
disc%
\footnote{More precisely, we may have to take some of the neighboring discs
with multiplicity two if they have two borders with $D$, a $p_{i}$-border
and a $q_{i}$-border (see Lemma \ref{lem:joint-boundaries-of-discs-in-surface}).
But the partition is colored and thus never merges these two copies
together.%
}. The neighboring type-$z_{i}$ discs in the same block are those
which are merged together through new {}``$z_{i}$-corridors'' formerly
belonging to $D$. It is enough to show that for any $\gamma$-arc
$\eta$, the two type-$z_{i}$ discs of $\aa$ it connects are also
in the same block in the partition leading from $\aa$ to $h_{m,n}(\bb)$.
This is clearly the case by Lemma \ref{lem:gamma-arcs in beta} and
the fact that all depth-$n$ leaves in $\bb$ are pruned in $h_{m,n}(\bb)$.
This completes the proof of Theorem \ref{thm:components-of-AP(Sigma,f)}
and thus also of Theorem \textbf{\ref{thm:pmp is K(G,1)}} and hence
of our main results, Theorems \ref{thm:main - general} and \ref{thm:K(G,1) for incompressible}.
\begin{rem}
A slightly different approach for the proof of contractability would
treat all guide-arcs at one shot, and define the depth of $\aa$ as
the maximal depth of one of $u_{1}\left(\aa\right),\ldots,u_{M}\left(\aa\right)$.
The only subtlety is that the basepoint in $\mathbb{T}_{2r,2}$ of
different $u_{m}\left(\aa\right)$'s may be different, depending on
the type of the disc where $\gamma_{m}$ begins and ends. There are
several ways to go around this: for example, one can prune the depth-$j$
leaves in two steps, one for each subset of the guide-arcs. Another
solution is to fix some $\aa_{0}$ which satisfies $\sigma_{\aa_{0}}=\tau_{\aa_{0}}$.
It is easy to see that in this case the guide-arcs can be taken to
be all inside type-$o$ discs.
\end{rem}

\section{More Consequences\label{sec:More-Consequences}}

In this section we gather some further consequences of our analysis
which are worth mentioning.

\subsubsection*{Finding all solutions and incompressible maps to the (generalized)
commutator problem}

Already in the late 1970's, several algorithms were found to determine
the commutator length of a given word $w\in\left[\F_{r},\F_{r}\right]$
(as mentioned on Page \pageref{Algorithms-to-compute-cl}). One of
these algorithms, due to Culler in \cite{CULLER}, basically follows
the same argument as in Lemma \ref{lem:every admissible and incompressible obtained from matchings}
above --- see Remark \ref{remark: algo-for-cl}. By enumerating all
matchings $\sigma\in\match\left(w\right)$, one can find $g=\cl\left(w\right)$
as $\frac{1}{2}\left(1-\max_{\sigma\in\match\left(w\right)}\chi\left(\sigma,\sigma\right)\right)$,
and then find representatives of every equivalence class of solutions
to 
\[
\left[u_{1},v_{1}\right]\cdots\left[u_{g},v_{g}\right]=w.
\]
By the same Lemma \ref{lem:every admissible and incompressible obtained from matchings},
the same algorithm extends to finding representatives for all classes
in $\sol\left(\wl\right)$ for any $\wl\in\F_{r}$, and more generally,
to all incompressible $\left[\left(\Sigma,f\right)\right]$ which
is admissible for $\wl$.

The main additional contributions of the current paper to this problem
are the following:
\begin{enumerate}
\item \textbf{Identifying all incompressible $\left(\Sigma,f\right)$. }It
can be inferred from the analysis in this paper that $\left(\Sigma_{\left(\st\right)},f_{\left(\st\right)}\right)$
is incompressible if and only if, roughly speaking, $\pmp\left(\Sigma_{\left(\st\right)},f_{\left(\st\right)}\right)$
is downward-closed. More accurately, $\left(\Sigma_{\left(\st\right)},f_{\left(\st\right)}\right)$
is \emph{compressible} if and only if there is a path (each step is
between comparable elements) in the poset $(\match\left(\wl\right)^{2},\preceq)$
from $\left(\st\right)$ to some $\left(\sigma',\tau'\right)$ with
$\chi\left(\sigma',\tau'\right)>\chi\left(\sigma,\tau\right)$ and
without going through elements of Euler characteristic smaller than
$\chi\left(\st\right)$.
\item \textbf{Distinguishing equivalence classes}. The current paper yields
a convenient way of distinguishing the different classes of solutions,
or more generally, of incompressible maps. By Theorem \ref{thm:pmp is K(G,1)},
given $\left(\st\right)\in\match\left(\wl\right)$ with $\left(\Sigma_{\left(\st\right)},f_{\left(\st\right)}\right)$
incompressible, we can construct $\pmp\left(\Sigma_{\left(\st\right)},f_{\left(\st\right)}\right)$
by restricting to pairs $\left(\sigma',\tau'\right)\in\match\left(\wl\right)^{2}$
with $\chi\left(\sigma',\tau'\right)=\chi\left(\st\right)$ and then
taking the connected component of $\left(\st\right)$. This allows
us to identify all $\left(\sigma',\tau'\right)$ belonging to the
same equivalence class of admissible incompressible $\left(\Sigma,f\right)$
as $\left(\st\right)$. \\
In fact, the analysis shows it is enough to follow this algorithm
solely in the bottom two layers of $\match\left(\wl\right)^{2}$:
namely, the pairs where $\left\Vert \sigma^{-1}\tau\right\Vert $
is $0$ (so $\sigma=\tau$) or 1 (so $\sigma^{-1}\tau$ is a transposition).
\end{enumerate}

\subsubsection*{A bound on the dimension of the $\mathrm{K}\left(G,1\right)$-complex
from Theorems \ref{thm:stabilizers K(G,1)} and \ref{thm:K(G,1) for incompressible}}

Recall that if $\left(\Sigma,f\right)$ is admissible for $\wl$ and
incompressible, then $\left|\pmp\left(\Sigma,f\right)\right|$ is
a finite $\mathrm{K}\left(G,1\right)$-complex for $G=\mathrm{Stab}_{\mathrm{MCG}\left(\Sigma\right)}\left(\tilde{f}\right)$.
We can bound the dimension of this $\mathrm{K\left(G,1\right)}$-complex
in terms of $\chi\left(\Sigma\right)$. 

Although we have not stressed it so far, some of the objects in this
paper, such as\linebreak{}
$\match\left(\wl\right)$, $\pmp$ or $\ap$ depend on the particular
presentation of $\wl$ as in \eqref{eq:word-expression}. In our analysis
we assume we fix a particular presentation (e.g.~the reduced one)
and stick to it. We say a presentation is cyclically reduced if $x_{i_{\left(j+1\right)\mod\left|w\right|}}^{\varepsilon_{\left(j+1\right)\mod\left|w\right|}}\ne x_{i_{j}}^{-\varepsilon_{j}}$
for every $1\le j\le\left|w\right|$. Since the objects we study depend
only on the conjugacy class of the words, we can assume they are taken
to be cyclically reduced. 
\begin{cor}
Assume $\wl\ne1$ and that the presentations of $\wl$ are cyclically
reduced. If $\left(\Sigma,f\right)$ is admissible for $\wl$ and
incompressible, then the dimension of $\left|\pmp\left(\Sigma,f\right)\right|$
is at most $-\chi\left(\Sigma\right)$.\end{cor}
\begin{proof}
It is enough to show that $\left\Vert \sigma^{-1}\tau\right\Vert \le-\chi\left(\Sigma\right)$
for every $\left(\sigma,\tau\right)\in\pmp\left(\Sigma,f\right)$.
The rank $\left\Vert \sigma^{-1}\tau\right\Vert $ is equal to $L-\#\mathrm{cycles}\left(\sigma^{-1}\tau\right)$
which is also equal to $\sum_{c}\left(\left|c\right|-1\right)$, the
summation being over all cycles of $\sigma^{-1}\tau$. These cycles
are in one-to-one correspondence with type-$z_{i}$ discs of $\Sigma_{\left(\sigma,\tau\right)}$,
and the size of a cycle is half the number of matching-edges at the
boundary of the corresponding type-$z_{i}$ disc. If we denote the
number of matching-edges at the boundary of a disc $D$ in $\Sigma_{\left(\sigma,\tau\right)}$
by $\deg\left(D\right)$, we obtain 
\begin{equation}
\left\Vert \sigma^{-1}\tau\right\Vert =\sum_{D:\,\mathrm{type-}z_{i}\,\mathrm{disc\, in}\,\Sigma_{\left(\sigma,\tau\right)}}\left(\frac{\deg\left(D\right)}{2}-1\right).\label{eq:rank-as-number-of-bijection-edges}
\end{equation}
Recall that the CW-complex $\Sigma_{\left(\sigma,\tau\right)}$ has
$4L$ $0$-cells, $4L$ 1-cells along the boundary $\partial\Sigma$
and $2L$ $1$-cells as matching-edges, so
\[
\chi\left(\st\right)=\chi\left(\Sigma_{\left(\sigma,\tau\right)}\right)=4L-\left(2L+4L\right)+\#\left\{ \mathrm{discs}\right\} =-2L+\#\left\{ \mathrm{discs}\right\} .
\]
Since every matching-edge is at the boundary of exactly two discs,
\begin{equation}
-\chi\left(\st\right)=2L-\#\left\{ \mathrm{discs}\right\} =\sum_{D:\,\mathrm{disc}}\left(\frac{\deg\left(D\right)}{2}-1\right)\label{eq:euler of surface}
\end{equation}
But when $\wl$ are cyclically reduced, every disc $D$ in $\Sigma_{\left(\sigma,\tau\right)}$
has at least two matching-edges at its boundary, i.e., $\deg\left(D\right)\ge2$.
Hence the right hand side of \eqref{eq:euler of surface} is an upper
bound for the rank in \eqref{eq:rank-as-number-of-bijection-edges}.
\end{proof}

\subsubsection*{Explicit finite presentations of the stabilizers in $\mathrm{Aut}_{\delta}\left(\F_{2g}\right)$
or $\mathrm{MCG}\left(\Sigma\right)$ }

Our analysis also yields a straight-forward algorithm to explicitly
find elements in the stabilizers of solutions $\phi\in\mathrm{Hom}_{w}\left(\F_{2g},\F_{r}\right)$
or $\left[\left(\Sigma,f\right)\right]\in\sol\left(\wl\right)$. One
way to obtain this is the following. For simplicity, we restrict to
the case of a single word $w$ with $\cl\left(w\right)=g$ and find
the stabilizer in $\mathrm{Aut}_{\delta}\left(\F_{2g}\right)$ of
a solution $\phi\in\mathrm{Hom}_{w}\left(\F_{2g},\F_{r}\right)$.
Let $\left(\Sigma,f\right)$ be associated with the solution $\phi$.
Choose an arc system $\aa_{0}\in\ap\left(\Sigma,f\right)$ sitting
above some $\left(\sigma_{0},\tau_{0}\right)\in\pmp\left(\Sigma,f\right)$.
Also fix generators $a_{1},b_{1},\ldots,a_{g},b_{g}$ to $\pi_{1}\left(\surface,v_{1}\right)$
with $\left[a_{1},b_{1}\right]\ldots\left[a_{g},b_{g}\right]=\left[\partial\surface\right]$,
and for each generator write down the sequence of discs it traverses
and the color of the arc it crosses at each step (a disc can be recognized
after an action of $\mathrm{MCG}\left(\surface\right)$ by the pieces
in $\partial\surface$ it touches). Then, for any element $\theta\in\pi_{1}\left(\left|\pmp\left(\Sigma,f\right)\right|,\left(\sigma_{0},\tau_{0}\right)\right)$,
lift it to $\left(\left|\ap\left(\Sigma,f\right)\right|,\aa_{0}\right)$
and find the corresponding element $\bb\in\ap\left(\Sigma,f\right)$.
For every generator $a_{i}$ (or $b_{i}$), follow the same sequence
of discs in $\bb$ as it traversed in $\aa_{0}$ (this is well-defined
by Lemma \ref{lem:joint-boundaries-of-discs-in-surface}). This defines
an element of $\pi_{1}\left(\surface,v_{1}\right)$, which is exactly
$\theta\left(a_{i}\right)$, where $\theta$ is identified with the
corresponding element of the stabilizer $\mathrm{Stab}_{\mathrm{Aut_{\delta}\left(\F_{2g}\right)}}\left(\phi\right)$.

As an example, let us return to the word $w=\left[x,y\right]\left[x,z\right]$
and two of the elements of $\ap\left(\Sigma,f\right)$ drawn in Figure
\ref{fig:5 elems of AP}. Let $\aa_{0}$ be the right most element
in this figure, and $\bb$ be the left most one, both of which sit
above the same element of $\pmp\left(\Sigma,f\right)$. These two
elements are redrawn in Figure \ref{fig:explicit-stabilizing-element},
and assume that $\theta\in\mathrm{MCG}\left(\surface\right)$ maps
$\aa_{0}$ to $\bb$. Let the generators $a_{1},b_{1},a_{2},b_{2}$
be the loops at $v_{1}$ around the four handles at the two sides
of the surface, so $a_{1}$ is a clockwise loop around the top-right
handle (drawn in Figure \ref{fig:explicit-stabilizing-element} on
the right), $b_{1}$ is a counter-clockwise loop around the bottom-right
handle, $a_{2}$ is clockwise around the bottom-left and $b_{2}$
is counter-clockwise around the top-left. In $\aa_{0}$, the loop
corresponding to $a_{1}$ traverses the discs marked by $I$, $II$
and $III$ in the following order:
\[
I\overset{p_{1}\mathrm{-arc}}{\to}II\overset{q_{1}\mathrm{-arc}}{\to}I\overset{p_{1}\mathrm{-arc}}{\to}III\overset{q_{1}\mathrm{-arc}}{\to}I\overset{q_{1}\mathrm{-arc}}{\to}II\overset{p_{1}\mathrm{-arc}}{\to}I.
\]
Following the same pattern in $\bb$ results in the dotted loop marked
on the left side of Figure \ref{fig:explicit-stabilizing-element}.
In the generators we chose for $\pi_{1}\left(\surface,v_{1}\right)$,
this new loop is $a_{1}a_{2}a_{1}A_{2}A_{1}$, so $\theta\left(a_{1}\right)=a_{1}a_{2}a_{1}A_{2}A_{1}$.
In the same manner we can figure out how $\theta$ acts on the other
generators:
\begin{equation}
a_{1}\mapsto a_{1}a_{2}a_{1}A_{2}A_{1}\,\,\,\,\, b_{1}\mapsto a_{1}a_{2}A_{1}A_{2}b_{1}a_{1}a_{1}A_{2}A_{1}\,\,\,\,\, a_{2}\mapsto a_{1}a_{2}A_{1}\,\,\,\,\, b_{2}\mapsto b_{2}a_{2}A_{1},\label{eq:generator for [x,y][x,z]}
\end{equation}
which gives an explicit description of $\theta$. Since in this case
$\left|\pmp\left(\Sigma,f\right)\right|$ is a cycle with four edges,
$\theta$ generates the stabilizer. The solution corresponding to
the entire connected component of $\aa_{0}$ and $\bb$ (with respect
to these generators of $\pi_{1}\left(\surface,v_{1}\right)$) is $w=\left[x,y\right]\left[x,z\right]$,
and we deduce 
\begin{equation}
\mathrm{Stab}_{\mathrm{Aut}_{\delta}\left(\F_{4}\right)}\left(\left[x,y\right]\left[x,z\right]\right)=\left\langle \theta|\,\right\rangle .\label{eq:stab of [x,y][x,z]}
\end{equation}

\begin{figure}[h]
\centering{}\includegraphics[bb=90bp 0bp 250bp 80bp,scale=1.7]{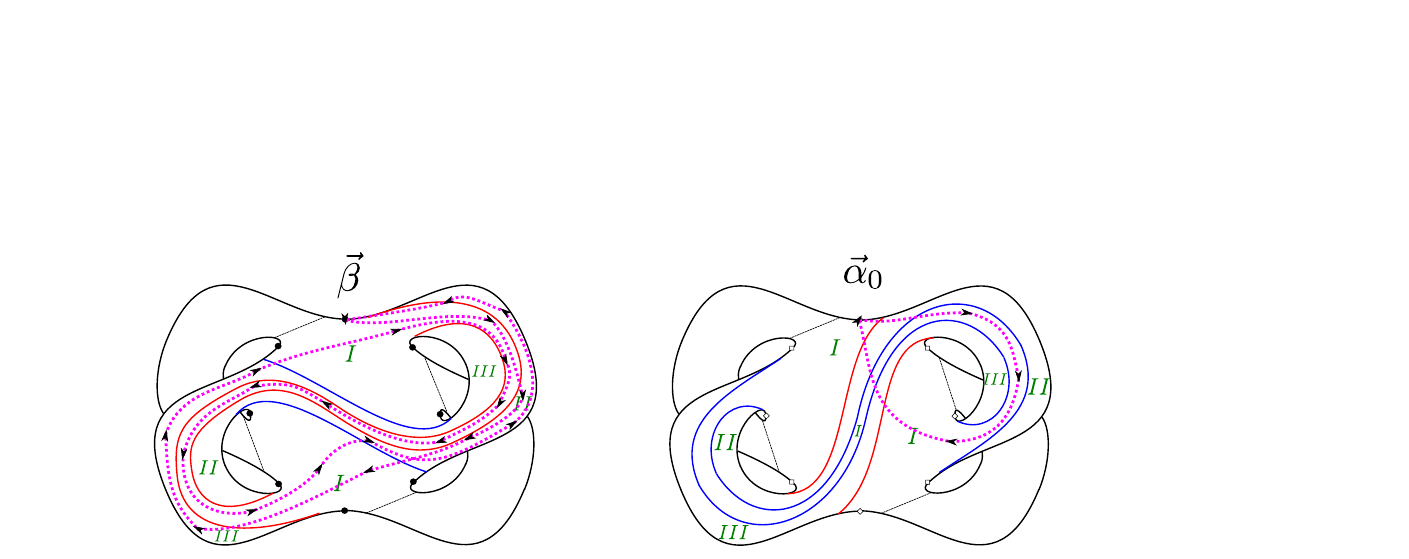}\caption{\label{fig:explicit-stabilizing-element}Two elements in the same
connected component of $\ap\left(\Sigma,f\right)$ for the sole admissible
$\left(\Sigma,f\right)$ for $w=\left[x,y\right]\left[x,z\right]$,
sitting above the same element of $\pmp\left(\Sigma,f\right)$. The
element of $\mathrm{MCG}\left(\surface\right)$ mapping $\aa_{0}$
to $\bb$ maps the generator $a_{1}$ marked in dotted pink line on
the right, to the dotted pink line on the left. }
\end{figure}

We can always find an explicit presentation for the stabilizers. One
method would be to find a generating set for the fundamental group
of the $1$-skeleton of $\left|\pmp\left(\Sigma,f\right)\right|$,
which is free, and then add a relation for every $2$-simplex. We
give one more detailed presentation in Section \ref{sec:Examples}.

\subsubsection*{Solvability of the word problem}

Finally, let us mention another consequence of our constructions:
they show that the word problem for the stabilizers is solvable. To
illustrate this, use the generators we constructed in the previous
paragraph. For every word in these generators, trace the lift in $\left|\ap\left(\Sigma,f\right)\right|$
of the corresponding loop in $\left|\pmp\left(\Sigma,f\right)\right|$.
This word is the identity if and only if the lifted path is also closed,
which can be easily checked algorithmically.

\section{Examples \label{sec:Examples}}

In this section we gather some concrete examples of the solutions
of the commutator equation for a single word and their stabilizers
in $\mathrm{Aut}_{\delta}\left(\F_{2g}\right)$. We always denote
$g=\cl\left(w\right)$.
\begin{itemize}
\item As mentioned in Remark \ref{remark: free solutions}, if $\phi\in\mathrm{Hom}_{w}\left(\F_{2g,}\F_{r}\right)$
is injective, namely, if $\left\{ \phi\left(a_{1}\right),\ldots,\phi\left(b_{g}\right)\right\} $
is a free set in $\F_{r}$, then the stabilizer of $\phi$ is trivial,
and thus its Euler characteristic is $1$. For instance,

\begin{itemize}
\item If $\cl\left(w\right)=1$, every solution is free.
\item The word $w=\left[x,y\right]^{3}$ has commutator length $2$, and
admits $9$ equivalence classes of solutions, each of which is injective.
One of them was already mentioned in Section \ref{sub:cl-of-word}:
$\left[x,y\right]^{3}=\left[xyX,YxyX^{2}\right]\left[Yxy,y^{2}\right]$.
The coefficient of $\frac{1}{n^{3}}$ in $\tr_{\left[x,y\right]^{3}}\left(n\right)$
is, therefore, $9$. Each of the nine complexes $\left|\pmp\left(\cdot\right)\right|$
consists of a single isolated point. The full expression is $\tr_{\left[x,y\right]^{3}}\left(n\right)=\frac{9\left(n^{2}+4\right)}{n^{5}-5n^{3}+4n}$.
\end{itemize}
\item There are also {}``non-injective'' solutions with trivial stabilizer.
For example,\linebreak{}
\foreignlanguage{american}{$w=\left[x,y\right]\left[x^{2}y^{2},z\right]$
has $\cl\left(w\right)=2$ with one solution which is non-injective.
Yet, $\left|\pmp\left(\cdot\right)\right|$ is a path composed of
ten edges, and is contractible. Hence the stabilizer is trivial, and
the coefficient of $\frac{1}{n^{3}}$ is $1$. The full expression
is $\frac{n^{2}-8}{n^{5}-5n^{3}+4n}$.}
\item Along the paper we mentioned the word $w=\left[x,y\right]\left[x,z\right]$.
We computed the only pairs of matchings poset associated with it and
the corresponding simplicial complex (a cycle of length $4$), showed
pieces of its arc poset and also computed its stabilizer in \eqref{eq:stab of [x,y][x,z]}.
The Euler characteristic of this $\left|\pmp\left(\cdot\right)\right|$
is $0$, and thus so is the coefficient of $\frac{1}{n^{3}}$. As
we mentioned in Example \ref{example:[x,y][x,z] leading term 0},
$\tr_{\left[x,y\right]\left[x,z\right]}\left(n\right)=0$ for $n\ge2$
in this case.
\item The leading term vanishes also for $w=\left[x,y\right]\left[x,z\right]\left[x,t\right]$.
Here $\cl\left(w\right)=3$ and there is a single equivalence class
of solutions. The pairs of matchings poset $\pmp$ is of size $30$:
six of rank 0, eighteen of rank 1 and six of rank $2$. Hence $\left|\pmp\right|$
is $2$-dimensional. It consists of $30$ vertices, $102$ edges and
$72$ $2$-simplices, and thus $\chi\left(\left|\pmp\right|\right)=0$
and the coefficient of $\frac{1}{n^{5}}$ is $0$. In fact, here too,
$\tr_{\left[x,y\right]\left[x,z\right]\left[x,t\right]}\left(n\right)\equiv0$
(for $n\ge3$). A closer look at $\left|\pmp\right|$ reveals it is
homeomorphic to the cross product of $S^{1}$ with a Theta figure,
so its fundamental group is isomorphic to $\mathbb{Z}\times\F_{2}$.
A computation conducted as explained in Section \ref{sec:More-Consequences}
reveals that 
\[
\mathrm{Stab}_{\mathrm{Aut}_{\delta}\left(\F_{6}\right)}\left(\left[x,y\right]\left[x,z\right]\left[x,t\right]\right)=\left\langle \theta_{1},\theta_{2},\theta_{3}\,|\,\left[\theta_{1},\theta_{2}\right],\left[\theta_{1},\theta_{3}\right]\right\rangle ,
\]
where the $\theta_{i}$'s are given by:\\
\begin{tabular}{|c||c|c|c|}
\cline{2-4} 
\multicolumn{1}{c|}{} & $\theta_{1}$ & $\theta_{2}$ & $\theta_{3}$\tabularnewline
\hline 
$a_{1}\mapsto$ & $a_{1}^{a_{1}a_{2}a_{3}}$ & $a_{1}^{a_{1}a_{2}}$ & $a_{1}^{a_{1}a_{3}}$\tabularnewline
\hline 
$b_{1}\mapsto$ & $\left(a_{2}a_{3}a_{1}b_{1}A_{1}A_{1}A_{1}\right)^{a_{1}a_{2}a_{3}}$ & $\left(a_{2}a_{1}b_{1}A_{1}A_{1}\right)^{a_{1}a_{2}}$ & $\left(a_{3}a_{1}b_{1}A_{1}A_{1}\right)^{a_{1}a_{3}}$\tabularnewline
\hline 
$a_{2}\mapsto$ & $a_{2}^{a_{1}a_{2}a_{3}}$ & $a_{2}^{a_{1}a_{2}}$ & $a_{2}^{A_{1}A_{3}a_{1}a_{3}}$\tabularnewline
\hline 
$b_{2}\mapsto$ & $\left(a_{3}a_{1}a_{2}b_{2}A_{2}A_{2}A_{2}\right)^{a_{1}a_{2}a_{3}}$ & $\left(a_{1}a_{2}b_{2}A_{2}A_{2}\right)^{a_{1}a_{2}}$ & $b_{2}^{A_{1}A_{3}a_{1}a_{3}}$\tabularnewline
\hline 
$a_{3}\mapsto$ & $a_{3}^{a_{1}a_{2}a_{3}}$ & $a_{3}$ & $a_{3}^{a_{1}a_{3}}$\tabularnewline
\hline 
$b_{3}\mapsto$ & $\left(a_{1}a_{2}a_{3}b_{3}A_{3}A_{3}A_{3}\right)^{a_{1}a_{2}a_{3}}$ & $b_{3}$ & $\left(a_{1}a_{3}b_{3}A_{3}A_{3}\right)^{a_{1}a_{3}}$\tabularnewline
\hline 
\end{tabular}\\
(by $u^{v}$ we mean $v^{-1}uv$, so $a_{1}^{a_{1}a_{2}a_{3}}=A_{3}A_{2}A_{1}a_{1}a_{1}a_{2}a_{3}=A_{3}A_{2}a_{1}a_{2}a_{3}$).
\item If $w=\left[x,y\right]^{2}$, then $\cl\left(w\right)=2$ with exactly
one solution. The sole $\left|\pmp\right|$ is $1$-dimensional with
$12$ vertices and $ $$16$ edges. Here $\chi\left(\left|\pmp\right|\right)=-4$
is the leading coefficient. The stabilizer is isomorphic to $\F_{5}$.
One possible generator (a primitive element of this $\F_{5}$) is
given in \eqref{eq:generator for [x,y][x,z]}.
\item If $w=w_{1}w_{2}$ is a product of two words with disjoint letters
(or more generally of two words from complementing free factors of
$\F_{r}$), then $\trw\left(n\right)=\tr_{w_{1}}\left(n\right)\cdot\tr_{w_{2}}\left(n\right)\cdot n^{-1}$,
the stabilizer of a solution is the direct product of the stabilizer
of the corresponding solution of $w_{1}$ and that of $w_{2}$, and
the Euler characteristics of the stabilizers are multiplicative as
well. 
\end{itemize}

\section{Some Open Problems\label{sec:Some-Open-Problems}}

We mention some open problems that naturally arise from the discussion
in this paper.
\selectlanguage{american}%
\begin{enumerate}
\item In this work we analyzed the expected trace of a random element of
$\U\left(n\right)$, which corresponds to a natural series of (irreducible)
characters $\xi_{n}$ of ${\cal U}\left(n\right)$. As explained in
Section \ref{sub:Word-measures}, the more general Theorem \ref{thm:main - general}
also gives information about other series of irreducible characters
of ${\cal U}\left(n\right)$. A similar question was studied in \cite{PP15}
regarding the series of irreducible characters of $S_{n}$ which count
the number of fixed points in a permutation (minus one). It should
be very interesting to realize what $\mathrm{Aut}\left(\F_{r}\right)$-invariants
of words play a role in similar questions surrounding:

\begin{itemize}
\item The expected trace of elements in the orthogonal group $O\left(n\right)$
or the symplectic group $Sp$$\left(n\right)$: as the results of
Collins and \'{S}niady \cite{CS} extend to these groups, there should
be rational expressions in $n$ as in Theorem \ref{thm:trw-rational}.
What is the leading term of each expression?
\item There should also be rational expressions for other series of characters
of the groups $S_{n}$, $O\left(n\right)$ and $\mathrm{Sp}\left(n\right)$.
What is the leading term for each series?
\selectlanguage{english}%
\item In particular, what are the $\mathrm{Aut}\left(\F_{r}\right)$-invariants
of words controlling (the asymptotics of) balanced characters of ${\cal U}\left(n\right)$
(recall that balanced characters are those invariant under rotations
- see \foreignlanguage{american}{Section \ref{sub:Word-measures}).}
\selectlanguage{american}%
\item What about completely different families of groups? For example, consider
the action of $\mathrm{PSL}_{2}\left(q\right)$ on the projective
line $\mathbb{P}^{1}\left(q\right)$. What it the expected number
of fixed points in this action when $g\in\mathrm{PSL}_{2}\left(q\right)$
is sampled by some $w$-measure and $q$ varies?
\item Is it possible to find the algebraic meaning of the other ($\mathrm{Aut}(\F_{r})$-invariant)
coefficients of the rational function $\trw\left(n\right)$?
\end{itemize}
\item In some cases, the coefficient of $\trwl\left(n\right)$ we analyze
in Theorem \ref{thm:main - general} vanishes. This is the case, for
example, for $w=\left[x,y\right]\left[x,z\right]$ and also for $w=\left[x,y\right]\left[x,z\right]\left[x,t\right]$.
What is the leading coefficient in these cases? Interestingly, among
the dozens of concrete examples we computed, there were a handful
where the coefficient from Theorem \ref{thm:main} vanished. In all
these cases the entire expression turned out to be zero, namely, $\trwl\left(n\right)=0$
for any large enough $n$.
\end{enumerate}
\selectlanguage{english}%

\section*{Acknowledgments}

We would also like to thank Danny Calegari, Alexei Entin, Mark Feighn,
Alex Gamburd, Peter Sarnak, Zlil Sela, Avi Wigderson and Ofer Zeitouni
for valuable discussions about this work.

\section*{Appendices}

\begin{appendices}

\section{Appendix: Posets and Complexes}

In this appendix we include some auxiliary general results regarding
posets and complexes, which are directly used in the proofs along
the paper. These results are not new.

\subsection{Homotopy of poset morphisms\label{sub:Homotopy-of-poset}}

In our proof of contractability of the connected components of $\left|\ap\left(\Sigma,f\right)\right|$
in Section \ref{sub:Proof-of-contractability}, we use a series of
deformation retracts of simplicial complexes associated with posets
(see Definition \ref{def:SC-of-poset}). Here, we establish a criterion
which guarantees that a retract of posets $f:P_{2}\to P_{1}$, where
$P_{1}$ is a subposet of $P_{2}$, is a deformation retract of the
associated simplicial complexes. This is the criterion we use in the
proof of contractability.

The main ingredient in establishing this criterion deals with direct
products of posets. The direct product $P\times Q$ of the posets
$\left(P,\le_{P}\right)$ and $\left(Q,\le_{Q}\right)$ is defined
on the set $P\times Q$ with partial order $\left(p_{1},q_{1}\right)\le_{P\times Q}\left(p_{2},q_{2}\right)$
if and only if $p_{1}\le_{P}p_{2}$ and $q_{1}\le_{Q}q_{2}$. The
following lemma is well known: see, for instance, \cite[Theorem 3.2]{walker1988canonical}.
\begin{lem}
\label{lem:direct products of posets}Let $P$ and $Q$ be posets.
The function $\gamma\colon\left|P\times Q\right|\to\left|P\right|\times\left|Q\right|$
defined by 
\[
\sum\lambda_{i}\left(p_{i},q_{i}\right)\,\mapsto\,\left(\sum\lambda_{i}p_{i},\sum\lambda_{i}q_{i}\right)
\]
is an homeomorphism.
\end{lem}
The following corollary appears in \cite[Section 1.3]{QUILLEN}. Recall
that a map $f$ between posets is called a poset-morphism if it is
order preserving. If $f\colon P\to Q$ is a poset morphism, we let
$\left|f\right|$ denote the induced map 
\[
\left|f\right|\colon\left|P\right|\to\left|Q\right|
\]
defined naturally as $\left|f\right|\left(\sum\lambda_{i}p_{i}\right)=\sum\lambda_{i}f\left(p_{i}\right)$.
\begin{cor}
\label{cor:homotopy of simplicial maps}Let $P$ and $Q$ be posets,
and $f,g\colon P\to Q$ poset morphisms. If $f\left(p\right)\le g\left(p\right)$
for every $p\in P$, then $\left|f\right|$ and $\left|g\right|$
are homotopic.\end{cor}
\begin{proof}
Let $\left\{ 0\le1\right\} $ denote the poset with two comparable
elements $0$ and 1. Define a map $\left(f,g\right)\colon P\times\left\{ 0\le1\right\} \to Q$
by $\left(p,0\right)\mapsto f\left(p\right)$ and $\left(p,1\right)\mapsto g\left(p\right)$.
This is clearly a poset-morphism by the assumptions, so it induces
a continuous map 
\[
\left|\left(f,g\right)\right|\colon\left|P\times\left\{ 0\le1\right\} \right|\to\left|Q\right|.
\]
By Lemma \ref{lem:direct products of posets}, there is an homeomorphism
\[
\left|P\times\left\{ 0\le1\right\} \right|\overset{\cong}{\to}\left|P\right|\times\left|\left\{ 0\le1\right\} \right|=\left|P\right|\times\left[0,1\right],
\]
so we get that $\left|\left(f,g\right)\right|$ is a continuous map
$\left|P\right|\times\left[0,1\right]\to\left|Q\right|$. Because
$\left|\left(f,g\right)\right|\Big|_{\left|P\times\left\{ 0\right\} \right|}\equiv\left|f\right|$
and $\left|\left(f,g\right)\right|\Big|_{\left|P\times\left\{ 1\right\} \right|}\equiv\left|g\right|$,
the map $\left|\left(f,g\right)\right|$ is the sought after homotopy.\end{proof}
\begin{rem}
\label{remark:when f=00003Dg}Note that the homotopy does not move
the points where $f$ and $g$ agree. Namely, if $P_{0}\subseteq P$
is the subposet where $f\left(p\right)=g\left(p\right)$, then $\left|\left(f,g\right)\right|\left(x,t\right)=f\left(x\right)=g\left(x\right)$
for every $x\in\left|P_{0}\right|$ and $t\in\left[0,1\right]$.\end{rem}
\begin{cor}
\label{cor:deformation-retract}Let $P$ be a subposet of the poset
$Q$. Assume that $f\colon Q\to P$ satisfies the following:
\begin{itemize}
\item it is a poset morphism, 
\item it is a retract (i.e., $f\Big|_{P}\equiv\id$), and 
\item $f\left(q\right)\le q$ for every $q\in Q$, or $q\le f\left(q\right)$
for every $q\in Q$.
\end{itemize}

Then $\left|f\right|$ is a (strong) deformation retract.

\end{cor}
\noindent By a strong deformation retract we mean that there is a
homotopy of $\left|f\right|$ with the identity on $\left|Q\right|$
which fixes the points in $\left|P\right|$ throughout the homotopy.
\begin{proof}
Simply note that the map $f\colon Q\to Q$ and the identity $\id\colon Q\to Q$
satisfy the conditions in Corollary \ref{cor:homotopy of simplicial maps}
hence $\left|f\right|$ is homotopic to $ $the identity. The fact
that the homotopy fixes $\left|P\right|$ pointwise follows from
Remark \ref{remark:when f=00003Dg}.
\end{proof}

\subsection{Regular $G$-complexes\label{sub:Regular G-complexes}}

When we say that a discrete group $G$ acts on a simplicial complex
$K$, we mean, in particular, that the action is simplicial. Namely,
we mean that $G$ acts on the set of vertices, and the induced map
on the subsets of vertices maps every simplex to a simplex. There
are two natural ways to construct a quotient space for this action.
One way is to construct a simplicial complex as follows: the set of
vertices consists of the orbits $\nicefrac{V\left(K\right)}{G}$ of
vertices and whenever $\left(v_{0},\ldots,v_{r}\right)$ is an $r$-simplex
of $K$, then $\left(\left[v_{0}\right],\ldots,\left[v_{r}\right]\right)$
is an $r$-simplex of the quotient. We denote this quotient by $\left|\nicefrac{K}{G}\right|$.
The second way is to consider the geometric realization of $K$, which
$G$ clearly acts on, and take the usual quotient of an action on
a topological space. We denote this quotient by $\nicefrac{\left|K\right|}{G}$.

The problem is that these two quotient spaces do not coincide in general.
First, if the action mixes different vertices of the same simplex,
the topological quotient results in pieces which are fractions of
simplices. This is the case, for example, in the case that $\nicefrac{\mathbb{Z}}{2\mathbb{Z}}$
acts on a graph with a single edge by flipping the edge. Secondly,
as illustrated by the action of $\nicefrac{\mathbb{Z}}{2\mathbb{Z}}$
on the boundary of a square by a $180^{\circ}$-rotation mentioned
in Remark \ref{remark:need-for-regular-action}, the orbits of the
simplices in the geometric realization are not always determined by
the orbits of the vertices. 

These, however, can be easily remedied by adding the following assumptions:
\begin{defn}
\label{def:regular-action}\cite[Definition III.1.2]{bredon1972introduction}
A simplicial $G$-action on the simplicial complex $K$ is called
\textbf{regular}, if 
\begin{enumerate}
\item \label{enu:regular-action-1}If $v\in V\left(K\right)$ and $g.v$
belong to same simplex for some $g\in G$, then $g.v=v$.
\item \label{enu:regualr-action-2}Whenever $g_{0},\ldots,g_{r}$ are elements
of $G$ and $\left(v_{0},\ldots,v_{r}\right)$ and $\left(g_{0}.v_{0},\ldots,g_{r}.v_{r}\right)$
are $r$-simplices of $K$, there is some $g\in G$ with $\left(g_{0}.v_{0},\ldots,g_{r}.v_{r}\right)=\left(g.v_{0},\ldots,g.v_{r}\right)$.
\end{enumerate}
\end{defn}
In other words, these additional conditions exactly guarantee that
$\left(1\right)$ the action does not {}``break'' simplices by identifying
different points of the same simplex, and that $\left(2\right)$ the
orbits of the simplices in the geometric realization can be deduced
from those of the vertices. 
\begin{lem}
\label{lem:regular-simplicial-actions}\cite[Page 117]{bredon1972introduction}
If the action of $G$ on the simplicial complex $K$ is regular then
\[
\left|\nicefrac{K}{G}\right|\cong\nicefrac{\left|K\right|}{G}.
\]

\end{lem}
In the current paper, we are interested in $G$-actions on graded
posets and on their corresponding simplicial complexes. Lemma \ref{lem:regular-simplicial-actions}
translates to the following (see Definition \ref{def:SC-of-poset}
and the footnote on Page \pageref{fn:graded-poset} for some of the
terminology):
\begin{cor}
\label{cor:regular-poset-actions}Let $G$ act on a locally-finite
graded poset $\left(P,\le\right)$ by a graded-poset action, and assume
that whenever $x_{0}<\ldots<x_{r}$ and $g_{0}.x_{0}<\ldots<g_{r}.x_{r}$
for some $g_{0},\ldots,g_{r}\in G$ and $x_{0},\ldots,x_{r}\in P$,
there is a $g\in G$ with $g.x_{i}=g_{i}.x_{i}$ for every $i$. Then
\[
\left|\nicefrac{P}{G}\right|\cong\nicefrac{\left|P\right|}{G}.
\]
\end{cor}
\begin{proof}
We only need to check that the action is regular. Item \ref{enu:regualr-action-2}
of Definition \ref{def:regular-action} holds by our extra assumption,
while item \ref{enu:regular-action-1} follows from the fact that
the action preserves rank, thus guaranteeing that $x$ and $g.x$
cannot belong to same simplex of $\left|P\right|$ unless $x=g.x$. 
\end{proof}
\end{appendices}

\section*{Glossary\label{sec:Glossary}}

\begin{center}
\begin{tabular}{|>{\centering}m{0.17\columnwidth}|>{\centering}m{0.45\columnwidth}|>{\centering}m{0.18\columnwidth}|>{\centering}m{0.2\columnwidth}|}
\hline 
 &  & \textbf{Reference} & \textbf{Remarks}\tabularnewline[\doublerulesep]
\hline 
$\F_{r}$ & the free group on $r$ generators &  & \tabularnewline[\doublerulesep]
\hline 
$x_{1},\ldots,x_{r}$ & a set of generators for $\F_{r}$ &  & sometimes $x,y,z,t$ used instead\tabularnewline[\doublerulesep]
\hline 
$X_{1},\ldots,X_{r}$ & $X_{i}=x_{i}^{-1}$ marks the inverse &  & likewise, $X,Y,Z,T$\tabularnewline[\doublerulesep]
\hline 
$\U\left(n\right)$ & the group of $n\times n$ unitary matrices &  & \tabularnewline[\doublerulesep]
\hline 
$\mu_{n}$ & the Haar measure on $\U\left(n\right)$ &  & \tabularnewline[\doublerulesep]
\hline 
$\trw\left(n\right)$ & expected trace of $A\in\U\left(n\right)$ sampled according to the
$w$-measure  & \eqref{eq:firstorder} & \tabularnewline[\doublerulesep]
\hline 
$\cl\left(w\right)$ & the commutator length of $w$ & Page \pageref{cl} & \tabularnewline[\doublerulesep]
\hline 
$a_{1},b_{1},\ldots,a_{g},b_{g}$ & a set of generators for $\F_{2g}$ &  & $A_{1},B_{1},\ldots,A_{g},B_{g}$ mark inverses\tabularnewline[\doublerulesep]
\hline 
$\delta_{g}$ & $\left[a_{1},b_{1}\right]\ldots\left[a_{g},b_{g}\right]$ &  & \tabularnewline[\doublerulesep]
\hline 
$\mathrm{Hom}_{w}\left(\F_{2g},\F_{r}\right)$ & $\left\{ \phi\in\mathrm{Hom}\left(\F_{2g},\F_{r}\right)\,\middle|\,\phi\left(\delta_{g}\right)=w\right\} $ &  & \tabularnewline[\doublerulesep]
\hline 
$\mathrm{Aut}_{\delta}\left(\F_{2g}\right)$ & $\left\{ \rho\in\mathrm{Aut}\left(\F_{2g}\right)\,\middle|\,\rho\left(\delta_{g}\right)=\delta_{g}\right\} $ &  & \tabularnewline[\doublerulesep]
\hline 
$\chi$ & Euler characteristic of a space or a group & Page \ref{The-Euler-characteristic} & \tabularnewline[\doublerulesep]
\hline 
$\surface$ & Orientable surface of genus $g$ and one boundary component. & Section \ref{sub:Expected-product-of} & \tabularnewline[\doublerulesep]
\hline 
$\trwl\left(n\right)$ & $\mathbb{E}\left[\prod_{i=1}^{\ell}\mathrm{tr}\left(w_{i}\left(U_{1}^{\left(n\right)},\ldots,U_{r}^{\left(n\right)}\right)\right)\right]$ &  & \tabularnewline[\doublerulesep]
\hline 
$\left(\wedger,\mbox{o}\right)$ & a wedge of $r$ circles, fundamental group identified with $\F_{r}$,
pointed at the wedge point $o$ & Section \ref{sub:Expected-product-of}, Figure \ref{fig:S^1(w) and marked wedge} & sometimes additional marked points\tabularnewline[\doublerulesep]
\hline 
$\left(\Sigma,f\right)$ admissible for $\wl$ & $\Sigma$ a compact oriented surface with $\ell$ boundary components
and $f\colon\Sigma\to\wedger$ maps these components to $\wl$  & Definition \ref{def: admissible maps} & \tabularnewline[\doublerulesep]
\hline 
$\left(\Sigma,f\right)$ incompressible & no essential simple closed curve mapped to nullhomotopic loop & Definition \ref{def:incompressible} & \tabularnewline[\doublerulesep]
\hline 
$\tilde{f}$ & homotopy class of $f\colon\Sigma\to\wedger$, relative $\partial\Sigma$ & Theorem \ref{thm:main - general} & \tabularnewline[\doublerulesep]
\hline 
\end{tabular}\\
\par\end{center}

\begin{center}
\begin{tabular}{|>{\centering}m{0.17\columnwidth}|>{\centering}m{0.45\columnwidth}|>{\centering}m{0.18\columnwidth}|>{\centering}m{0.2\columnwidth}|}
\hline 
 &  & \textbf{Reference} & \textbf{Remarks}\tabularnewline[\doublerulesep]
\hline 
$v_{1},\ldots,v_{\ell}$  & {}``basepoints'' of $\Sigma$, one at every boundary component &  & \tabularnewline[\doublerulesep]
\hline 
$\partial_{1},\ldots,\partial_{\ell}$ & identifications of boundary components of $\Sigma$ with $S^{1}$ &  & \tabularnewline[\doublerulesep]
\hline 
$f_{w}$ & a map $\left(S^{1},1\right)\to\left(\wedger,o\right)$ with image
representing $w$ & Page \pageref{f_w} and more detailed in Section \ref{sub:cl-of-word} & \tabularnewline[\doublerulesep]
\hline 
$\mathrm{MCG}\left(\Sigma\right)$ & mapping class group of $\Sigma$, consisting of mapping classes which
fix $\partial\Sigma$ pointwise &  & \tabularnewline[\doublerulesep]
\hline 
$\ch\left(\wl\right)$ &  & Definition \ref{def:chi(words)} & \tabularnewline[\doublerulesep]
\hline 
$\left(\Sigma,f\right)\sim\left(\Sigma',f'\right)$ &  & Page \pageref{(Sigma,f) sim (Sigma',f')} & \tabularnewline[\doublerulesep]
\hline 
$\left[\left(\Sigma,f\right)\right]$ & equivalence class of $\left(\Sigma,f\right)$ & Definition \pageref{def:sol(words)} & \tabularnewline[\doublerulesep]
\hline 
$\sol\left(\wl\right)$ & set of equivalence classes of admissible maps of maximal $\chi$ & Definition \pageref{def:sol(words)} & \tabularnewline[\doublerulesep]
\hline 
balanced set of words & words such that the total number of $x_{i}^{+1}$ is the same as total
number of $x_{i}^{-1}$ &  & \tabularnewline[\doublerulesep]
\hline 
$S^{1}\left(w\right)$ & a marked circle which spells out $w$ & Section \ref{sec:surface-from-matchings} and Figure \ref{fig:S^1(w) and marked wedge} & \tabularnewline[\doublerulesep]
\hline 
$o$, $p_{i}$, $z_{i},$ $q_{i}$ & marked points on $\wedger$ & Sections \ref{sub:cl-of-word} and \ref{sec:surface-from-matchings} & \tabularnewline[\doublerulesep]
\hline 
$p_{i}^{\pm}$, $q_{i}^{\pm}$ & marked points of $S^{1}\left(w\right)$, $\partial\Sigma$ & Sections \ref{sub:cl-of-word} and \ref{sec:surface-from-matchings} & \tabularnewline[\doublerulesep]
\hline 
$\mathrm{Wg}$ & the Weingarten function & Definition \ref{def:weingarten}  & \tabularnewline[\doublerulesep]
\hline 
$\left\Vert \sigma\right\Vert $ & the norm of the permutation $\sigma$ & Section \ref{sub:Weingarten-function-and-Collins-Sniady} & \tabularnewline[\doublerulesep]
\hline 
$\moeb\left(\sigma\right)$ & the Möbius function of $\sigma$ & Proposition \ref{prop:mobius function} & \tabularnewline[\doublerulesep]
\hline 
$L$, $L_{i}$ & assuming $\wl$ balanced , $2L=\sum\left|w_{i}\right|$ and $L_{i}$
is the number of appearances of $x_{i}^{+1}$ & Section \ref{sec:A-Rational-Expression} & \tabularnewline[\doublerulesep]
\hline 
$E^{\pm},E_{i}^{\pm}$ & subsets of the letter of $\wl$ & Section \ref{sec:surface-from-matchings} & \tabularnewline[\doublerulesep]
\hline 
${\scriptstyle \match\left(\wl\right)}$ & the set of bijections $E^{+}\overset{\sim}{\to}E^{-}$ which map $E_{i}^{+}$
to $E_{i}^{-}$ & Definition \ref{def:match(w) and B(sigma,tau)} & \tabularnewline[\doublerulesep]
\hline 
$\Sigma_{\left(\sigma,\tau\right)}$ & the CW-complex associated with $\left(\sigma,\tau\right)\in\match\left(\wl\right)^{2}$ & Definition \ref{def:surface-from-perms} & \tabularnewline[\doublerulesep]
\hline 
matching-edges &  & Definition \ref{def:surface-from-perms} & \tabularnewline[\doublerulesep]
\hline 
$\chi\left(\sigma,\tau\right)$ & the Euler characteristic of $\Sigma_{\left(\sigma,\tau\right)}$ & Definition \ref{def:genus(sigma,tau)}  & \tabularnewline[\doublerulesep]
\hline 
type-$o$ and type-$z_{i}$ discs & types of discs in $\Sigma_{\left(\sigma,\tau\right)}$ as well as
in $\Sigma_{\aa}$ & Claims \ref{claim:properties-of-perms-surface}, Section \ref{sub:Arc-systems} & \tabularnewline[\doublerulesep]
\hline 
$f_{\left(\st\right)},f_{\aa}$ & $f_{\left(\st\right)}\colon\Sigma_{\left(\st\right)}\to\wedger$ and
$f_{\aa}\colon\Sigma_{\aa}\to\wedger$ are homotopy classes of maps & Definition \ref{def:f_(sigma,tau)} and Section \ref{sub:Arc-systems} & \tabularnewline[\doublerulesep]
\hline 
$\pmp\left(\Sigma,f\right)$, $\left|\pmp\left(\Sigma,f\right)\right|$ & the pair of matchings poset and its associated simplicial complex & Definitions \ref{def:perm-poset} and \ref{def:SC-of-poset} & \tabularnewline[\doublerulesep]
\hline 
$\sigma_{\aa},\tau_{\aa}$ & the pair of matchings induced by the arc system $\aa$ & Section \ref{sub:Arc-systems} & \tabularnewline[\doublerulesep]
\hline 
\end{tabular}\\
\par\end{center}

\begin{center}
\begin{tabular}{|>{\centering}m{0.17\columnwidth}|>{\centering}m{0.45\columnwidth}|>{\centering}m{0.18\columnwidth}|>{\centering}m{0.2\columnwidth}|}
\hline 
 &  & \textbf{Reference} & \textbf{Remarks}\tabularnewline[\doublerulesep]
\hline 
$\Sigma_{\aa}$ & the CW-complex structure induced on $\Sigma$ by the arc system $\aa$ & Section \ref{sub:Arc-systems} & \tabularnewline[\doublerulesep]
\hline 
$\ap\left(\Sigma,f\right)$, $\left|\ap\left(\Sigma,f\right)\right|$ & the arc poset and its associated simplicial complex & Definitions \ref{def:arc-poset}, \ref{def:SC-of-poset} & \tabularnewline[\doublerulesep]
\hline 
$\preceq$ & partial orders defined on $S_{L}$, $S_{L}^{\,2}$, $\match\left(\wl\right)^{2}$,
$\pmp\left(\Sigma,f\right)$, $\ap\left(\Sigma,f\right)$ & Sections \ref{sub:Weingarten-function-and-Collins-Sniady} and \ref{sec:The-pairs-of-matchings-Poset}
and Definitions \ref{def:perm-poset} and \ref{def:arc-poset} & \tabularnewline[\doublerulesep]
\hline 
graded poset &  & \multirow{2}{0.18\columnwidth}{Footnote on Page \pageref{fn:graded-poset}} & \tabularnewline[\doublerulesep]
\cline{1-2} \cline{4-4} 
$x$ covers $y$ &  &  & for $x,y$ in a poset\tabularnewline[\doublerulesep]
\hline 
\end{tabular}
\par\end{center}

\bibliographystyle{alpha}
\bibliography{database}

\noindent Michael Magee,\\
Department of Mathematics,\\
Yale University,

\noindent PO Box 208283, New Haven, CT 06520 USA \\
michael.magee@yale.edu\\

\noindent Doron Puder, \\
School of Mathematics,\\
Institute for Advanced Study,\\
Einstein Drive, Princeton, NJ 08540 USA\\
doronpuder@gmail.com
\end{document}